\documentclass{amsart}
\usepackage{amsmath}
\usepackage{amsthm}
\usepackage{amssymb}
\usepackage{color}
\usepackage{graphicx}
\usepackage{apptools}
\usepackage{chngcntr}

\newtheorem{thm}{Theorem}\newtheorem{lemma}[thm]{Lemma}

\newtheorem{defi}[thm]{Definition}
\newtheorem{prop}[thm]{Proposition}
\newtheorem{rk}[thm]{Remark}

\AtAppendix{\counterwithin{thm}{section}}
\AtAppendix{\counterwithin{equation}{section}}

\newcommand{\rr}{{\mathbb{R}}}
\newcommand{\ttheta}{{\tilde{\theta}}}

\newcommand{\nn}{{\mathbb{N}}}
\newcommand{\sS}{{\mathbb{S}}}
\newcommand{\E}{{\mathbb{E}}}
\newcommand{\uU}{{\mathbb{U}}}

\newcommand{\PP}{{\mathbb{P}}}
\newcommand{\QQ}{{\mathbb{Q}}}
\newcommand{\e}{\varepsilon}
\newcommand{\vip}{\vskip.18cm}
\newcommand{\indiq}{\hbox{\rm 1}{\hskip -2.8 pt}\hbox{\rm I}}
\newcommand{\dd}{{\rm d}}
\newcommand{\ddiv}{{\rm div}}
\newcommand{\cX}{{\mathcal{X}}}
\newcommand{\cG}{{\mathcal{G}}}
\newcommand{\cY}{{\mathcal{Y}}}
\newcommand{\cZ}{{\mathcal{Z}}}

\newcommand{\bK}{{\mathbf K}}
\newcommand{\Et}{{E_\triangle}}
\newcommand{\cXt}{{\mathcal{X}_\triangle}}
\newcommand{\cUt}{{\mathcal{U}_\triangle}}
\newcommand{\cU}{{\mathcal{U}}}
\newcommand{\cL}{{\mathcal{L}}}
\newcommand{\cK}{{\mathcal{K}}}
\newcommand{\cM}{{\mathcal{M}}}
\newcommand{\cN}{{\mathcal{N}}}
\newcommand{\cF}{{\mathcal{F}}}
\newcommand{\cP}{{\mathcal{P}}}
\newcommand{\cA}{{\mathcal{A}}}

\newcommand{\bm}{{\mathbf m}}
\newcommand{\cC}{{\mathcal{C}}}
\newcommand{\cD}{{\mathcal{D}}}
\newcommand{\cE}{{\mathcal{E}}}

\newcommand{\xX}{{\mathbb{X}}}
\newcommand{\yY}{{\mathbb{Y}}}

\newcommand{\xXs}{{\mathbb{X}^*}}
\newcommand{\Xs}{X^*}
\newcommand{\ig}{[\![}
\newcommand{\id}{]\!]}
\newcommand{\intot}{\int_0^t}

\newcommand{\bla}{\color{black}}
\newcommand{\blue}{\color{black}}

\begin{document}

\title{Collisions of the supercritical Keller-Segel particle system}

\author{Nicolas Fournier and Yoan Tardy}
\address{Sorbonne Universit\'e, LPSM-UMR 8001, Case courrier 158,75252 Paris Cedex 05, France.}
\email{nicolas.fournier@sorbonne-universite.fr, yoan.tardy@sorbonne-universite.fr}
\subjclass[2010]{60H10, 60K35}

\keywords{Keller-Segel equation, Stochastic particle systems, Bessel processes, Collisions}
\thanks{The fruitful comments of the two referees helped us to
substantially improve the presentation of the paper.}

\begin{abstract}
We study a particle system naturally associated to the $2$-dimensional Keller-Segel equation. 
It consists
of $N$ Brownian particles in the plane, interacting through a binary 
attraction in $\theta/(Nr)$,
where $r$ stands for the distance between two particles.
When the intensity $\theta$ of this attraction is greater than $2$, this particle 
system explodes in finite time.
We assume that $N>3\theta$ and study in details what happens near explosion. 
There are two slightly different scenarios, depending on
the values of $N$ and $\theta$, here is one:
at explosion, a cluster consisting of precisely $k_0$ particles emerges, for some deterministic 
$k_0\geq 7$ depending on $N$ and
$\theta$.
Just before explosion, there are infinitely many $(k_0-1)$-ary collisions.
There are also infinitely many $(k_0-2)$-ary collisions
before each $(k_0-1)$-ary collision. And there are infinitely many 
binary collisions before each $(k_0-2)$-ary collision. Finally, collisions
of subsets of $3,\dots,k_0-3$ particles never occur.
The other scenario is similar except that there are no $(k_0-2)$-ary collisions.
\end{abstract}

\maketitle

\section{Introduction and main results}
\setcounter{equation}{0}
\subsection{Informal definition of the model}
We consider some scalar parameter $\theta>0$ and a number $N \geq 2$ of particles with positions
$X_t=(X^1_t,\dots,X^N_t) \in (\rr^2)^N$ at time $t\geq 0$. Informally, we assume that 
the dynamics of these
particles are given by the system of S.D.E.s
\begin{align}\label{EDS}
\dd X^{i}_{t} = \dd B^{i}_{t} - \frac{\theta }{N} \sum_{j\neq i}
\frac{X^{i}_{t}-X^{j}_{t}}{\|X^{i}_{t}-X^{j}_{t} \|^{2}}\dd t, \qquad i\in \ig 1,N\id,
\end{align}
where the $2$-dimensional Brownian motions $((B^i_t)_{t\geq 0})_{i \in \ig 1,N\id}$ are independent.
In other words, we have $N$ Brownian particles in the plane interacting through 
an attraction in $1/r$,
which is Coulombian in dimension $2$. Actually,
this S.D.E. does not clearly make sense, due to the singularity of the drift, and we will use,
as suggested by Cattiaux-P\'ed\`eches \cite{cp}, the theory of
Dirichlet spaces, see Fukushima-Oshima-Takeda \cite{f}.

\subsection{Brief motivation and informal presentation of the main results}
This particle system is very natural from a physical point of view, because, as we will see, 
there is a tight competition between the Brownian excitation and the Coulombian attraction.
It can also be seen as an approximation of the famous Keller-Segel equation \cite{ks}, see also
Patlak \cite{p}. This nonlinear P.D.E. has been introduced to model the collective motion of cells, 
which are attracted by a chemical substance that they emit. It is well-known that a 
phase transition occurs:
if the intensity of the attraction is small, then there exist global solutions, 
while if the attraction
is large, the solution explodes in finite time. 

\vip

 We will show that this phase 
transition already occurs at the level
of the particle system \eqref{EDS}: there exist 
global (very weak) solutions if
$\theta\in (0,2)$ (subcritical case, see Proposition~\ref{reloumaisafaire} below), 
but solutions must explode in finite time if $\theta\geq 2$ 
(supercritical case).

\vip

To our knowledge, the supercritical case has not been studied in details, and we aim
to describe precisely the explosion phenomenon. Informally, we will show the following 
(see Theorem~\ref{theoremecollisions}
below).
We assume that $\theta \geq 2$ and $N>3\theta$, we set $k_0=\lceil 2N/\theta \rceil \in \ig 7,N\id$.
There exists a (very weak) solution  $(X_t)_{t\in [0,\zeta)}$ to \eqref{EDS}, with $\zeta<\infty$ a.s.
and such that $X_{\zeta-}=\lim_{t\to \zeta-} X_t$ exists. Moreover, there is a cluster containing precisely 
$k_0$ particles in the configuration $X_{\zeta-}$, and no cluster containing strictly more than $k_0$ particles.
Such a cluster containing $k_0$ particles is inseparable, so that \eqref{EDS} 
is meaningless (even in a very weak sense)
after $\zeta$.
Just before explosion, there are infinitely many $k_1$-ary collisions, where $k_1=k_0-1$.
If $(k_0-3)(2-(k_0-2)\theta/N)<2$, we set $k_2=k_1-2$ and just before each  $k_1$-ary collision, 
there are infinitely many $k_2$-collisions.
Else, we set $k_2=k_1$. In any case, there are infinitely many binary collisions just 
before each $k_2$-ary collision.
During the whole time interval $[0,\zeta)$, there are no $k$-ary collisions, for any
$k \in \ig 3,k_2-1 \id$. 

\vip

This phenomenon seems surprising and original, in particular because of the gap between
binary and $k_2$-ary collisions.

\subsection{Sets of configurations}\label{conf}
We introduce, for all $K \subset \ig 1,N\id$ and all $x=(x^1,\dots,x^N) \in (\rr^{2})^{N}$,
\begin{align*} 
S_{K}(x) = \frac{1}{|K|} \sum _{i \in K} x^{i} \in \rr^2 \quad \hbox{and}\quad
R_{K}(x) = \sum _{i \in K} \| x^{i} -S_{K}(x) \|^2 = \frac{1}{2|K|} 
\sum _{i,j \in K} \|x^{i}-x^{j}\|^{2} \geq 0.
\end{align*}
Here $|K|$ is the cardinal of $K$ and $\|\cdot\|$ stands for the Euclidean norm in $\rr^2$. 
Observe that 
$R_K(x)=0$ if and only if all the particles indexed in $K$ are at the same place.
We also set, for $k\geq 2$,
$$
E_k= \Big\{ x \in (\rr^{2})^{N} : \forall K \subset \ig 1,N\id \mbox{ with cardinal } 
|K| = k, \; R_{K}(x) >0 \Big\},
$$
which represents the set of configurations with no cluster of $k$ (or more) particles.
Observe that $E_k=(\rr^2)^N$ for all $k> N$.

\subsection{Bessel processes}
We recall that a squared Bessel process  $(Z_{t})_{t\ge 0}$ of dimension $\delta \in \rr $ is a 
nonnegative solution, killed when it reaches $0$ if $\delta \leq 0$, of the equation
$$Z_{t} = Z_{0} + 2 \int _{0} ^{t} \sqrt{Z_{s}}\dd W_{s} + \delta t,$$
where $(W_t)_{t\geq 0}$ is a $1$-dimensional Brownian motion. We then say that $(\sqrt{Z_t})_{t\geq 0}$ 
is a Bessel 
process of dimension $\delta$.
This process has the following property, see Revuz-Yor \cite[Chapter XI]{ry}:
\vip
\noindent $\bullet$ if $\delta \ge 2$, then a.s., for all $t>0$, $Z_{t} >0$;
\vip
	
\noindent $\bullet$ if $\delta \in (0,2)$, then a.s., $Z$ 
is reflected infinitely often at $0$;
\vip
\noindent $\bullet$ if $\delta \le 0$, then $Z$ a.s. hits $0$ and is then killed.

\vip
Applying informally the It\^o formula, one finds that $Y_t=\sqrt{Z_t}$ should solve
$$
Y_t = Y_0 + W_t + \frac{\delta-1}{2} \intot \frac{\dd s}{Y_s},
$$
which resembles \eqref{EDS} in that we have a Brownian excitation in competition 
with an attraction by $0$, or a repulsion by $0$, depending on the value of $\delta$, 
proportional to $1/r$.
This formula rigorously holds true only when 
$\delta>1$, see \cite[Chapter XI]{ry}.

\subsection{Some important quantities}\label{imp}

Consider a (possibly very weak) solution $(X_t)_{t\geq 0}$ to \eqref{EDS}.
As we will see, when fixing a subset $K \subset \ig 1,N \id$ and when
neglecting the interactions between the particles indexed in $K$ and the other ones,
one finds that the process $(R_K(X_t))_{t\geq 0}$ behaves like a squared Bessel process with dimension 
$d_{\theta,N}(|K|)$, where
\begin{equation}\label{dtn}
d_{\theta , N}(k) = (k-1)\Big(2-\frac{k\theta }{N} \Big).
\end{equation}
Similar computations already appear in Ha\v{s}kovec-Schmeiser \cite{hs1}, see also \cite{fj}.
A little study, see Appendix~\ref{ppre}, see also Figure~\ref{nul?} and
Subsection~\ref{reex} for numerical examples,
shows the following facts. For $r \in \rr_+$, we set
$\lceil r \rceil = \min \{ n \in \nn : n \geq r\}$.

\begin{lemma} \label{dimensions}
Fix $\theta>0$ and $N \geq 2$ such that $N> \theta$. For
$k_{0} = \lceil \frac{2N}{\theta} \rceil \geq 3$, we have
\begin{equation}\label{hh}
d_{\theta , N}(k) >0\quad \hbox{if}\quad  k\in \ig 2,k_0-1\id \quad \hbox{and}\quad
d_{\theta , N}(k)\leq 0 \quad \hbox{if} \quad   k\geq k_0.
\end{equation}
We also define $k_1 = k_{0}-1$, and 
$$
k_2 = \left\{ \begin{array}{lll} k_{0}-2 &\hbox{ if }& d_{\theta , N}(k_{0}-2) <2,\\[4pt]
k_0-1 & \hbox{ if }& d_{\theta , N}(k_{0}-2) \geq 2. \end{array} \right.
$$
If $\theta\geq 2$ and $N> 3\theta$, then $k_0 \in \ig 7,N \id$ and it holds that
\vip
\noindent $\bullet$ $d_{\theta ,N}(2) \in (0,2)$;
\vip
\noindent $\bullet$ $d_{\theta ,N}(k) \ge 2$ if $ k \in \ig 3,k_2-1 \id$;
\vip
\noindent $\bullet$ $d_{\theta ,N} (k) \in (0,2)$  if $ k \in \{k_2, k_1 \}$;
\vip
\noindent $\bullet$ $d_{\theta ,N}(k) \le 0$  if  $k\ge k_{0}$.
\end{lemma}

\begin{figure}[ht]\label{nul?}
\noindent\fbox{\begin{minipage}{0.9\textwidth}
\vskip-.3cm
\caption{\small{Plot of $d_{\theta,N}(k)$ as a function of $k\in \ig 2,N\id$ with $N=9$ 
and with $\theta=2.35$ (left) and $\theta=2.42$ (right).}}
\centerline{\includegraphics[width=\textwidth]{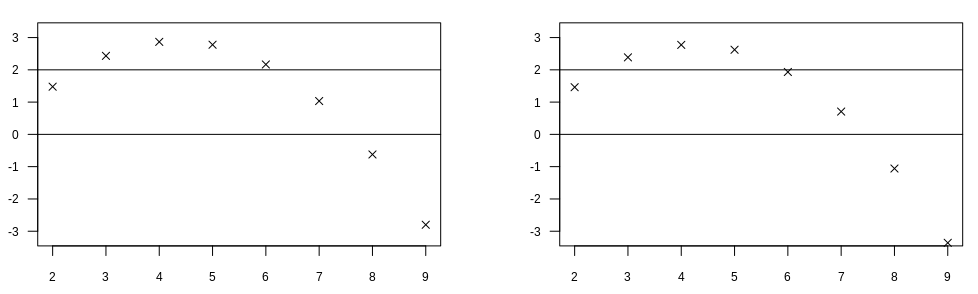}}
\vskip-0.2cm
\centerline{\begin{minipage}{0.75\textwidth}
\small{$k_0=8$, $k_1=7$, $k_2=7$ \hfill $k_0=8$, $k_1=7$, $k_2=6$}
\end{minipage}}
\end{minipage}}
\end{figure}

We thus expect that there may be some non sticky $k$-ary collisions for
$k \in \{2,k_2,k_1\}$, some sticky $k$-ary collisions when $k\geq k_0$, 
but no $k$-ary collision for $k\in \ig 3,k_2-1\id$.

\subsection{Generator and invariant measure}\label{inv}
\blue 
As we will see in Subsection~\ref{nonint}, the S.D.E. \eqref{EDS} cannot have a solution in the classical sense,
at least when $d_{\theta,N}(k_1) \in (0,1)$,
because the drift term cannot be integrable in time.
We will thus define a solution through the theory of the Dirichlet spaces. \bla

\vip

For $x=(x^1,\dots,x^N)\in(\rr^2)^N$ and for $\dd x$ the Lebesgue measure on $(\rr^2)^N$, we set 
\begin{equation}\label{mmu}
\bm(x) = \prod _{1\le i\neq j \le N} \|x^{i}-x^{j}\|^{-\theta/N} \quad \hbox{and}
\quad \mu(\dd x)=\bm(x)\dd x,
\end{equation}
\blue where $\{1\le i\neq j \le N\}$ stands for the set $\{(i,j) \in \ig 1, N\id^2 : i\neq j\}$.
\bla
Informally, the generator of the solution to \eqref{EDS} is given by $\cL^X$, where 
for $\varphi \in C^2((\rr^2)^N)$,
\begin{equation}\label{lx}
\cL^{X}\varphi(x) = \frac{1}{2}\Delta \varphi(x) 
- \frac{\theta }{N} \sum _{1\leq i\neq j \leq N} \frac{x^{i}-x^{j}}{\|x^{i}-x^{j}\|^{2}} \cdot 
\nabla _{x^{i}} \varphi(x)
= \frac{1}{2\bm(x)} \ddiv [\bm(x)\nabla \varphi(x)],
\end{equation}
\blue see \eqref{betm} for the last equality. \bla
It is well-defined for all $x\in E_2$ and $\mu$-symmetric.
Indeed, an integration by parts shows that
\begin{equation}\label{ipp}
\forall \; \varphi ,\psi \in C^2_c (E_2),\quad  \int _{(\rr ^2)^N} 
\varphi \cL^{X} \psi \; \dd \mu = 
-\frac12 \int _{(\rr ^2)^N} \nabla \varphi \cdot \nabla \psi \; \dd \mu = 
\int _{(\rr ^2)^N} \psi \cL^{X} \varphi \; \dd \mu.
\end{equation}

As we will see in Proposition~\ref{radon}, the measure $\mu$ is Radon on $(\rr^2)^N$ 
in the subcritical case $\theta \in (0,2)$, while it 
is Radon on $E_{k_0}$ (and not on $E_{k_0+1}$) in the supercritical case  $\theta \ge 2$.
This will allow us
to use some results found in 
Fukushima-Oshima-Takeda \cite{f} and to obtain the following existence result.

\begin{prop}\label{densite}
We fix $N\geq 2$ and $\theta>0$ such that $N> \theta$
and recall that $k_0=\lceil 2N/\theta \rceil$.
We set $\cX=E_{k_0}$ and $\cX_\triangle=\cX \cup\{\triangle\}$, where $\triangle$ is a cemetery point.
There exists a \blue diffusion \bla $\mathbb{X} = (\Omega^X,\cM^X,(X_t)_{t\geq 0},
(\PP_x^X)_{x\in \cXt})$ with values in $\cX_\triangle$, which is
$\mu$-symmetric, with regular Dirichlet space $(\cE^X,\cF^X)$ on $L^2((\rr ^2)^N,\mu)$ with core 
$C^\infty_c(\cX)$ defined by
$$
\hbox{for all}\quad  \varphi \in C^\infty_c(\cX), \quad
\mathcal{E}^{X} (\varphi,\varphi) = \frac12\int_{(\rr^2)^N} \|\nabla  \varphi\|^{2} \dd \mu
= - \int_{(\rr^2)^N}  \varphi \cL^X\varphi \;\dd  \mu
$$
and such that for all $x \in E_2$, all $t>0$, the law of $X_t$ under $\PP_x$ has a density with respect
to the Lebesgue measure on $(\rr^2)^N$.
We call such a process a $KS(\theta ,N)$-process and denote by $\zeta=\inf\{t\geq 0 : X_t= \triangle\}$
its life-time.
\end{prop}

We refer to Subsection~\ref{ap1} for a quick summary about the notions used in this proposition: 
\blue diffusion (i.e. continuous Hunt process), \bla
link between its generator, semi-group and Dirichlet space, definition of the
one-point compactification topology endowing $\cX_\triangle$, etc.
Let us mention that by definition, $\triangle$ is absorbing, i.e.
$X_t=\triangle$ for all $t\geq \zeta$. Also, $t\mapsto X_t$ is {\it a priori} continuous on $[0,\infty)$ 
only for the one-point
compactification topology on $\cXt$, which precisely means that it is continuous for the usual
topology of $(\rr^2)^N$ during $[0,\zeta)$, and it holds that 
$\zeta=\lim_{n\to \infty} \inf\{t\geq 0 : X_t\notin  \cK_n\}$
for any increasing sequence of compact subsets 
$(\cK_n)_{n\geq 1}$ of $E_{k_0}$ such that $\cup_{n\geq 1} \cK_n=E_{k_0}$.

\vip

As we will see in Remark~\ref{nsar}, for all $x\in E_2$, 
under $\PP^X_x$, $X_t$ solves \eqref{EDS} during $[0,\sigma)$,
where $\sigma=\inf\{t\geq 0 : X_t \notin E_2\}$. By the Markov property, this implies $X_t$ solves \eqref{EDS} 
during any open time-interval on which it does not visit $\cX\setminus E_2$.

\vip

When $\theta<2$, we have $k_0 > N$ and thus $E_{k_0}=(\rr^2)^N$.
We will easily prove the following non-explosion result, which is almost contained
in Cattiaux-P\'ed\`eches \cite{cp}, who treat the case where $\theta \in (0,2(N-2)/(N-1))$.

\begin{prop}\label{reloumaisafaire}
Fix $\theta \in (0,2)$ and $N\geq 2$.
Consider the $KS(\theta ,N)$-process $\xX$ introduced in Proposition~\ref{densite}. 
For all $x \in E_2$, we have $\PP_x(\zeta=\infty)=1$.
\end{prop}

When $\theta\geq 2$, we will see that there is explosion. Note that any collision of a set of 
$k\geq k_0$ particles makes the process leave $E_{k_0}$ and thus explode.
However, it is not clear at all at this point that explosion is due to a precise collision: 
the process could explode
because it tends to infinity (which is not hard to exclude) or to the boundary 
of $E_{k_0}$ with possibly many 
oscillations.

\subsection{Main result}

To avoid any confusion, let us define precisely what we call a collision.

\begin{defi}\label{dfcol}
(i) For $K \subset \ig 1,N\id$, 
we say that there is a $K$-collision in the configuration $x\in (\rr^2)^N$
if $R_K(x)=0$ and if  $R_{K\cup \{i \}}(x) >0$ for all $i\in \ig1,N\id\setminus K$.

\vip

(ii) For a $(\rr^2)^N$-valued process $(X_t)_{t\in [0,\zeta)}$, we say that there is a $K$-collision at time 
$s\in [0,\zeta)$ if there is a $K$-collision in the configuration $X_s$.
\end{defi}

The main result of this paper is the following description of the explosion phenomenon.

\begin{thm}\label{theoremecollisions}
Assume that $\theta \geq 2$, that $N>3\theta$  and recall that $k_0 \in \ig 7,N\id$, $k_1 = k_0-1$
and $k_2\in \{k_0-1,k_0-2\}$ were defined in Lemma~\ref{dimensions}.
Consider the $KS(\theta ,N)$-process $\xX$
introduced in Proposition~\ref{densite}. For all $x \in E_2$,
we $\PP_x$-a.s. 
have the following properties:
\vip
(i) $\zeta$ is finite and $X_{\zeta- } =\lim _{t \rightarrow \zeta- } X_{t}$ exists 
for the usual topology of $(\rr^2)^N$;
\vip
(ii) there is $K_0\subset \ig 1,N \id $ with cardinal $|K_0| = k_0$ such that 
there is a $K_0$-collision in the configuration $X_{\zeta-}$, and
for all $K\subset \ig 1,N \id$ such that $|K| >k_0$, there is no $K$-collision 
in the configuration $X_{\zeta -}$;
\vip
(iii) for all $t\in [0,\zeta)$ and all $K \subset K_0$ with cardinal $|K| = k_1$, 
there is an infinite number of $K$-collisions during $(t,\zeta)$ and none of these instants 
of $K$-collision 
is isolated;
\vip
(iv) if $k_2= k_0-2$, then for all $L\subset K\subset K_0$ such that 
$|L| = k_2$ and $|K| = k_1$, for all instant $t \in (0,\zeta)$ of $K$-collision 
and all $s\in [0,t)$,
there is an infinite number of $L$-collisions during $(s,t)$ and none of these instants of 
$L$-collision is isolated;
\vip
(v) for all $ K \subset \ig 1,N \id$ with cardinal $|K| \in \ig 3, k_2-1 \id$, 
there is no $K$-collision during $[0,\zeta )$;
\vip
(vi) for all $L\subset K\subset K_0$ such that $|L| = 2$ and $|K| = k_2$, 
for all instant $t \in (0,\zeta)$ of $K$-collision and all $s\in [0,t)$,
there is an infinite number of $L$-collisions during $(s,t)$ and none of these instants 
of $L$-collision is isolated.
\end{thm}

The condition  $\theta\geq 2$ is crucial to guarantee that $k_0 \leq N$. On the contrary, we impose
$N>3\theta$ for simplicity, because Lemma~\ref{dimensions} does not hold true without 
this assumption.
The other cases may also be studied, but we believe this
is not very restrictive:  $N$ is thought as very large when compared to $\theta$, at 
least as far as the 
approximation of the Keller-Segel equation is concerned. 

\subsection{Comments}\label{reex}
Let us mention that the very precise values of $N$ and $\theta$ influence the value $k_2$.

\vip

(a) If $N=200$ and $\theta=4.04$, we have $k_0=100$, $k_1=99$ and $k_2=98$.

\vip

(b) If $N=200$ and $\theta=4.015$, we have $k_0=100$ and $k_1=k_2=99$.

\vip

Let us describe informally, in the chronological order, what happens e.g. in case (b) above.
We start with $200$ particles at $200$ different places.
During the whole story, there is no $k$-ary collision for $k=3,\dots,98$. 
Here and there, two particles meet, they collide an infinite number of times, but manage to separate.
Then at some times, we have $98$ particles close to each other and there are many 
binary collisions.
Then, if a $99$-th particle arrives in the same zone (and this eventually occurs), there are infinitely many
$99$-ary collisions, with infinitely many binary collisions of all possible pairs before each. 
These $99$ particles may manage to
separate forever, or for a large time, but if a $100$-th particle arrives in the zone 
(and this situation eventually occurs), then there are infinitely many $99$-ary collisions of
all the possible subsets and, finally, a $100$-ary collision producing explosion, and the story is finished.
Informally, the resulting cluster is not able to separate, because the attraction
dominates the Brownian excitation, since a Bessel process of dimension 
$d_{\theta,N}(100)\leq 0$ is absorbed
when it reaches $0$. We hope to be able, in a future work, to propose and justify a 
model describing what
happens after explosion.

\subsection{References}
In many papers about the Keller-Segel equation, the parameter
$\chi=4\pi \theta$ is used, so that the transition at $\theta=2$ corresponds to the transition
at $\chi=8\pi$. As already mentioned, this nonlinear P.D.E. has been introduced to 
model the collective 
motion of cells, which are attracted by a chemical substance that they emit. It 
describes the density $f_t(x)$
of particles (cells) with position $x\in \rr^2$ at time $t\geq 0$ and writes,
in the so-called parabolic-elliptic case,
\begin{equation}\label{pde}
\partial_t f_t(x)+ \theta \ddiv_x ((K\star f_t)(x) f_t(x))=\frac12 \Delta_xf_t(x),
\quad \hbox{where}\quad
K(x)=-\frac{x}{|x|^2}.
\end{equation}
Informally, this solution should be the mean-field 
limit of the particle system \eqref{EDS} as $N\to\infty$.

\vip

We refer to the recent review paper on \eqref{pde} by Arumugam-Tyagi \cite{at}. 
The best existence of a global solution to \eqref{pde},
including all the subcritical parameters $\theta\in (0,2)$, is due to 
Blanchet-Dolbeault-Perthame \cite{bdp}. 
The blow-up of solutions to \eqref{pde},
in the supercritical case $\theta>2$, has been studied e.g. by Fatkullin \cite{fa} 
and Velasquez \cite{v1,v2}.
More close to our study, Suzuki \cite{s} has shown, still in the supercritical case, 
the appearance of a Dirac mass with a precise (critical)
weight, at explosion. This is the equivalent, in the limit $N\to\infty$,
to the fact that $\lim_{t\to \zeta-}X_t$ exists and corresponds to a $K$-collision, for some
$K\subset \ig 1,N\id$ with precise cardinal $k_0$.
Let us finally mention Dolbeault-Schmeiser \cite{ds}, who propose a post-explosion 
model in the supercritical case.
\vip
Concerning particle systems associated with \eqref{pde}, let us mention
Stevens \cite{st}, who studies a physically more complete particle system with two types 
of particles, 
for cells and chemo-attractant particles, with a regularized attraction kernel.
Ha\v{s}kovec and Schmeiser \cite{hs1,hs2} study a particle system closer to \eqref{EDS}, 
but with, again,
a regularized attraction kernel. 
\vip
Cattiaux-P\'ed\`eches \cite{cp}, as well as \cite{fj}, study the system
\eqref{EDS} without regularization in the subcritical case: existence of a global 
solution to \eqref{EDS}
has been shown in \cite{fj} when $\theta \in (0,2(N-2)/(N-1))$, and uniqueness of this solution has 
been established in \cite{cp}. Also, the theory of Dirichlet spaces has been used in \cite{cp}
to build a solution to \eqref{EDS}. Finally,
the limit as $N\to \infty$ to a solution of \eqref{pde}
is proved in \cite{fj} in the very subcritical case where $\theta\in (0,1/2)$, 
up to extraction
of a subsequence.
This last result has been improved by Bresch-Jabin-Wang \cite{bjw}, who remove the 
necessity of extracting
a subsequence and consider the (still very subcritical) case where $\theta\in (0,1)$. 
Olivera-Richard-Tomasevic \cite{ort} have recently established the $N\to\infty$ convergence
of a smoothed version of \eqref{EDS}, for all the subcritical
cases $\theta \in (0,2)$. Informally, in view of the mean distance between particles, 
the regularization used in \cite{ort} is not far from being physically reasonable.
There is also a related paper of Jabir-Talay-Tomasevic \cite{jtt}
about a one-dimensional but more complicated parabolic-parabolic model.

\vip
Let us finally mention the seminal paper of Osada \cite{o}, see also \cite{fhm} 
for a more recent study, 
which concerns the vortex model: this is very close to
\eqref{EDS}, but the attraction $-x/|x|^2$ is replaced by a rotating interaction $x^\perp/|x|^2$, 
so that particles never encounter.

\subsection{Originality and difficulties}
 To our knowledge, this is the first study
of the supercritical Keller-Segel particle system near explosion.
We hope that this model, which makes compete diffusion and Coulomb interactions,
is very natural from a physical point of view, beyond the Keller-Segel community.
The phenomenon we discovered seems surprising and original, in particular
because of the gap between binary and $k_2$-ary collisions.
We are not aware of other works, possibly dealing with other models, showing such a behavior.

\vip

In Section~\ref{miotp}, we give the main arguments of the proofs, with quite a high level of precision,
but ignoring the technical issues.
While it is rather clear, intuitively, that the process explodes in finite time when 
$\theta\geq 2$ and that no $K$-collisions may occur
for $|K|\in \ig 3,k_2-1\id$, the continuity at explosion is \blue delicate, \bla and some rather deep
arguments are required to show that that each $k_2$-ary collision is preceded
by many binary collisions, that each $k_1$-ary collision is preceded
by many $k_2$-ary collisions, that explosion is preceded by many $k_1$-ary collisions,
and that explosion is due to the emergence of a cluster with precise size $k_0$ (which more or
less says that a possible $(k_0+1)$-ary collision would necessarily be preceded by a $k_0$-collision).

\vip

Actually, the rigorous proofs are made technically much more
involved than those presented in Section~\ref{miotp}, because we have to use the theory of Dirichlet spaces.
Due to the singularity of the interactions and to the occurrence of many collisions near explosion,
we can unfortunately not, as already mentioned, deal at the rigorous level directly with the S.D.E. \eqref{EDS}.
We thus have to use some suitable heavy versions of some usual tools such as 
It\^o's formula, Girsanov's theorem, time-change, etc.

\subsection{Plan of the paper}

In Section~\ref{nota}, we introduce some notation of constant use.
In Section~\ref{miotp}, we explain the main ideas of the proofs, with quite a high level of precision, but
without speaking of the heavy technical issues related to the use of the theory of Dirichlet spaces.
Section~\ref{constr} is devoted to the existence of a first version of the Keller-Segel process,
namely without the property that $\PP^X_x \circ X_t^{-1}$ has a density, and we introduce a spherical
Keller-Segel process.
In Section~\ref{deccc}, we show that the Keller-Segel process enjoys a crucial and noticeable decomposition
in terms of a $2$-dimensional Brownian motion, a squared Bessel process and a spherical process.
Section~\ref{cutcut} consists in building some smooth approximations of some indicator functions that
behave well under the action of the generator $\cL^X$.
In Section~\ref{gggir}, we make use of the Girsanov theorem to prove that when two sets of particles of
a $KS$-process are not too close from each other, they behave as two independent smaller $KS$-processes.
In Section~\ref{explacon}, we study explosion and continuity (in the usual sense) at the explosion time.
Section~\ref{specialc} is devoted to establish some parts of Theorem~\ref{theoremecollisions}
for some particular ranges of values of $N$ and $\theta$.
Using the results of Section~\ref{gggir}, we
reduce the general study to the special cases of Section~\ref{specialc} and we prove, 
in Section~\ref{conccc},
that the conclusions of 
Theorem~\ref{theoremecollisions} hold true \blue quasi-everywhere. \bla
Finally, in Section~\ref{rq}, we remove the restriction \blue {\it quasi-everywhere} \bla and 
conclude the proofs of Propositions~\ref{densite} and \ref{reloumaisafaire} 
and of Theorem~\ref{theoremecollisions}.
\vip
Appendix~\ref{ppre} contains a few elementary computations: proof of Lemma~\ref{dimensions}, proof
that $\mu$ is Radon on $E_{k_0}$, and study of a similar measure on a sphere.
We end the paper with Appendix~\ref{aapp}, that summarizes all the notions and results
about Dirichlet spaces and Hunt processes we 
shall use.

\section{Notation}\label{nota}
We introduce the spaces
\begin{gather*}
H = \Big\{ x \in (\rr^{2})^{N}: S_{\ig 1,N\id}(x)=0 \Big\},\quad
S = \Big\{ x\in (\rr ^2)^N : \sum_{i=1}^N \|x^i\|^2=1 \Big\}\quad
\hbox{and} \quad \sS = H  \cap S.
\end{gather*}
For $u\in \sS$, we have $S_{\ig 1,N\id}(u)=0$ and $R_{\ig 1,N\id}(u)=1$.
We consider the (unnormalized) Lebesgue measure $\sigma$ on $\sS$, as well as, recall \eqref{mmu},
\begin{equation}\label{beta}
\beta (\dd u)= \bm(u) \sigma (\dd u).
\end{equation}

We define $\gamma:\rr^2 \to (\rr^2)^N$ by
$\gamma(z)=(z,\dots,z)$ and $\Psi : \rr^2 \times \rr_+^*\times \sS \to E_N\subset (\rr^2)^N$ by
\begin{equation}\label{Psi}
\Psi(z,r,u)=\gamma(z)+ \sqrt r \;u,\qquad \hbox{i.e.} \qquad
(\Psi(z,r,u))^i=z - \sqrt{r} u^i \quad \hbox{for}\quad  i \in \ig 1,N \id.
\end{equation}
We have $S_{\ig 1,N\id}(\Psi(z,r,u))=z$ and $R_{\ig 1,N\id}(\Psi(z,r,u))=r$.
\vip
The orthogonal projection $\pi_H: (\rr^2)^N \to H$ is given by
$$
\pi_H(x)=x-\gamma(S_{\ig 1,N\id}(x)),
\qquad \hbox{i.e.} \qquad
(\pi_H(x))^i=x^i -S_{\ig 1,N\id}(x) \quad \hbox{for}\quad  i \in \ig 1,N \id
$$
and we introduce $\Phi_\sS : E_{N}\to \sS$ defined by
\begin{align}\label{phis}
\Phi_\sS(x)=\frac{\pi_H x}{||\pi_H x||}, \qquad \hbox{i.e.}\qquad 
(\Phi_\sS(x))^i=\frac{x^i-S_{\ig 1,N\id}(x)}{\sqrt{R_{\ig 1,N\id}(x)}}\quad \hbox{for}\quad i \in \ig 1,N \id.
\end{align}
For $x \in (\rr^2)^N \setminus\{0\}$, the projections $\pi_{x^\perp}:(\rr^2)^N\to x^\perp$ 
and $\pi_x : (\rr^2)^N \to $ span$(x)$ are given by
$$
\pi_{x^\perp}(y)=y-\frac{x \cdot y}{||x||^2} x \qquad \hbox{and}\qquad 
\pi_x(y)= \frac{x \cdot y}{||x||^2} x,
$$
where $x \cdot y = \sum_{i=1}^N x^i \cdot y^i$.

\vip

We denote by $b:E_2 \to(\rr^2)^N$ the drift coefficient of \eqref{EDS}: for $x=(x^1,\dots,x^N)\in E_2$,
\begin{equation}\label{betm}
b(x)= \frac{\nabla \bm(x)}{2\bm(x)}\blue = \frac{\nabla \log \bm(x)}{2} \bla \in (\rr^2)^N,\qquad \hbox{i.e.}
\qquad b^i(x)=-\frac{\theta}{N} \sum_{j\neq i}\frac{x^i-x^j}{\|x^i-x^j\|^2} \in \rr^2
\end{equation}
\blue for  $i \in \ig 1,N \id$. Indeed, since $\log \bm(x)=-\frac{\theta}{2N}\sum_{1\leq i \neq j \leq N} 
\log ||x^i-x^j||^{2}$,
we e.g. have
$$
\frac{\nabla_{x^1} \log \bm(x)}2
= - \frac\theta{4N} \nabla_{x^1}\Big[
\sum_{i=2}^N \log ||x^i-x^1||^2 + \sum_{j=2}^N \log ||x^1-x^j||^2\Big]
= -\frac\theta {2N}\nabla_{x^1}\sum_{j=2}^N \log ||x^1-x^j||^2,
$$
whence
$$
\frac{\nabla_{x^1} \log \bm(x)}{2}=-\frac{\theta}{N} \sum_{j=2}^N\frac{x^1-x^j}{\|x^1-x^j\|^2}.
$$
\bla

Finally, we introduce the natural operators defined for $\varphi \in C^1(\sS)$ and $u \in \sS$ by
\begin{equation}\label{nablaS}
\nabla_\sS \varphi(u)= \nabla [\varphi\circ \Phi_\sS] (u) \in (\rr^2)^N \quad \hbox{and} \quad
\Delta_\sS \varphi(u)=\Delta[\varphi\circ \Phi_\sS](u) \in \rr,
\end{equation}
where
$\nabla$ and $\Delta$ stand for the usual gradient and Laplacian in $(\rr^2)^N$.
Since $\sS \subset E_{N}\subset (\rr^2)^N$, with $E_{N}$ open, and since $\Phi_\sS$ is smooth on
$E_N$, we can indeed define $\nabla [\varphi\circ \Phi_\sS] (u)$ and $\Delta[\varphi\circ \Phi_\sS] (u)$ 
for all $u \in \sS$.
Similarly, for $\varphi \in C^1(\sS,(\rr^2)^N)$ and $u \in \sS$, we set 
\begin{equation}\label{ddivS}
\ddiv_\sS \varphi(u)= \ddiv [\varphi\circ \Phi_\sS] (u) \in \rr.
\end{equation}
To conclude this subsection, we note that 
for all $\varphi \in C^\infty ( (\rr^2)^N)$, for all $u \in \sS$,
\begin{align}\label{tradgrad}
\nabla _\sS (\varphi|_\sS) (u) = \pi _H (\pi _{u^\perp}(\nabla \varphi (u))).
\end{align}
Indeed, it suffices to observe that setting $G(x)=x/||x||$ for all $x\in (\rr^2)^N \setminus \{0\}$, 
we have $\Phi_\sS=G\circ \pi_H$, $\dd_x G=\pi_{x^\perp}/||x||$ and $\dd_x \pi_H=\pi_H$
 and that for $u \in \sS$, we have $\pi_H(u)=u$ and $||\pi_H(u)||=1$. 

\section{Main ideas of the proofs}\label{miotp}

Here we explain the main ideas of the proofs of
Proposition~\ref{reloumaisafaire} and Theorem~\ref{theoremecollisions}.
The arguments below are completely informal. In particular, we do as if our $KS(\theta,N)$-process
$(X_t)_{t\in [0,\zeta)}$ was a true solution to \eqref{EDS} until explosion
and we apply It\^o's formula without care.
We always assume at least that $N\geq 2$, $\theta>0$ and $N> \theta$,
which implies that  $k_0=\lceil 2N/\theta\rceil \geq 3$.

\subsection{Existence}\label{existence}
The existence of the $KS(\theta,N)$-process
$(X_t)_{t\in [0,\zeta)}$, with values in $E_{k_0}$, is an easy application of 
Fukushima-Oshima-Takeda \cite[Theorem 7.2.1]{f}.
The only difficulty is to show that
the invariant measure $\mu$ is a Radon on $E_{k_0}$, see Proposition~\ref{radon}. 
The process may explode, i.e.
get out of any compact subset of $E_{k_0}$ in finite time. Observe that a typical
compact subset of $E_{k_0}$ is of the form, for $\e>0$,
$$
\cK_\e=\{x \in (\rr^2)^N : ||x||\leq 1/\e \mbox{ and for all } K \subset \ig 1,N \id 
\mbox{ such that } |K|=k_0,\; R_K(x)\geq \e\}.
$$

\subsection{Center of mass and dispersion process}\label{cmdp}
One can verify, using It\^o's formula,
that the center of mass $S_{\ig 1,N\id}(X)$ is a $2$-dimensional 
Brownian motion with diffusion constant $N^{-1/2}$, that
the dispersion process $R_{\ig 1,N \id}(X)$ is a squared Bessel process with 
dimension $d_{\theta,N}(N)$, 
recall \eqref{dtn}, and that these two processes are independent.

\vip

Consequently, if $\zeta<\infty$, the limits $\lim_{t\to \zeta-} S_{\ig 1,N\id}(X_t)$ and 
$\lim_{t\to \zeta-} R_{\ig 1,N\id}(X_t)$ a.s. exist, and this 
implies that $\limsup_{t\to\zeta-}||X_t||<\infty$:
the process cannot explode to infinity, it can only explode because it tends to the boundary
of $E_{k_0}$.
If moreover $k_0 > N$ (i.e. if $\theta < 2$), this is sufficient to show that 
$\zeta=\infty$, since then
$E_{k_0}=(\rr^2)^N$.

\subsection{Behavior of distant subsets of particles}\label{introgirsanov}
Consider a partition $K_1,\dots ,K_p$ of $\ig 1,N \id$.
If we neglect interactions between particles of which the indexes are not in the same subset, 
we have, for each $\ell \in \ig 1,p \id$, setting $\ttheta_\ell= \theta|K_\ell|/N$,
$$
\dd X^i_t = \dd B^i_t - \frac{\ttheta_\ell}{|K_\ell|} \sum _{j\in K_{\ell}\setminus \{i\} } 
\frac{X^i_t-X^j_t}{\|X^i_t-X^j_t\|^2}\dd t,\qquad i \in K_\ell
$$
and we recognize a $KS(\ttheta_\ell,|K_{\ell}|)$-process.
\vip

During time intervals where particles indexed in different subsets are far enough 
from each other, we can indeed bound the interaction between those particles, so that
the Girsanov theorem tells us that 
$(X^{i}_t)_{i \in K_1},\dots,(X^{i}_t)_{i \in K_p}$ behave similarly,
in the sense of trajectories, as
independent $KS(\ttheta_1,|K_{1}|)$, ..., $KS(\ttheta_p,|K_{p}|)$-processes.

\vip

\subsection{\blue Brownian and Bessel behaviors \bla of isolated subsets of particles}\label{girsanovbessel}
Consider $K \subset \ig 1,N\id$. As seen just above, during time intervals 
where the particles indexed in $K$ are 
far from all the other ones, the system $(X^i_t)_{i \in K}$ behaves, in the sense of trajectories,
like a $KS(\theta|K|/N,|K|)$-process.
Hence, as seen in Subsection~\ref{cmdp}, \blue $S_K(X_t)$ behaves like
a $2$-dimensional Brownian motion with diffusion constant $|K|^{-1/2}$ \bla and 
$R_K(X_t)$ behaves like a squared Bessel process
of dimension $d_{\theta|K|/N,|K|}(|K|)$, which equals
$d_{\theta,N}(|K|)$, recall \eqref{dtn}.

\subsection{Continuity at explosion}\label{cae}
Here we assume that $N> \theta\geq 2$, so that $k_0 \in \ig 2,N\id$ and we explain
why a.s., $\zeta<\infty$ and $X_{\zeta-}=\lim_{t\to \zeta-} X_t$ exists, in the usual sense of $(\rr^2)^N$.

\vip

(a) We first show that $\zeta<\infty$ a.s. On the event where $\zeta=\infty$, the squared Bessel process
$R_{\ig 1,N\id}(X)$ is defined for all times. \blue Recall that $d_{\theta,N}(N)\leq 0$ (because $\theta\geq 2$)
and that a squared Bessel process with negative dimension can be defined on the whole time half-line
and a.s. becomes negative in finite time. Since  $R_{\ig 1,N\id}(X)\geq 0$ by definition,
this contradicts the fact that $\zeta=\infty$. \bla
\vip
\blue Similarly, one can show that a $KS(\theta,N)$-process
has no chance to be defined after
the first hitting time $\tau_K$ of $0$ by $R_K(X_t)$, where $|K|= k_0$: this makes the
choice of the space $E_{k_0}$ very natural.
Indeed, assume that $X$ is defined during $[0,\zeta')$ with $\zeta'>\tau_K$. 
Consider the maximal subset $L$ of $\ig 1,N\id$ containing $K$ 
and such that $R_L(X_{\tau_K})=0$. Then there is $\e>0$ such that during $[\tau_K,\tau_K+\e]\subset [0,\zeta')$,
the particles labeled in $L$ are far from the ones labeled outside $L$. 
By Subsection~\ref{girsanovbessel}, $(R_L(X_{\tau_K+t}))_{t\in [0,\e]}$ behaves like a squared Bessel process 
with dimension $d_{\theta,N}(|L|)$ issued from $0$. But such a process is instantaneously negative,
because $d_{\theta,N}(|L|)\leq 0$ (since $|L|\geq k_0$). Since  $R_{L}(X)\geq 0$,
this contradicts the fact that $\tau_K \in [0,\zeta')$.
\bla

\vip

(b) We next show by reverse induction that a.s. for all $K\subset \ig 1,N\id$ with $|K|\geq 2$,
we have 
\begin{equation}\label{art}
\mbox{either}\quad \lim_{t\to \zeta-} R_K(X_t)=0  \quad \mbox{or}\quad \liminf_{t\to \zeta-} R_K(X_t)>0.
\end{equation}
If $K=\ig 1,N\id$, $\lim_{t\to \zeta-} R_K(X_t)$ exists by continuity of the (true) 
squared Bessel process
$R_K(X_t)$ and this implies the result. We now fix $n \in \ig 3,N \id$ and assume 
that \eqref{art} holds true
for all $K$ such that $|K|\geq n$. We consider $K\subset \ig 1,N\id$ with $|K|=n-1$: 
by induction assumption, 
either
there is $i \notin K$ such that $\lim _{t \to \zeta- }R_{K\cup\{i\}}(X_t)= 0$ 
and then $\lim _{t \to \zeta- }R_{K}(X_t)= 0$,
or for all $i\in \ig 1,N\id \setminus K$, $\liminf_{t \to \zeta- }R_{K\cup\{i\}}(X_t)>0$.
In this last case, and when $\limsup_{t\to\zeta-} R_K(X_t)>0$ and $\liminf_{t\to \zeta-} R_K(X_t)=0$
(which is the negation of \eqref{art}),
there are $\alpha>0$ and $\e>0$ such that (i) 
$R_K(X_t)$ upcrosses $[\e/2,\e]$ infinitely often during $[\zeta-\alpha,\zeta)$ and
(ii) for all $t\in[\zeta-\alpha,\zeta)$ such that
$R_K(X_t)<\e$, the particles indexed in $K$ are far from all the other ones
(because then $R_K(X_t)$ is small and $R_{K\cup\{i\}}(X_t)$ is large for all $i \notin K$),
so that $R_K(X_t)$ behaves like a squared Bessel process with dimension $d_{\theta,N}(|K|)$,
see Subsection~\ref{girsanovbessel}.
Points (i) and (ii) are in contradiction, since a squared Bessel process is continuous
and thus cannot upcross $[\e/2,\e]$ infinitely often during a finite time interval.

\vip
\blue
(c) We now show that $\lim_{t\to\zeta-}X_t$ exists. Using (b) and the (random) equivalence relation on $\ig1,N\id$
defined by $i\sim j$ if and only if $\lim_{t\to \zeta-} R_{\{i,j\}}(X_t)=0$, one can build a (random) partition
$\bK=(K_p)_{p\in \ig 1,\ell \id}$ of $\ig 1,N\id$ such that for all $p\in\ig 1,\ell\id$,
$\lim_{t\to\zeta-}R_{K_p}(X_t)=0$ and $\liminf_{t\to\zeta-} \min_{i\notin K_p}R_{K_p\cup\{i\}}(X_t)>0$.
Hence, there is $\alpha \in [0,\zeta)$ such that for all $p\neq q$, the particles labeled
in $K_p$ are far from the ones labeled in $K_q$ during $[\alpha,\zeta)$. As seen in Subsection 
\ref{girsanovbessel}, we conclude that for all $p \in \ig 1,\ell\id$, $S_{K_p}(X_t)$ behaves
like a Brownian motion during $[\alpha,\zeta)$, and thus $M_p=\lim_{t\to \zeta-} S_{K_p}(X_t)$ exists. Since moreover
$\lim_{t\to\zeta-}R_{K_p}(X_t)=0$, we deduce that for all $i \in K_p$, $\lim_{t\to\zeta-} X^i_t=M_p$.
As a conclusion $\lim_{t\to\zeta-} X^i_t$ exists for all $i\in \ig1,N\id$.
\bla

\subsection{A spherical process}\label{dec} 
We recall that $\sS$, $\pi_H$, $\pi_{u^\perp}$ and $b$ were introduced in Section~\ref{nota} and introduce 
the possibly exploding (with life-time $\xi$)
process $(U_t)_{t\in [0,\xi)}$ with values in $\sS  \cap E_{k_0}$,
informally solving (we will also use here the theory of Dirichlet spaces),  
for some given $U_0 \in \sS \cap E_{k_0}$ and some $(\rr^2)^N$-valued Brownian motion
$(B_t)_{t\geq 0}$,
$$
U_{t} = U_{0} + \int _{0} ^{t} \pi _{U_{s}^{\perp}} \pi_{H} \dd B_{s} 
+ \int _{0}^{t} \pi _{U_{s}^{\perp}} \pi _{H} b(U_{s})\dd s - \frac{2N-3}{2} \int _{0}^{t} U_{s}\dd s.
$$
We call such a process a $SKS(\theta,N)$-process.

\vip
One can check that this process is $\beta$-symmetric, where $\beta$ is defined in \eqref{beta}, and
that $\beta$ is Radon on $\sS \cap E_{k_0}$, see Proposition~\ref{radonsphere}. 
And we will see that if
$k_0\geq N$, then $\beta(\sS)<\infty$, so that the process $(U_t)_{t\geq 0}$ is
non-exploding and positive recurrent.

\subsection{Decomposition of the process}\label{decprocess}
We assume that $N\geq 2$ and $\theta>0$ are such $d_{\theta,N}(N)<2$ and, as usual, $N> \theta$.
We consider a $2$-dimensional Brownian $(M_t)_{t\geq 0}$ with 
diffusion constant $N^{-1/2}$, 
a squared Bessel process $(D_t)_{t\in [0,\tau_D)}$ with dimension $d_{\theta,N}(N)$ killed
when it hits $0$, with life-time $\tau_D$,
and a $SKS(\theta,N)$-process $(U_t)_{t\in [0,\xi)}$,
these three processes being independent. We introduce the time-change
$$
A_{t} = \int _{0}^{t} \frac{\dd s}{D_s}, \quad t \in [0,\tau_D).
$$
Since $\tau_D<\infty$ (because $d_{\theta,N}(N)<2$), since $D_{\tau_D}=0$ and since, roughly, the
paths of $(\sqrt{D_t})_{t \in [0,\tau_D)}$ are $1/2$-H\"older continuous, it holds that
$A_{\tau_D}=\infty$ a.s. We  introduce 
the inverse function $\rho:[0,\infty) \to [0,\tau_D)$ of $A:[0,\tau_D)\to [0,\infty)$.

\vip

We also set $\zeta'=\rho_\xi$ and observe that $\zeta'\leq \tau_D$, since $\rho$ is $[0,\tau_D)$-valued,
and that $\zeta'<\tau_D$ if and only if $\xi<\infty$.
A \blue fastidious \bla but straightforward computation shows that, recalling \eqref{Psi},
$$
X_t=\Psi(M_t,D_t,U_{A_t}),\qquad \hbox{i.e.}\quad
X_t^i= M_t+ \sqrt{D_t}U_{A_t}^i,\qquad i \in \ig 1,N \id,
$$
which is well-defined during $[0,\zeta')$,
solves \eqref{EDS}. 

\vip

{\it This decomposition of the $KS(\theta,N)$-process, which is noticeable in that $U$ satisfies 
an autonomous
S.D.E. and thus is Markov, is at the basis of our analysis.} 

\vip
In other words, $(X_t)_{t\in [0,\zeta')}$ is the
restriction to the time interval $[0,\zeta')$ of a $KS(\theta,N)$-process
$(X_t)_{t\in [0,\zeta)}$. Moreover, we have $\zeta'=\zeta\land \tau_D$:  if $\xi$ is finite, then
$U$ gets out of $\sS \cap E_{k_0}$ at time $\xi$, so that $X$ gets out of $E_{k_0}$
at time $\zeta'=\rho_\xi<\tau_D$, whence $\zeta=\zeta'=\zeta\land \tau_D$; if next $\xi=\infty$, then
$\zeta'=\tau_D$ and $U$ remains in $E_{k_0}$ for all times, 
so that $X$ remains in $E_{k_0}$ during $[0,\tau_D)$, whence $\zeta \geq \tau_D$.

\vip
We have $S_{\ig1,N\id}(X_t)=M_t$ and $R_{\ig1,N\id}(X_t)=D_t$ for all $t\in [0,\zeta\land \tau_D)$,
because $U$ is $\sS$-valued.
By definition of $\sS$, the process $U$ cannot have any $\ig 1,N \id$-collision.
But for any $K \subset \ig 1,N \id$ with cardinal at most $N-1$,
\begin{gather}\label{etoile}
\hbox{$U$ has a $K$-collision at $t \in [0,\xi)$ if and only if $X$ has a 
$K$-collision at $\rho_t \in [0,\zeta\land \tau_D)$}.
\end{gather}
Moreover, as seen a few lines above, $\xi<\infty$ is equivalent to $\zeta<\tau_D$.
In other words, since $R_{\ig1,N\id}(X_t)=D_t$ for all $t\in [0,\zeta\land \tau_D)$
and since $\tau_D=\inf\{t>0 : D_t=0\}$, we have 
\begin{gather}
\xi<\infty \qquad \hbox{if and only if} \qquad \inf_{t\in[0,\zeta)} R_{\ig1,N\id}(X_t) >0.
\label{etoile2}
\end{gather}

\subsection{Some special cases}\label{mainidea}
Using the Girsanov theorem, see Subsection~\ref{girsanovbessel}, we will manage to reduce a large
part of the study to the special cases that we examine in the present subsection.
Here we explain the following facts, for $N\geq 2$ and $\theta>0$ with $N> \theta$:
\vip
\noindent (a) if $d_{\theta ,N}(N-1) \in (0,2)$, then a.s., 
$\tau_D=\inf\{t>0 : R_{\ig 1,N \id }(X_t)=0\}\leq \zeta$ 
and for all $r\in [0,\tau_D)$, 
all $K\subset \ig 1,N \id $ with $|K|=N-1$,  $(X_t)_{t\in [0,\zeta)}$ has infinitely many
$K$-collisions during $[r,\tau_D)$;
\vip
\noindent (b) if $d_{\theta ,N}(N-1)\leq 0$ (whence $k_0 \leq N-1$),
then a.s., $\inf_{t\in [0,\zeta)} R_{\ig 1,N \id }(X_t)>0$.
\vip

We keep the same notation as in the previous subsection.
\vip

(i) We first verify that in (a), $\tau_D\leq \zeta$. Since $d_{\theta ,N}(N-1)\in(0,2)$, 
it holds that $k_0\geq N$.
If first $k_0>N$, then $\zeta=\infty$ by Subsection~\ref{cmdp} and we are done. 
If next $k_0=N$, then
$\zeta<\infty$ and $X_{\zeta-}$ exists
by Subsection~\ref{cae}. Moreover
$X_{\zeta-}$ cannot belong to $E_{k_0}=E_{N}$ by definition of $\zeta$ 
and thus has its $N$ particles at the same place, i.e.  $R_{\ig 1,N \id }(X_{\zeta-})=0$: we have $\zeta=\tau_D$.

\vip

(ii) In (b), $\zeta<\infty$ by Subsection~\ref{cae} because
$d_{\theta ,N}(N-1)\leq 0$ implies that $\theta \geq 2$.

\vip

(iii) We consider, in any case, the spherical process $(U_t)_{t\in [0,\xi)}$ and 
assume that $\xi=\infty$.
\blue An It\^o computation \bla shows that for $K \subset \ig 1,N \id$,
for some $1$-dimensional Brownian motion $(W_t)_{t\geq 0}$,
\begin{align*}
 \dd R_K(U_t) =& 2 \sqrt{R_K(U_t) (1 - R_K(U_t))}\dd W_{t} + d_{\theta ,N}(|K|) \dd t -d_{\theta,N}(N)R_K(U_t) \dd t \\ 
&-\frac{2 \theta}{N} \sum _{i \in K, j \notin K} \frac{U^i_t-U^j_t}{||U^i_t-U^j_t||^2} \cdot (U^i_t-S_K(U_t)) \dd t.
\notag
\end{align*}
We fix $\e>0$ to be chosen later. During time intervals where 
$\min_{i\in K,j\notin K} \|U^i_t-U^j_t\| \geq \e$,
we thus have, for some constant $C_\e$,
\begin{align}\label{majorationbesselsphere}
\dd R_K(U_t) \le& 2 \sqrt{R_K(U_t) (1 - R_K(U_t) )}\dd W_{t} + d_{\theta ,N}(|K|) \dd t +  
C_\e \sqrt{R_K(U_t)} \dd t,
\end{align} 
where we used the Cauchy-Schwarz inequality and that $R_K(U_t)$ is uniformly bounded 
(because $U$ is $\sS$-valued).
Hence, still during time intervals where $\min_{i\in K,j\notin K} \|U^i_t-U^j_t\| \geq \e$, 
by comparison, $R_K(U_t)$ is
smaller than $S_t$, the solution to
\begin{equation}\label{zK}
\dd S_t = 2 \sqrt{S_t(1 - S_t)}\dd W_{t} + d_{\theta ,N}(|K|) \dd t +  
C_\e \sqrt{S_t} \dd t.
\end{equation}
And a little study involving scale functions/speed measures shows that this process
hits zero in finite time if and only if $d_{\theta ,N}(|K|)< 2$, exactly as a squared Bessel process
with dimension $d_{\theta ,N}(|K|)$.

\vip

(iv) We end the proof of (a). In this case, $k_0\geq N$, so that $U$ is non-exploding,
as seen in Subsection~\ref{dec}. Hence $\xi=\infty$ and we can use (iii).
Moreover, $U$ is  recurrent, still by Subsection~\ref{dec}.
We fix $K$ with $|K|=N-1$ and we choose $\e>0$ small enough so that we have  
$$
\beta\Big(\Big\{u \in \sS : \min_{i\in K,j\notin K} \|u^i-u^j\| \geq \e \Big\}\Big)>0,
$$ 
where
$\beta$ is the invariant measure \eqref{beta} of $U$. 
Hence the process $\min_{i\in K,j\notin K} \|U^i_t-U^j_t\|$ 
visits the zone 
$(\e,\infty)$ infinitely often and each time, $R_K(U)$ has a (uniformly) positive probability to hit $0$
by (iii) and since $d_{\theta ,N}(|K|)=d_{\theta ,N}(N-1)<2$.
Consequently, for any $s>0$, $(U_t)_{t\geq 0}$ has infinitely many $K$-collisions during $[s,\infty)$.
Recalling \eqref{etoile}  
and that $\zeta \land\tau_D=\tau_D$ by (i), we conclude that for any $r\in [0,\tau_D)$,
$(X_t)_{t \in [0,\zeta)}$ has infinitely many $K$-collisions during 
$[r, \tau_D)$.

\vip

(v) We finally complete the proof of (b). By \eqref{etoile2}, it is sufficient to show that $\xi<\infty$ a.s. 

\vip

Assume that $U$ is recurrent (and thus non-exploding). Then we take $K=\ig 2,N\id$ and apply the same 
reasoning as in (iv): since $d_{\theta,N}(|K|)\leq 0<2$, 
$R_K(U)$ hits zero in finite time and this makes $U$ get out of
$E_{N-1}$ and thus explode, since $U$ is $(E_{k_0}\cap\sS)$-valued and since $k_0\leq N-1$. 
We thus have a contradiction.

\vip

Hence $U$ is transient and it eventually gets out of the compact of $E_{k_0}\cap \sS$
$$
\cK = \{ u \in \sS : \forall K \subset \ig 1,N \id \mbox{ such that } |K|=k_0,\mbox{ we have } R_K(u) \ge \e \},
$$ 
for any fixed $\e>0$. Hence on the event where $\xi=\infty$, $\lim_{t\to \infty} \min_{|K|=k_0} R_K(U_t)=0$ a.s.
Recalling now that $k_0\leq N-1$ and that $U$ is $\sS$-valued (whence $R_{\ig 1,N\id}(U_t)=1$)
we can a.s. find $K$ with $|K|\in \ig k_0 , N-1 \id$ such that $\liminf_{t\to \infty} R_K(U_t)=0$ but
$\liminf_{t\to \infty} \min_{i\notin K} R_{K\cup\{i\}}(U_t)>0$. It is then not too hard to find
$\alpha>0$ and $\e>0$ such that each time $R_K(U_t)<\alpha$ (which often happens), 
all the particles indexed in $K$ 
are far from all the other ones with a distance greater than $\e>0$.
We conclude from (iii), since $d_{\theta,N}(|K|)\leq  0$ (because $|K|\geq k_0$) that
each time $R_K(U_t)<\alpha$, it has a (uniformly) positive probability to hit zero. 
On the event $\xi=\infty$, this will eventually happen,
so that the process 
$U$ will have a $K$-collision and thus will leave $E_{k_0}$ in finite time. Hence
$U$ will explode, so that $\xi<\infty$.

\subsection{Size of the cluster}
We assume that $N>3\theta\geq 6$. Hence $\zeta<\infty$ and $X_{\zeta-}$ exists, by Subsection~\ref{cae}.
Moreover, by definition of $\zeta$, we know that $X_{\zeta-} \notin E_{k_0}$.
We want now to show that
$X_{\zeta-} \in E_{k_0+1}$, i.e. that the cluster causing explosion is precisely composed
of $k_0$ particles.
If $k_0=N$, there is nothing to do, since then $E_{k_0+1}=(\rr^2)^N$.
Now if $k_0 \leq N-1$, we assume by contradiction, that there is 
$K \subset \ig 1,N\id$ with $|K|\geq k_0+1$
such that $R_K(X_{\zeta-})=0$ and $\min_{i \notin K}R_{K\cup\{i\}}(X_{\zeta-})>0$. 
Then there is $\alpha>0$
such that during $[\zeta-\alpha,\zeta)$, the particles indexed in $K$ are far from the other ones,
so that $(X^i_t)_{t\in[0,\zeta),i\in K}$ behaves like a $KS(\theta |K|/N,|K|)$-process
by Subsection~\ref{introgirsanov}. 
Observe now that $d_{\theta |K|/N,|K|}(|K|-1)=d_{\theta,N}(|K|-1)\leq 0$ because 
$|K|-1\geq k_0$ and $|K| > \theta |K|/N$ because $N> \theta$. 
We thus know from the special case (b) of Subsection~\ref{mainidea}
that $\inf_{t\in [\zeta-\alpha,\zeta)} R_K(X_t)>0$, which contradicts the fact that $R_K(X_{\zeta-})=0$.

\subsection{Collisions before explosion}
We fix again $N>3\theta\geq 6$.
We recall that $k_1=k_0-1$ and we show that there are infinitely many $k_1$-ary collisions 
just before explosion.
We know from the previous subsection that there exists $K_0 \subset \ig 1,N \id$ 
such that $|K_0| = k_0$ and $R_{K_0}(X_{\zeta-})=0$
and $\min_{i\notin K_0}R_{K_0\cup\{i\}}(X_{\zeta-})>0$. 
Then there is $\alpha>0$
such that during $[\zeta-\alpha,\zeta)$, the particles indexed in $K_0$ are far from the other ones,
so that $(X^i_t)_{i\in K_0}$ behaves like a $KS(\theta k_0/N,k_0)$-process by Subsection 
\ref{introgirsanov}. 
Observe now that $d_{\theta k_0/N,k_0}(k_0-1)=d_{\theta,N}(k_0-1)\in (0,2)$ thanks to 
Lemma~\ref{dimensions} and that $k_0 > \theta k_0/N$ because $N> \theta$.
We thus know from the special case (a) of Subsection~\ref{mainidea}
that $(X^i_t)_{i\in K_0}$ has infinitely many $(K_0\setminus\{i\})$-collisions 
just before
$\zeta$, for all $i \in K_0$.

\vip

When $k_2=k_1-1$, one can show in the very same way that for all $K$ with $|K|=k_1$,
for all $i \in K$, there are infinitely many $(K\setminus\{i\})$-collisions 
just before each $K$-collision.
We may also use Subsection~\ref{mainidea}-(a), since 
$d_{\theta k_1/N,k_1}(k_1-1)=d_{\theta,N}(k_2)\in (0,2)$,
see Lemma~\ref{dimensions}.

\subsection{Absence of other collisions}
We want to show that when $N>3\theta\geq 6$, 
for $K \subset \ig 1,N \id$ with $|K| \in \ig 3,k_2-1 \id$, there is no
$K$-collision during $ ( 0,\zeta)$. Suppose by contradiction that there is 
$K \subset \ig 1,N \id$ with $|K| \in \ig 3,k_2-1 \id$ and $t\in ( 0,\zeta)$ such that
$R_K(X_t)=0$ and for all $i \notin K$, $R_{K\cup\{i\}}(X_t)>0$. Then there is $\alpha>0$
such that during $[t-\alpha,t]$, the particles indexed in $K$ are far from the other ones,
so that $R_K(X_t)$ behaves like a squared Bessel process with dimension $d_{\theta |K|/N,|K|}(|K|)$,
see Subsection~\ref{girsanovbessel}.
Since $d_{\theta |K|/N,|K|}(|K|)=d_{\theta,N}(|K|)\geq 2$ because $|K| \in \ig 3,k_2-1 \id$, 
see Lemma~\ref{dimensions},
such a Bessel process cannot hit zero, whence a contradiction.

\subsection{Binary collisions}
We still assume that $N>3\theta\geq 6$,
we suppose that there is a $K$-collision for some $K\subset \ig 1,N \id$ such that $|K| = k_2$
at some time $t \in  ( 0,\zeta)$ and we want to show that there are infinitely many binary collisions
just before $t$. There is $\alpha>0$ such that the particles indexed in $K$ are far 
from all the other ones
during $[t-\alpha,t]$, so that Subsection~\ref{introgirsanov} tells us that 
 $(X^i_t)_{i\in K}$ behaves 
like a $KS(\theta k_2/N,k_2)$-process. We observe that $k_2\geq 5$, that 
$d_{\theta k_2/N,k_2}(k_2-1)=d_{\theta,N}(k_2-1)\geq 2$ and that $d_{\theta k_2/N,k_2}(k_2)=d_{\theta,N}(k_2)\in (0,2)$
by Lemma~\ref{dimensions}.

\vip

We are reduced to show that a 
$KS(\theta,N)$-process, that we still denote by $(X_t^i)_{i \in \ig 1,N \id,t\geq 0}$, 
such that $N\geq 5$, $d_{\theta,N}(N-1)\geq 2$ and $d_{\theta,N}(N) \in (0,2)$,
a.s. has infinitely many binary
collisions before the first instant $\tau_D$ of $\ig 1,N\id$-collision. 
Such a process does not explode, because $k_0>N$ (since $d_{\theta,N}(N)>0$), 
see Subsection~\ref{cmdp}.
Hence using \eqref{etoile} (which is licit since $d_{\theta,N}(N)<2$), 
we only have to show that e.g. $U^1$ collides infinitely often with $U^2$ during
$[0,\infty)$. 

\vip

First, one easily gets convinced that the probability that e.g. $X^1$ collides with $X^2$ before 
$\tau_D$ is positive, because the probability that 
all the particles are pairwise far from each other, except $X^1$ and 
$X^2$, during the time interval $[0,1]$,
is positive. On this kind of event, by Subsection~\ref{girsanovbessel},
$R_{\{1,2\}}(X_t)$ behaves like a squared Bessel process with dimension 
$d_{\theta,N}(2)\in (0,2)$ and
thus hits zero during $[0,1]$ (and thus before $\tau_D$) with positive probability.

\vip

Using again \eqref{etoile}, we conclude that the probability that $U^1$ collides 
with $U^2$ in finite time
is positive. Since now $U$ is positive recurrent, recall Subsection~\ref{dec} 
and that $k_0> N$ (because $d_{\theta,N}(N)>0$),
we conclude that $U^1$ collides infinitely often with $U^2$ during $[0,\infty)$ as desired.
 
\blue
\subsection{Non-integrability of the drift term}\label{nonint} 
Here we check that when $d_{\theta , N}(k_1)\in (0,1)$, the S.D.E. \eqref{EDS} cannot have a solution 
in the classical sense, because the drift term is not integrable in time. More precisely,
recall that there is some $K$-collision at some time $\tau$ strictly before explosion,
for some $K\subset \ig1,N\id$ with cardinal $k_1$. We now show that a.s., for $a>0$,
$$
\int_{\tau-a}^{\tau+a} \sum_{i=1}^N \Big\|\sum_{j\neq i} \frac{X^i_s-X^j_s}{||X^i_s-X^j_s||^2} \Big\| \dd s =\infty,
$$
which indeed shows the non-integrability of the drift term. 
Since $\tau$ is an instant of $K$-collision, there exists 
$a>0$ small enough so that
during $[\tau-a,\tau+a]\subset [0,\zeta)$, the particles labeled in $K$ are far from the 
particles labeled in $K^c$. 
It clearly suffices to show that $Z=\infty$ a.s., where
$$
Z=\int_{\tau-a}^{\tau+a} \sum_{i\in K} \Big\|\sum_{j\in K,j\neq i} \frac{X^i_s-X^j_s}{||X^i_s-X^j_s||^2} \Big\| \dd s.
$$
But
$$
Z=\int_{\tau-a}^{\tau+a} \frac{f(V_{s})}{\sqrt{R_K(X_s)}} \dd s, \quad \hbox{where}\quad 
V_s=(V^i_s)_{i\in K} \quad \hbox{is defined by} \quad V^i_s=\frac{X_s^i-S_{K}(X_s)}{\sqrt{R_K(X_s)}},
$$
so that $V_s$ a.s. belongs to $\sS_K=\{(v^i)_{i\in K} \in (\rr^2)^{|K|} : 
\sum_{i\in K}v^i=0,\;\sum_{i\in K}||v^i||^2=1\}$, and where
$$
f(v)=\sum_{i \in K}\Big\|\sum_{j \in K,j\neq i} \frac{v^i-v^j}{||v^i-v^j||^2} \Big\|
$$
for each $v \in \sS_K$. Since the invariant measure $\bm$ of $X$ satisfies $\bm(E_2^c)=0$, it a.s. holds true
that $X_s\in E_2$ for a.e.  $s\in [0,\zeta)$ (at least for a.e. initial condition),
so that a.s., $f(V_s)$ is well-defined for a.e. $s\in [0,\zeta)$.
We now show that $f$ is bounded from below on $\sS_K$. We have
$$
f(v) \geq \max_{i\in K} \Big\|\sum_{j\in K,j\neq i} \frac{v^i-v^j}{||v^i-v^j||^2}\Big\|\geq  
\sqrt{\frac 1 {|K|}\sum_{i\in K}\Big\|\sum_{j \in K,j\neq i} \frac{v^i-v^j}{||v^i-v^j||^2} \Big\|^2}.
$$
Using now the Cauchy-Schwarz inequality and the fact that $\sum_{i\in K}||v^i||^2=1$, we find that
$$
f(v) \geq \frac{1}{\sqrt {|K|}} \sum_{i\in K} \sum_{j\in K,j\neq i} \frac{v^i-v^j}{||v^i-v^j||^2}\cdot v^i
=\frac{1}{2\sqrt {|K|}} \sum_{i,j\in K, j\neq i} \frac{v^i-v^j}{||v^i-v^j||^2}\cdot (v^i-v^j)
=\frac{|K|(|K|-1)}{2\sqrt {|K|}}.
$$
To conclude that $Z=\infty$ a.s., it remains to verify that 
$\int_{\tau-a}^{\tau+a} [R_K(X_s)]^{-1/2}\dd s =\infty$ a.s. By Subsection~\ref{girsanovbessel}, $R_{K}(X)$ 
behaves like a squared Bessel process with dimension $d_{\theta , N}(k_1)$ during $[\tau-a,\tau+a]$.
Since $d_{\theta,N}(k_1)\in (0,1)$ and $R_K(X_\tau)=0$, we conclude that indeed,
$\int_{\tau-a}^{\tau+a} [R_K(X_s)]^{-1/2}\dd s =\infty$ a.s.: this can be
shown by comparison with the $1$-dimensional Brownian motion.
\bla

\section{Construction of the Keller-Segel particle system}\label{constr}

The aim of this section is to build a first version of the Keller-Segel particle system
using the book of Fukushima-Oshima-Takeda \cite{f}.
We also build a $\sS$-valued process for later use.

\begin{prop}\label{existenceXU}
We fix $N\geq 2$ and $\theta>0$ such that $N> \theta$,
recall that $k_0=\lceil 2N/\theta \rceil$ and that $\mu$ and $\beta$ were 
defined in \eqref{mmu} and \eqref{beta}.
We set $\cX=E_{k_0}$ and $\cX_\triangle=\cX \cup\{\triangle\}$, as well as
$\cU=\sS \cap E_{k_0}$ and $\cU_\triangle=\cU\cup\{\triangle\}$,
where $\triangle$ is a cemetery point.
\vip
(i) There exists a unique \blue diffusion \bla $\mathbb{X} = (\Omega^X,\cM^X,(X_t)_{t\geq 0},
(\PP_x^X)_{x\in \cXt})$ with values in $\cX_\triangle$, which is  
$\mu$-symmetric, with regular Dirichlet space $(\cE^X\!,\!\cF^X)$ on $L^2((\rr ^2)^N\!,\!\mu)$ with core 
$C^\infty_c(\cX)$ defined by
$$
\hbox{for all}\quad  \varphi \in C^\infty_c(\cX), \quad
\mathcal{E}^{X} (\varphi,\varphi) = \frac12\int_{(\rr^2)^N} \|\nabla  \varphi\|^{2} \dd \mu.
$$
We call such a process a $QKS(\theta ,N)$-process and denote by $\zeta=\inf\{t\geq 0 : X_t= \triangle\}$
its life-time.

\vip

(ii) There exists a unique \blue diffusion \bla $\mathbb{U} =(\Omega^U,\cM^U,(U_t)_{t\geq 0},
(\PP^U_u)_{u\in \cUt})$ with values in $\cUt$, which is  
$\beta$-symmetric,  with regular Dirichlet space $(\cE^U,\cF^U)$ on ${L^2(\sS,\beta)}$ with core 
$C^\infty_c(\cU)$ defined by
$$
\hbox{for all}\quad  \varphi \in C^\infty_c(\cU), \quad
\mathcal{E}^{U} (\varphi,\varphi) = \frac 12 \int_{\sS} \|\nabla_\sS  \varphi\|^{2} \dd \beta.
$$
We call such a process a   $QSKS(\theta ,N)$  -process and denote by $\xi=\inf\{t\geq 0 : U_t= \triangle\}$
its life-time.
\end{prop}

The proof that we can build a $KS(\theta,N)$-process, i.e. a $QKS(\theta,N)$-process 
such that $\PP_x^X\circ X_t^{-1}$ has density 
for all $x \in E_2$ and all $t>0$ will be handled in Section~\ref{rq}.

\vip
We refer to Subsection~\ref{ap1} for some explanations about the notions used in this proposition: 
link between a \blue diffusion (i.e. a continuous Hunt process), \bla 
its generator, semi-group and its Dirichlet space, definition of the
one-point compactification topology, i.e. the topology endowing $\cX_\triangle$ and $\cU_\triangle$,
and about the \blue {\it quasi-everywhere} \bla notion. The state $\triangle$ is absorbing, i.e.
$X_t=\triangle$ for all $t\geq \zeta$
and $U_t=\triangle$ for all $t\geq \xi$.

\begin{rk}\label{rklt}
By definition of the one-point compactification topology, for any increasing sequence of compact subsets
$(\cK_n)_{n\geq 1}$ of $\cX$ such that
$\cup_{n\geq 1} \cK_n=\cX$,  $\zeta=\lim_{n\to \infty} \inf\{t\geq 0 : X_t\notin  \cK_n\}$.
\vip
Similarly, for any increasing sequence of compact subsets
$(\cL_n)_{n\geq 1}$ of $\cU$ such that $\cup_{n\geq 1} \cL_n=\cU$,
 $\xi=\lim_{n\to \infty} \inf\{t\geq 0 : U_t\notin  \cL_n\}$.
\end{rk}

\vip

The uniqueness stated e.g. in Proposition~\ref{existenceXU}-(i) has to be understood in the 
following sense, see \cite[Theorem 4.2.8 p 167]{f}: if we have another \blue diffusion \bla  
$\yY=(\Omega^Y,\cM^Y,(Y_t)_{t\geq 0},
(\PP_x^Y)_{x\in \cX})$ enjoying the same properties, then \blue quasi-everywhere, \bla
the law of $(Y_t)_{t\geq 0}$ under $\PP_x^Y$
equals the law of $(X_t)_{t\geq 0}$ under $\PP_x^X$.
The \blue quasi-everywhere \bla 
notion depends on the Hunt process under consideration but, as recalled in Subsection
\ref{ap1}, two Hunt processes with the same Dirichlet space share the same 
\blue quasi-everywhere \bla  notion.

\vip

\begin{proof}[Proof of Proposition~\ref{existenceXU}]
We start with (i). We consider the bilinear form $\cE^X$ on $C^\infty_c(\cX)$ defined by
$\cE^X(\varphi,\varphi)\!=\!\blue \frac12 \bla \int_{(\rr^2)^N}\! ||\nabla \varphi||^2 \dd \mu$. It is well-defined, since
$\mu$ is Radon on $\cX=E_{k_0}$ by Proposition~\ref{radon}.

\vip

We first show that it is closable, see \cite[page 2]{f}, i.e. that if 
$(\varphi_n)_{n\geq 1} \subset C^\infty_c(\cX)$ is
such that $\lim_n \varphi_n= 0$ in $L^2((\rr^2)^N,\mu)$ and 
$\lim_{n,m}\cE^X(\varphi_n-\varphi_m,\varphi_n-\varphi_m)=0$, then 
$\lim_{n}\cE^X(\varphi_n,\varphi_n)=0$: since $\nabla \varphi_n$ is a Cauchy sequence in 
$L^2((\rr^2)^N,\mu)$, it converges
to a limit $g$ and it suffices to prove that $g=0$  a.e. For   $\psi \in C^\infty_c(E_2, (\rr^2)^N)$,  
we have $\int_{(\rr^2)^N} g \cdot \psi \dd \mu=\lim_n\int_{(\rr^2)^N} \nabla \varphi_n \cdot \psi \dd \mu$. But,
recalling \eqref{mmu},
$$
\int_{(\rr^2)^N} \nabla \varphi_n \cdot \psi \dd \mu= 
\int_{(\rr^2)^N} \nabla \varphi_n(x) \cdot\psi(x) \bm(x)\dd x=- \int_{(\rr^2)^N} \varphi_n(x) 
\ddiv (\bm(x) \psi(x))\dd x.
$$
  Thus by the Cauchy-Schwarz inequality,
$$
\Big|\int_{(\rr^2)^N} \nabla \varphi_n \cdot \psi \dd \mu\Big| \leq 
\Big(\int_{(\rr^2)^N} \varphi_n^2 \dd \mu\Big)^{1/2}\Big(\int_{(\rr^2)^N} \frac{|\ddiv(\bm(x)\psi(x) )|^2}
{\bm(x)} \dd x \Big)^{1/2},
$$
which tends to $0$ since $\lim_n \varphi_n= 0$ in $L^2((\rr^2)^N,\mu)$, since $\psi \in C^\infty_c(E_2,(\rr^2)^N)$
and since $\bm$ is smooth and positive on $E_2$. 
Thus $\int_{(\rr^2)^N} g \cdot \psi \dd \mu=0$ for all 
$\psi \in C^\infty_c(E_2, (\rr^2)^N)$, so that $g=0$   a.e.  

\vip

We can thus consider the extension of $\cE^X$ to $\cF^X=
\overline{C_{c}^{\infty }( \cX)}^{\cE^{X}_{1}}$, where we have set 
$\cE^X_1(\varphi,\varphi)= \int_{(\rr^2)^N} (\varphi^2+\frac12||\nabla \varphi||^2) \dd \mu$ for
 $\varphi \in C^\infty_c(\cX)$. 

\vip

Next, $(\cE^X,\cF^X)$ is obviously regular with core $C^\infty_c(\cX)$, see \cite[page 6]{f}, 
because $C^\infty_c(\cX)$ is dense
in $\cF^X$ for the norm associated to $\cE_1^X$ by definition of $\cF^X$ and $C^\infty_c(\cX)$ is dense,
for the uniform norm,
in $C_c(\cX)$.  It is also strongly local, see \cite[page 6]{f}, i.e. $\cE^X(\varphi,\psi)=
\blue \frac12 \bla \int_{(\rr^2)^N} 
\nabla \varphi \cdot \nabla \psi \dd \mu=0$ if $\varphi,\psi \in C^\infty_c(\cX)$
and   if   $\varphi$ is constant on a neighborhood of ${\rm Supp}\; \psi$.

\vip

Then \cite[Theorems 7.2.2 page 380 and 4.2.8 page 167]{f} imply the existence and uniqueness of
a Hunt process $\mathbb{X} = (\Omega^X,\cM^X,(X_t)_{t\geq 0},
(\PP_x^X)_{x\in \cXt})$ with values in $\cXt$, which is 
$\mu$-symmetric, of which the Dirichlet space is $(\cE ^X,\cF ^X)$,
and such that $t\mapsto X_t$ is $\PP^X_x$-a.s. continuous on $[0,\zeta)$ for all $x\in \cX$,
where $\zeta=\inf\{t\geq 0 : X_t = \triangle\}$.

\vip

Furthermore, since $\cE^X$ is strongly local, we know from \cite[Theorem 4.5.3 page 186]{f} that we can choose 
$\xX$ \blue (modifying $\PP^X_x$ only on a properly exceptional set) \bla
such that $\PP_x(\zeta<\infty, X_{\zeta-} = \triangle)=1$
for all $x \in \cX$. This implies that for all $x\in\cX$, $\PP_x$-a.s., the map 
$t\mapsto X_t$ is continuous from $[0,\infty)$ to $\cX_\triangle$, endowed with the 
one-point compactification topology on $\cXt$ recalled in Subsection~\ref{ap1}.
\blue Hence $\xX$ is a diffusion. \bla

\vip

For (ii), the very same strategy applies. The only difference is the integration
by parts to be used for the closability: for $\varphi \in C^1_c(\cU)$
and   $\psi \in C^1_c(\sS \cap E_2, (\rr^2)^N)$, it classically holds that
\begin{equation}\label{ippu}
\int_\sS (\nabla_\sS \varphi )\cdot \psi \dd \beta=
\int_\sS (\nabla_\sS \varphi (u))\cdot \psi(u) \bm(u) \sigma (\dd u) 
= - \int_\sS \varphi(u) \ddiv_\sS (\bm(u)\psi(u) )  \sigma (\dd u).
\end{equation}
This can be shown naively using Lemma~\ref{changeofvariable}.
\end{proof}

We now make explicit the generators of $\xX$ and $\uU$ when applied to some functions 
enjoying a few properties. See Subsection~\ref{ap1}
for a precise definition of the generator of a Hunt process. We have to introduce
a few notation.
\vip
For $\varphi\in C^\infty((\rr^2)^N)$, $\alpha \in (0,1]$ and $x\in (\rr^2)^N$,
we set
\begin{equation}\label{dflxa}
\cL^X_\alpha \varphi(x) = \frac12 \Delta\varphi(x) - \frac \theta N  \sum_{1\leq i\neq j \leq N} 
\frac{x^i-x^j}
{\|x^i-x^j\|^2+\alpha}\cdot (\nabla\varphi(x))^i= \frac1{2\bm_\alpha(x)}\ddiv 
[\bm_\alpha(x)  \nabla  \varphi(x)],
\end{equation}
where
$$
\bm_\alpha(x)=\prod _{1\le i\neq j \le N} (\|x^{i}-x^{j}\|^2+\alpha)^{-\theta/(2N)}.
$$
\blue This is in accordance with \eqref{mmu}, in the sense that $\bm_0=\bm$.
The formula \eqref{dflxa} makes sense for $x\in E_2$ when $\alpha=0$  (with $\bm_\alpha$ replaced by $\bm$)
and we recall
that for $\varphi \in C^\infty((\rr^2)^N)$ and $x \in E_2$, 
$\cL^X \varphi(x)$ was defined in \eqref{lx} by $\cL^X\varphi(x)=\cL^X_0\varphi(x)$.
We will often use that 
for all $\varphi,\psi \in C^\infty ((\rr^2)^N)$, all $x\in (\rr^2)^N$,  all
$\alpha \in (0,1]$,
\begin{align} \label{lproduit}
\cL_\alpha^X (\varphi\psi)(x) = \varphi(x) \cL_\alpha^X\psi(x) + \psi(x) \cL_\alpha^X\varphi(x) + 
\nabla \varphi(x) \cdot \nabla \psi(x).
\end{align}
\bla

For $\varphi\in C^\infty(\sS)$, $\alpha \in (0,1]$ and $u\in \sS$,
we set  
\begin{equation}\label{dflua}
\cL^U_\alpha \varphi(u) = \frac12 \Delta_\sS\varphi(u) - \frac \theta N  \sum_{1\leq i\neq j \leq N} 
\frac{u^i-u^j}
{\|u^i-u^j\|^2+\alpha}\cdot (\nabla_\sS\varphi(u))^i= \frac1{2\bm_\alpha(u)}\ddiv _\sS
[\bm_\alpha(u)\nabla_\sS \varphi(u)].
\end{equation}  
This formula makes sense for $u\in \sS\cap E_2$ when $\alpha=0$ (with $\bm_\alpha$ replaced by $\bm$)
and we set,
for $\varphi \in C^\infty(\sS)$ and $u \in \sS\cap E_2$, 
$\cL^U\varphi(u)=\cL^U_0\varphi(u)$.

\begin{rk}\label{caracdomaine}
(i) Denote by $(\cA^X,\cD_{A^X})$ the generator of the process $\xX$ of 
Proposition~\ref{existenceXU}-(i).
If $\varphi \in C^\infty_c(\cX)$ 
satisfies $\sup_{\alpha\in (0,1]}\sup_{x\in (\rr^2)^N} |\cL^X_\alpha \varphi(x)|<\infty$,
then $\varphi \in \cD_{A^X}$ and $\cA^X \varphi = \cL^X \varphi$.

\vip
(ii) Denote by $(\cA^U,\cD_{A^U})$ the generator of the process $\mathbb{U}$ of Proposition 
\ref{existenceXU}-(ii).
If $\varphi \in C^\infty_c(\cU)$ satisfies $\sup_{\alpha\in (0,1]}\sup_{u\in \sS} 
|\cL^U_\alpha \varphi(u)|<\infty$, then $\varphi \in \cD_{A^U}$
and $\cA^U \varphi = \cL^U \varphi$.
\end{rk}

\begin{proof}
To check (i), it suffices by \eqref{caracdomaine0} to verify that 
(a) $\varphi \in \cF^X$, (b) $\cL^X \varphi \in L^2(\cX,\mu)$ and (c)  for all
$\psi \in \cF^X$, we have $\cE^X (\varphi,\psi) = -\int_\cX (\cL^X \varphi)\psi\dd \mu$. 

\vip
Point (a) is clear, since $\varphi \in C^\infty_c(\cX)$. Point (b) follows from the facts that 
$\mu$ is Radon on $\cX$, that $\varphi$ is compactly supported in $\cX$ and that
$\cL^X \varphi \in L^\infty((\rr^2)^N,\dd x)$, because for all $x\in E_2$,
$\cL^X\varphi(x)=\lim_{\alpha\to 0} \cL_\alpha^X\varphi(x)$. Concerning (c) it suffices, by 
definition of $(\cE^X,\cF^X)$ and since $\cL^X \varphi \in L^2(\cX,\mu)$,
to show that for all $\psi \in C^\infty_c(\cX)$, we have $\frac12\int_{(\rr^2)^N} \nabla \varphi\cdot\nabla \psi
\dd \mu = - \int_{(\rr^2)^N} (\cL^X \varphi)\psi \dd \mu$. But for $\alpha \in (0,1]$, by a standard integration
by parts, since $\varphi,\psi$ and $\bm_\alpha$ are smooth,
\begin{align*}
\frac12\int_{(\rr^2)^N}\nabla \varphi(x)\cdot\nabla \psi(x) \bm_\alpha (x)\dd x
=&-\frac12\int_{(\rr^2)^N}\ddiv(\bm_\alpha(x)\nabla\varphi(x)) \psi(x) \dd x\\
=&-\int_{(\rr^2)^N}  [\cL^X_\alpha \varphi(x)]  \psi(x) \bm_\alpha(x)\dd x.
\end{align*}
We conclude letting $\alpha\to 0$ by dominated convergence, since
$\bm_\alpha\to\bm$ and $\cL_\alpha^X \varphi\to\cL^X\varphi$ a.e., since by assumption, 
$|\nabla \varphi(x)\cdot\nabla \psi(x) \bm_\alpha (x)|+| [\cL^X_\alpha \varphi(x)] \psi(x) \bm_\alpha(x)|
\leq C \indiq_{\{x\in \cK\}} \bm(x)$ for some constant $C$ and for $\cK=$Supp $\psi$ which is compact
in $\cX$, and since $\mu(\cK)=\int_\cK \bm(x)\dd x<\infty$.

\vip
 
The proof of (ii) is exactly the same, using that
if $\varphi,\psi \in C^\infty(\sS)$, it holds that
$$
\frac12\int_{\sS}\nabla_\sS \varphi \cdot\nabla_\sS \psi \; \bm_\alpha \dd \sigma
=-\frac12\int_{\sS}   \ddiv_\sS  (\bm_\alpha   \nabla_\sS   \varphi) \psi \dd \sigma
=-\int_{\sS}  [\cL^U_\alpha \varphi]  \psi \bm_\alpha\dd \sigma,
$$
which can be shown naively using the projection $\Phi_\sS$, see \eqref{phis}, and Lemma~\ref{changeofvariable}.
\end{proof}

We end the section with a quick irreducibility/recurrence/transience study of the spherical process,
see Subsection~\ref{ap1} again for definitions.

\begin{lemma}\label{visite}
We fix $N\geq 2$ and $\theta>0$ such that $N>\theta$ and consider the process $\mathbb{U}$ and its 
Dirichlet space $(\cE^U,\cF^U)$ as in Proposition
\ref{existenceXU}-(ii).

\vip

(i) $(\cE^U,\cF^U)$ is irreducible and we have the alternative:
\vip
$\bullet$ either $(\cE^U,\cF^U)$ is recurrent and in particular 
it is non-exploding and for all measurable $A\subset \cU$ such that 
$\beta(A)>0$, \blue $\PP^U_u(\limsup_{t\to \infty}\{U_t\in A\})=1$ quasi-everywhere; \bla
\vip
$\bullet$ or $(\cE^U,\cF^U)$ is transient and in particular for all compact set $\cK$ of $\cU$, 
we have
\blue quasi-everywhere \bla
$\PP^U_u (\liminf_{t\to\infty}\{U_t \in \cK\}) = 0$.
\vip
(ii) If $d_{\theta ,N}(N-1)> 0$, then $(\cE^U, \cF^U)$ is recurrent.
\end{lemma}

In the transient case, one might also prove that 
$\PP^U_u(\limsup_{t\to \infty}\{U_t\in \cK\})=0$, but this would be useless for our purpose.

\begin{proof}
We start with (i). 
We first show that in any case, $(\cE^U,\cF^U)$ is irreducible. By \cite[Corollary 4.6.4 page 195]{f}
and since $\cE^U(\varphi,\varphi)=\blue \frac12 \bla 
\int_{\sS} \|\nabla _\sS \varphi  \|^2 \bm \dd \sigma$ with $\bm$ bounded 
from below by a constant
(on $\sS$), it suffices to prove that the $\sigma$-symmetric Hunt process with regular 
Dirichlet space 
$(\cE,\cF)$ on $L^2(\cU, \sigma)$ with core $C^\infty _c ( \cU)$ such that for all 
$\varphi\in C^\infty _c (\cU)$,
$\cE (\varphi,\varphi) = \blue \frac12 \bla \int _{\sS} \|\nabla _\sS \varphi  \|^2 \dd \sigma$
is irreducible. But this Hunt process is nothing but a $\sS$-valued Brownian motion. 
This Brownian motion is {\it a priori} killed when it gets out of $\cU$, 
but this does a.s. never occur since such a Brownian motion never has two (bi-dimensional)
coordinates equal. This $\sS$-valued Brownian motion is of course irreducible.
We conclude from \cite[Lemma 1.6.4 page 55]{f} that $(\cE^U,\cF^U)$ is either recurrent or transient. 
 
\vip
 
$\bullet$ When $(\cE^U,\cF^U)$ is recurrent, \cite[Theorem 4.7.1-(iii) page 202]{f} gives us the result.
 
\vip

$\bullet$ When $(\cE^U,\cF^U)$ is transient, we fix a compact set $\cK$ of $\cU$ and we know from 
Lemma~\ref{radonsphere} that $\beta (\cK ) < \infty$, so that by definition of transience,
for $\beta$-a.e $u\in \cU$,
$\E^U_u [\int_0^\infty \indiq_\cK (U_s)\dd s ] < \infty$.
Setting $\tau_{\cK^c} = \inf \{ t\ge 0 : U_t \notin \cK \}$, we get in particular that for $\beta$-a.e 
$u\in \cU$, $\PP^U_u ( \tau_{\cK^c} < \infty ) =1$.
But, by \cite[(4.1.9) page 155]{f}, $u\mapsto \PP^U_u (\tau_{\cK^c} < \infty)$ is finely continuous.
Using \cite[Lemma 4.1.5 page 155]{f}, we deduce that 
$\PP^U_u (\tau_{\cK^c} < \infty) =1$ \blue quasi-everywhere. \bla The Markov property allows us to conclude.

\vip
Concerning (ii), we recall from Proposition~\ref{radonsphere} that $\beta (\sS) < \infty$, because
$d_{\theta, N}(N-1) >0$ implies that $k_0\geq N$, see Lemma~\ref{dimensions}.
Moreover, $k_0\geq N$ implies that
$E_{k_0}\supset E_N \supset \sS$, whence $\cU=E_{k_0}\cap\sS=\sS$ is compact:
the process cannot explode, i.e. $\xi = \infty$. 
Consequently, $(\cE^U,\cF^U)$ is recurrent, since $\varphi\equiv 1$ belongs to $L^1(\cU,\beta)$
and since  $\E_u^U[\int_0^\infty \varphi(U_s) \dd s]=\E_u^U[\xi]=\infty$. Indeed, 
as recalled Subsection~\ref{ap1}, if $(\cE^U,\cF^U)$ was transient, we would have
$\E_u^U[\int_0^\infty \varphi(U_s) \dd s]<\infty$ for all $\varphi \in L^1(\cU,\beta)$, with 
the convention that $\varphi(\triangle)=0$.
\end{proof}

\section{Decomposition}\label{deccc}

The goal of this section is to prove the following decomposition of the Keller-Segel particle system
defined in Proposition~\ref{existenceXU}-(i). This 
decomposition is noticeable and crucial for our purpose. 

\begin{prop}\label{deco}
We fix $N\geq 2$ and $\theta>0$ such that $N > \theta$, and we recall that
$k_0=\lceil 2N/\theta\rceil$, that $\cX=E_{k_0}$ and that $\cU=\sS\cap E_{k_0}$.

\vip

For $x \in E_N$, we set $r=R_{\ig 1,N\id}(x)>0$, $z=S_{\ig 1,N\id}(x)\in \rr^2$ and
$u=(x-\gamma(z))/\sqrt{r} \in \sS$ and we consider three independent processes:
\vip
\noindent $\bullet$  $(M_t)_{t\ge 0}$, a $2$-dimensional Brownian motion with diffusion constant $N^{-1/2}$ 
starting from $z$,
\vip
\noindent $\bullet$ $(D_t)_{t\ge 0}$ a squared Bessel process with dimension $d_{\theta ,N}(N)$ starting from $r$
and killed when it gets out of $(0,\infty)$, with life-time $\tau_D = \inf \{ t\ge 0 : D_t =\triangle \}$,
\vip
\noindent $\bullet$ $(U_t)_{t\ge 0}$, a  $QSKS(\theta,N)$  -process starting from $u$,
with life-time $\xi= \inf\{ t\ge 0 : U_t =\triangle \} $.
\vip
We introduce  $A_t= \int _0 ^{t\land \tau_D} D_s^{-1} \dd s$,
and its generalized inverse $\rho_t=\inf\{s>0 : A_s>t\}$.
We define $Y_t = \Psi(M_t,D_t,U_{A_t})$, where we recall from \eqref{Psi} that 
$\Psi(z,r,u)=\gamma(z)+\sqrt{r}u\in E_N$ when $(z,r,u)\in \rr^2\times (0,\infty)
\times \sS$ and where we set $\Psi(z,r,u)=\triangle$ when $r=\triangle$ or $u=\triangle$.
Observe that the life-time of $Y$ equals $\zeta'=\rho_\xi\land \tau_D$.
\vip
Consider also a $QKS(\theta,N)$-process $\xX=(\Omega^X,\cM^X,(X_t)_{t\geq 0},(\PP_x^X)_{x\in \cXt})$,
with life-time $\zeta$, and
$\xXs=(\Omega^X,\cM^X,(\Xs_t)_{t\geq 0},(\PP_x^X)_{x\in  (\cX \cap E_{N}) \cup \{\triangle\}})$, 
where $\Xs_t=X_t\indiq_{\{t<\tau\}}+\triangle\indiq_{\{t\geq \tau\}}$ and 
where $\tau=\inf \{t\ge 0 : R_{\ig 1,N \id }(X_t) \notin (0,\infty)\}$. 
In other words, $\xXs$ is the version of $\xX$ killed when it
gets out of $E_{N}$. The life-time of $\xXs$ is $\tau$.

\vip

The law of $(Y_t)_{t \geq 0}$ is the same as that of
$(\Xs_t)_{t\geq 0}$ under $\PP_x^X$, \blue quasi-everywhere in $\cX\cap E_N$. \bla
\end{prop}

We take the convention that $R_{\ig 1,N \id }(\triangle)=0$, so that $\tau \in [0,\zeta]$.
  Since $R_{\ig1,N\id}(Y_t)=D_t$ and $S_{\ig1,N\id}(Y_t)=M_t$ for all  $t \in [0,\zeta')$, Proposition~\ref{deco}
in particular implies that $(R_{\ig1,N\id}(X_t))_{t\ge 0}$ and $(S_{\ig1,N\id}(X_t))_{t\ge 0}$ are some independent 
squared Bessel process and  Brownian motion until the first time $(R_{\ig1,N\id}(X_t))_{t\ge 0}$ vanishes. 
This actually holds true until explosion, as shown in Lemma~\ref{besbro} below.
The \blue quasi-everywhere \bla notion refers to the Hunt process $\mathbb{X}$.
Observe that when $\theta \geq 2$, we have $k_0 \leq N$, so that $\cX \cap E_N = \cX$ and $\xX=\xX^*$.  

\begin{proof}
We slice the proof in several steps. 
\blue The two first steps are more or less classical, even if we give all the details:
we determine the Dirichlet spaces
of the three processes $(M_t)_{t\geq 0}$, $(D_t)_{t\geq 0}$ and $(U_t)_{t\geq 0}$ involved in the 
construction of $(Y_t)_{t\geq 0}$; then we compute the Dirichlet space of $(D_{\rho_t})_{t\geq 0}$;
we next identify the Dirichlet space of $(D_{\rho_t},U_t)_{t\geq 0}$, which allows us to find the one
of $(D_{t},U_{A_t})_{t\geq 0}$ by a second time-change; by concatenation, we deduce the Dirichlet space of
$(M_t,D_{t},U_{A_t})_{t\geq 0}$. The main computations are handled in Steps 3 and 4, 
where we find the Dirichlet space of
$(Y_t)_{t\geq 0}$, which allows us to conclude in Step 5 by uniqueness.\bla
\vip

{\it Step 1.}
First, take 
$\mathbb{U}=(\Omega^U,\cM^U,(U_t)_{t\geq 0},
(\PP^U_u)_{u\in \cUt})$ as in Proposition~\ref{existenceXU}-(ii). 

\vip

Second, consider a 
$2$-dimensional Brownian motion 
$\mathbb{M} = (\Omega ^M,\cM^M,(M_t)_{t\ge 0}, (\PP^M_z)_{z\in \mathbb{R}^2 })$ with diffusion 
constant $N^{-1/2}$.
We know from \cite[Example 4.2.1 page 167]{f} that $\mathbb{M}$ is a $\dd z$-symmetric
(here $\dd z$ is the Lebesgue measure on $\rr^2$) \blue diffusion \bla with regular Dirichlet space 
$(\cE^M,\cF^M)$ on $L^2(\rr ^2, \dd z)$ with core $C_{c}^{\infty}(\rr^{2})$ and for all 
$\varphi\in C_{c}^{\infty}(\rr^{2})$,
\begin{equation}\label{jabM}
\mathcal{E}^{M}(\varphi,\varphi) = \frac{1}{2N} \int _{\rr^{2}} \|\nabla _z \varphi(z)\|^{2}\dd z. 
\end{equation}

Finally, let $\mathbb{D} = (\Omega^D,\cM^D,(D_t)_{t\ge 0}, (\PP^D_r)_{r\in \rr_+^*\cup \{\triangle\}})$ 
be a squared Bessel process of dimension $d_{\theta ,N}(N)$ killed when it gets out of $\rr_+^*=(0,\infty)$
and set $\nu = d_{\theta ,N}(N)/2-1$, see Revuz-Yor \cite[page 443]{ry}.
Fukushima \cite[Theorem 3.3]{mf} tells us that $\mathbb{D}$ is a $r^\nu \dd r$-symmetric \blue diffusion \bla 
(here $\dd r$ is the Lebesgue measure on $\rr_+^*$) 
with regular Dirichlet space $(\cE ^D,\cF ^D)$ on $L^2(\rr _+, r^\nu \dd r )$ with core 
$C_{c}^{\infty}(\rr_{+}^{*})$ where for all $ \varphi\in C_{c}^{\infty}(\rr_{+}^{*})$,  
\begin{equation}\label{jabD}
\mathcal{E}^{D}(\varphi,\varphi) =2\int _{\rr_{+}} |\varphi'(r)|^{2}r^{\nu +1}\dd r.
\end{equation}
Together with  \cite[Theorem 3.3]{mf}, this uses that the scale function and the speed measure 
of $(D_t)_{t\ge 0}$ are  respectively $r \mapsto r^{-\nu} $ and $-[r^\nu /(2\nu)]\dd r$.
Actually, we don't take the speed measure as reference measure but $r^\nu \dd r$ 
which is the same up to a constant.

\vip

{\it Step 2.} We apply Lemma~\ref{chgmttemps} to $\mathbb{D}$ with $g(r)=1/r$, i.e.
with $A_t=\int_0^t D_s^{-1}\dd s= \int_0^{t\land \tau_D} D_s^{-1}\dd s$ thanks to the convention $\triangle^{-1}=0$
and recall that $\rho$ is its generalized inverse: we find that setting 
$D_{\rho_t}=D_{\rho_t}\indiq_{\{\rho_t<\infty\}}+\triangle \indiq_{\{\rho_t=\infty\}}$,
$$
\mathbb{D}_\rho= (\Omega^D, \cM^D, (D_{\rho _t})_{t\ge 0}, (\PP^D_r)_{r\in \rr ^* _+})
$$ 
is a $r^{\nu -1}\dd r$-symmetric $(\rr ^* _+ \cup \{ \triangle \})$-valued \blue diffusion \bla
with regular Dirichlet space $(\cE^{D_\rho}, \cF ^{D_\rho})$ on $L^2(\rr _+, r^{\nu -1}\dd r )$ 
with core $C_c^\infty(\rr^*_+)$ such that for all $\varphi \in C_c^\infty(\rr^*_+)$,
\begin{equation}\label{jabD2}
\mathcal{E}^{D_\rho}(\varphi,\varphi)  = \mathcal{E}^{D}(\varphi,\varphi) 
=2\int _{\rr_{+}} |\varphi'(r)|^{2}r^{\nu +1}\dd r = 2\int _{\rr_{+}} |r\varphi'(r)|^{2}r^{\nu -1}\dd r.
\end{equation}

We use Lemma~\ref{concatenation} and the notation therein: 
recalling that $\cM^{(D,U)}=\sigma((D_{\rho_t},U_t) : t\geq 0)$, with the convention that
$(r,\triangle)=(\triangle,u)=(\triangle,\triangle)=\triangle$,
and 
that $\PP^{(D,U)}_{(r,u)}=\PP^D_r\otimes\PP^U_u$ if $(r,u) \in \mathbb{R}^*_+ \times \cU$
and $\PP^{(D,U)}_\triangle=\PP^D_\triangle\otimes\PP^U_\triangle$,  it holds that
$$
(\mathbb{D_\rho}, \mathbb{U}) = \Big( \Omega^D \times \Omega^U,\cM^{(D,U)},
(D_{\rho_t },U_t)_{t\ge 0}, (\PP^{(D,U)}_{(r,u)})_{(r,u)\in (\mathbb{R}^*_+ \times \cU) \cup \{ \triangle \} }\Big)
$$ 
is a $r^{\nu -1}\dd r \beta(\dd u)$-symmetric $(\mathbb{R}^*_+ \times \cU)\cup\{\triangle\}$-valued 
\blue diffusion \bla with regular Dirichlet space given by 
$(\mathcal{E}^{(D_\rho ,U)}, \mathcal{F}^{(D_\rho ,U)})$ on $L^2(\rr _+ \times \sS , r^{\nu -1}\dd r \beta(\dd u))$ 
with core 
$C_{c}^{\infty }(\rr_{+}^{*} \times \cU)$, and for all $\varphi \in C_{c}^{\infty }(\rr_{+}^{*} \times \cU)$, 
$$
\mathcal{E}^{(D_\rho ,U)} (\varphi,\varphi) 
= \int _{\rr_{+}} \mathcal{E}^{U}(\varphi(r,\cdot), \varphi(r,\cdot)) r^{\nu -1 }\dd r 
+ \int _{\sS} \mathcal{E}^{D_\rho}(\varphi(\cdot,u), \varphi(\cdot,u)) \beta (\dd u).
$$

We now apply Lemma~\ref{chgmttemps} to $(\mathbb{D_\rho}, \mathbb{U})$ with 
$g(r,u)=r$ for all $r\in \rr_+^*$ and all $u\in \cU$.
We consider the time-change $\alpha_t=\int_0^t g(D_{\rho_s},U_s)\dd s$, with the convention that 
$g(r,u)=0$ as soon as $(r,u)=\triangle$. We also set
$B_t=\inf\{s>0 : \alpha_s>t\}$. As we will see in a few lines, it holds that
\begin{equation}\label{claim}
(D_{\rho_{B_t}},U_{B_t})=(D_t,U_{A_t}) \qquad \hbox{for all $t\geq 0$}.
\end{equation}
Hence Lemma~\ref{chgmttemps} tells us that
$$
(\mathbb{D}, \mathbb{U}_A)=
\Big( \Omega^D \times \Omega^U,\cM^{(D,U)},
(D_{t},U_{A_t})_{t\ge 0}, (\PP^{(D,U)}_{(r,u)})_{(r,u)\in (\mathbb{R}^*_+  \times \cU)\cup \{ \triangle \}} \Big)
$$ 
is a $r^{\nu}\dd r \beta(\dd u)$-symmetric $(\mathbb{R}^*_+ \times \cU)\cup\{\triangle\}$-valued 
\blue diffusion \bla with Dirichlet space
$(\cE^{(D,U_A)},\cF^{(D,U_A)} )$ on $L^2(\rr_+ \times \sS , r^{\nu }\dd r \beta (\dd u))$, regular with core 
$ C_c^\infty (\rr^*_+ \times \cU)$ and for all $\varphi\in C_c^\infty (\rr^*_+ \times \cU)$, 
\begin{align}\label{tyty}
\cE ^{(D,U_A)}(\varphi,\varphi) =  \mathcal{E}^{(D_\rho ,U)} (\varphi,\varphi) 
= \int _{\rr_{+}}\!\! \mathcal{E}^{U}(\varphi(r,\cdot), \varphi(r,\cdot)) r^{\nu -1 }\dd r 
+ \int _{\sS }\!\! \mathcal{E}^{D_\rho}(\varphi(\cdot,u), \varphi(\cdot,u)) \beta (\dd u) .
\end{align}

We now check the claim \eqref{claim}. Recall that $D$ explodes at time $\tau_D$, that
$A_t=\int_0^{t\land \tau_D} D_s^{-1} \dd s$ and that $\rho$ is the generalized inverse of $A$. 
Hence $(\rho_t)_{t\in [0,A_{\tau_D})}$ is the true inverse of 
$(A_t)_{t\in [0,\tau_D)}$ and we have $\rho_t'=D_{\rho_t}$, whence
$\rho_t=\intot D_{\rho_s}\dd s $ for $t \in [0,A_{\tau_D})$. We also have $\rho_t=\infty$ for $t\geq A_{\tau_D}$.
Next, $\alpha_t=\intot D_{\rho_s}\dd s=\rho_t$ for $t\in [0,A_{\tau_D}\land \xi)$, because $g(D_{\rho_s},U_s)=D_{\rho_s}$
if $(D_{\rho_s},U_s)\neq \triangle$, i.e. if $s<A_{\tau_D}\land \xi$. Hence $B$, the generalized inverse
of $\alpha$, equals $A$ during $[0,\tau_D\land \rho_\xi)$, thus in particular $\rho_{B_t}=t$ for 
$t\in[0,A_{\tau_D}\land \xi)$. As conclusion, \eqref{claim} holds true for $t\in [0,A_{\tau_D}\land \xi)$.
If now $t\geq \tau_D\land \rho_\xi$, then $B_t=\infty$,
because $B$ is the generalized inverse of $\alpha$ and because for all $t\geq 0$,
$$
\alpha_t \leq \alpha_{A_{\tau_D}\land \xi}
=\rho_{A_{\tau_D}\land \xi}=\tau_D\land \rho_\xi.
$$
Hence, still if $t\geq \tau_D\land \rho_\xi$, we have
$(D_{\rho_{B_t}},U_{B_t})=\triangle$, while 
$(D_{t},U_{A_t})=\triangle$ because either $t\geq \tau_D$ and thus $D_t=\triangle$ or
$t\geq \rho_\xi$ and thus $A_t\geq \xi$ so that $U_{A_t}=\triangle$. We have proved \eqref{claim}.

\vip

We finally conclude, thanks to Lemma~\ref{concatenation} again,
setting $\cM^{(M,D,U)}=\sigma((M_t,D_t,U_{A_t}) : t\geq 0)$ with the convention that 
$(z,\triangle)=\triangle$ and setting $\PP_{(z,r,u)}^{(M,D,U)}=\PP^M_z\otimes \PP^{(D,U)}_{(r,u)}$
in the case where ${\blue (z,r,u)\bla \in \mathbb{R}^2 \times \mathbb{R}^*_+ \times \cU}$ and
$\PP_{\triangle}^{(M,D,U)}=\PP^M_\triangle \otimes \PP^{(D,U)}_{\triangle}$,
that 
\begin{align*}
(\mathbb{M}, \mathbb{D}, \mathbb{U_A})=\Big(\Omega^M \times \Omega ^D \times \Omega ^U,\cM^{(M,D,U)},
(M_t,D_t,U_{A_t})_{t\ge 0}, 
 (\PP_{(z,r,u)}^{(M,D,U)} )_{(z,r,u) \in (\mathbb{R}^2 \times \mathbb{R}^*_+ \times \cU )\cup \{\triangle\}}\Big)
\end{align*}
is a $\dd z r^\nu \dd r \beta(\dd u)$-symmetric 
$(\mathbb{R}^2 \times \mathbb{R}^*_+ \times \cU )\cup \{\triangle\}$-valued 
\blue diffusion \bla with regular Dirichlet space 
$(\cE ^{(M,D,U_A)},\cF^{(M,D,U_A)})$ on $L^2(\rr^2 \times \rr _+ \times \sS , \dd z r^\nu \dd r \beta(\dd u))$,
with core $C_c^\infty (\rr ^2 \times \rr ^*_+ \times \cU)$. Moreover, for all 
$\varphi \in C_c^\infty (\rr ^2 \times \rr ^*_+ \times \cU)$,
\begin{align}
\cE ^{(M,D,U_A)}(\varphi,&\varphi)   = \int _{\rr_+ \times \sS} \!\!\cE ^M (\varphi(\cdot,r,u),\varphi(\cdot,r,u))
r^\nu \dd r \beta (\dd u)
+\int _{\rr ^2} \!\!\mathcal{E}^{(D,U_A)}(\varphi(z,\cdot,\cdot), \varphi(z,\cdot,\cdot))\dd z   
\nonumber \\
=& \int _{\rr_+ \times \sS} \!\!\cE ^M (\varphi(\cdot,r,u),\varphi(\cdot,r,u))r^\nu \dd r \beta (\dd u)
+\int _{\rr ^2 \times \sS} \!\!\mathcal{E}^{D_\rho}(\varphi(z,\cdot,u), \varphi(z,\cdot,u))\dd z \beta (\dd u)\nonumber \\
&\hskip0.7cm+ \int _{\rr^2 \times \rr_{+}} \mathcal{E}^{U}(\varphi(z,r,\cdot), \varphi(z,r,\cdot)) \dd z
r^{\nu -1 } \dd r \nonumber \\
=& \int_{\rr^2\times \rr_+ \times \sS} \Big[\frac1{2N}||\nabla_z \varphi(z,r,u)||^2+ 2r|\partial_r \varphi(z,r,u)|^2
+ \frac 1 {2r}||\nabla_\sS \varphi(z,r,u)||^2\Big] \dd z r^\nu \dd r \beta (\dd u).
\label{formedirichletMUD}
\end{align}  
For the second line, we used \eqref{tyty}.
For the last line, we used \eqref{jabM}, \eqref{jabD2} and the expression of $\cE^U$, see
Proposition~\ref{existenceXU}-(ii). 
 
\vip

{\it Step 3.} We recall that 
$Y_t=\Psi(M_t,D_t,U_{A_t})$, where $\Psi(z,r,u)=\gamma(z)+\sqrt{r}u$ for $(z,r,u)\in \rr^2\times \rr_+^*
\times \cU$ and $\Psi(z,r,u)=\triangle$ for $(z,r,u)=\triangle$.
One easily checks that $\Psi$ is a bijection from $(\rr^2\times \rr_+^* \times \cU)\cup\{\triangle\}$
to $(\cX\cap E_{N})\cup\{\triangle\}$, recall that $\cX=E_{k_0}$ and $\cU=E_{k_0}\cap \sS$.

\vip

We now study
$$
\mathbb{Y} = (\Omega^Y,\cM^Y,
(Y_t)_{t\ge 0}, (\PP^{Y}_y)_{y \in (\cX\cap E_{N})\cup\{\triangle\}}),
$$ 
where $\Omega^Y=\Omega ^M \times \Omega^D \times\Omega^U$, 
$\cM^Y=\cM^{(M,D,U)}$ and $\PP^{Y}_y=\PP_{(z,r,u)}^{(M,D,U)}$ for $(z,r,u)=\Psi^{-1}(y)$.

\vip

First, $\mathbb{Y}$ is a $(\cX \cap E_{N}) \cup\{\triangle\}$-valued \blue diffusion, \bla
because the bijection $\Psi$ from $(\rr^2\times \rr_+^* \times \cU)\cup\{\triangle\}$
to $(\cX\cap E_{N})\cup\{\triangle\}$ is continuous, both sets being endowed with 
the one-point compactification 
topology, see Subsection~\ref{ap1}.

\vip

Next, we prove that $\mathbb{Y}$ is $\mu$-symmetric:
if $\varphi,\psi$ are nonnegative measurable functions on $\cX \cap E_N$ and $t\ge 0$, 
we have, thanks to Lemma~\ref{changeofvariable}
(recall that $\nu=d_{\theta,N}(N)/2-1$),
$$
\int _{(\rr^2)^N} [P^Y_t\varphi (y)] \psi(y) \mu (\dd y)= \frac12 \int _{\rr^2 \times \rr_+\times \sS} 
[(P^Y_t\varphi)(\Psi (z,r,u))] \psi(\Psi (z,r,u)) r^\nu \dd z \dd r \beta(\dd u ) .
$$
But $(P^Y_t\varphi)(\Psi (z,r,u))
=\mathbb{E}_{ (z,r,u)}[\varphi(\Psi(M_t,D_t,U_{A_t}))]= 
P^{(M,D,U_A)}_t(\varphi\circ \Psi)(z,r,u)$, so that
\begin{align*}
\int _{(\rr^2)^N} [P^Y_t\varphi (y)] \psi(y) \mu (\dd y)=& 
\frac12 \int _{\rr^2 \times \rr_+\times \sS} [P^{(M,D,U_A)}_t(\varphi\circ \Psi)(z,r,u)] [(\psi\circ \Psi) (z,r,u)]
r^\nu \dd z \dd r \beta (\dd u).
\end{align*}
Using that $(\mathbb{M},\mathbb{D},\mathbb{U_{A}})$ is $\dd z r^\nu \dd r \beta(\dd u)$-symmetric and
then the same computation in reverse order, one concludes that
$\int _{(\rr^2)^N} [P^Y_t\varphi] \psi \dd \mu =\int _{(\rr^2)^N} \varphi [P^Y_t\psi]\dd \mu$ as desired.

\vip

Thus $\mathbb{Y}$ has a Dirichlet space $(\cE ^Y, \cF ^Y)$ on $L^2((\rr ^2)^N, \mu )$ that we now determine. 
For $\varphi\in L^2 ( (\rr^2)^N, \mu)$, using as above Lemma~\ref{changeofvariable} and that 
$(P^Y_t\varphi)(\Psi (z,r,u))= P^{(M,D,U_A)}_t(\varphi\circ \Psi)(z,r,u)$,
\begin{align*}
&\frac1t \int _{(\rr^2)^N} (P^Y_t\varphi -\varphi )\varphi \dd \mu \\
=& \frac{1}{2t} \int _{\rr^2 \times \rr^*_+\times \sS} [P^{(M,D,U_A)}_t
(\varphi\circ \Psi) (z,r,u) - (\varphi\circ \Psi)(z,r,u)] 
[\varphi\circ \Psi (z,r,u)] r^\nu \dd z \dd r  \beta (\dd u) .
\end{align*}
Since $\Psi$ is bijective, we deduce, see \cite[Lemma 1.3.4 page 23]{f}, that
\begin{gather}\label{noyauY}
\cF^Y = \Big\{\varphi\in L^2((\rr^2)^N,\mu)  : \varphi\circ \Psi \in \cF ^{(M,D,U_A)}\Big\} \\[5pt]
\hbox{and for $\varphi \in \cF^Y$,}
\quad \cE ^Y(\varphi,\varphi) = \frac12\cE ^{(M,D,U_A)}(\varphi\circ \Psi , \varphi\circ \Psi ).
\label{dirichletZ}
\end{gather}

{\it Step 4.} We now compute $\cE ^Y(\varphi,\varphi)$ for $\varphi\in C^\infty_c(\cX\cap E_{N})$, so that
$\varphi\circ \Psi \in C^\infty_c(\rr^2\times \rr_+^*\times\cU)$.
Thanks to \eqref{formedirichletMUD} and \eqref{dirichletZ}, we have
\begin{align}\label{dirichletZ2}
\cE ^Y (\varphi ,\varphi) = \frac 12 
\int_{\rr^2\times \rr_+ \times \sS} I(z,r,u)
\dd z r^\nu \dd r \beta (\dd u),
\end{align}
where
\begin{gather*}
I(z,r,u)=\frac1{2N}||\nabla_z (\varphi\circ\Psi)(z,r,u)||^2
+2r|\partial_r (\varphi\circ\Psi)(z,r,u)|^2
+\frac 1 {2r}||\nabla_\sS (\varphi\circ\Psi)(z,r,u)||^2.
\end{gather*}

We recall that for $\varphi:(\rr^2)^N\to\rr$, we call $\nabla \varphi (x)=((\nabla \varphi(x))^1,\dots,
(\nabla \varphi(x))^N) \in (\rr^2)^N$ 
the total gradient of $\varphi$ at $x \in (\rr^2)^N$, and we have $(\nabla \varphi (x))^i \in \rr^2$  for each 
$i\in\ig 1,N\id$. And for $\phi:O \to \rr^p$, where $O$ is open in $\rr^n$, 
we denote by $\dd_z \phi$ the differential of $\phi$
at $z \in O$.

\vip

We start with the study of $\Psi(z,r,u)=\gamma(z)+\sqrt{r}u$, where we 
recall that $\gamma$ was introduced in Section~\ref{nota}
and that $\Phi_\sS(x)=\pi_H x / ||\pi_H x||$ 
is defined on a neighborhood of $\sS$ in $(\rr^2)^N$, see \eqref{phis}.
It holds that for all $(z,r,u)\in \rr^2\times\rr_+^*\times \sS$ and all
$h\in \rr^2$, $k\in \rr$ and $\ell \in (\rr^2)^N$,
\begin{gather*}
\dd_z \Psi(\cdot,r,u)(h) = \gamma(h),\qquad
\dd_r \Psi(z,\cdot,u)(k) = \frac k {2\sqrt{r}}u ,\qquad
\dd_u [\Psi(z,r,\Phi_\sS(\cdot))](\ell)=\sqrt r \pi_{u^\perp}(\pi_H (\ell)),
\end{gather*}
For the first equality, it suffices to use that $\gamma$ is linear, so that 
$\dd_z \Psi(\cdot,r,u)(h)=\dd_z \gamma (h) = \gamma(h)$. The second equality is obvious.
For the third equality, which is the differential at $u\in \sS$
of the function $F(x)=\gamma(z)+\sqrt{r}\Phi_\sS(x)$ defined for $x\in E_{N}$ (which is open in $(\rr^2)^N$
and contains $\sS$),
we write $\dd_u F=\sqrt{r}\dd_u\Phi_\sS$. But $\Phi_S=G\circ\pi_H$, where $G(x)=x/||x||$, 
and we have $\dd_u \pi_H=\pi_H$ 
and $\dd_{\pi_H(u)}G=\dd_uG = \pi_{u^\perp}$ for $u \in \sS$.
All in all, $\dd_u F = \sqrt{r}\pi_{u^\perp}\circ \pi_H$.

\vip
First, we have $\nabla_z (\varphi\circ \Psi)(z,r,u)=\sum _{i=1}^N [\nabla \varphi(\Psi (z,r,u))] ^i$. Indeed,
for all $h \in \rr^2$, it holds that
$$
\dd_z (\varphi\circ \Psi(\cdot,r,u ))(h)= (\dd_{\Psi (z,r,u)} \varphi)[(\dd_z \Psi (\cdot,r,u)) (h)]
=(\dd_{\Psi (z,r,u)} \varphi)(\gamma(h))  = \nabla\varphi(\Psi (z,r,u)) \cdot \gamma(h), 
$$ 
which, by definition of $\gamma$, equals
$h\cdot \sum _{i=1}^N [\nabla \varphi(\Psi (z,r,u))]^i$.

\vip

This implies that 
\begin{equation}\label{aa1}
\frac 1{2N} \| \nabla _z (\varphi\circ \Psi (z,r,u)) \|^2 = 
\frac 1{2N} \Big\|\sum _{i=1}^N [\nabla \varphi(\Psi (z,r,u))] ^i\Big\|^2
=\frac 12 \| \pi _{H^\perp } ( \nabla \varphi(\Psi (z,r,u)))\|^2. 
\end{equation}
Indeed, recalling the expression of $\pi _H$, see Section~\ref{nota}, it suffices to note that for all 
$x \in (\rr ^2)^N$, $\|\pi _{H^\perp } (x)\|^2 = \| \gamma ( S_{\ig 1,N\id }(x))\|^2= 
N\|S_{\ig 1,N\id }(x)\|^2= N^{-1}\|\sum _{i=1} ^N x^i \|^2 $.

\vip

Next, $\partial_r (\varphi\circ\Psi)(z,r,u)=(\nabla \varphi) (\Psi(z,r,u)) \cdot u / (2\sqrt r) $. Indeed, 
for $k \in \rr$,  
$$
\dd_r(\varphi\circ\Psi(z,\cdot,u))(k)=(\dd_{\Psi (z,r,u)} \varphi)[(\dd_r \Psi (z,\cdot,u)) (k)]
=(\dd_{\Psi (z,r,u)} \varphi)(u) \times \frac k{2\sqrt r}, 
$$
which is nothing but 
$(\nabla \varphi)(\Psi(z,r,u)) \cdot u \times k/ (2\sqrt r)$.

\vip

This implies, recalling that $\pi_u$ is the orthogonal projection on ${\rm Span}(u)\subset (\rr^2)^N$,
that
\begin{equation}\label{aa2}
2r|\partial_r (\varphi\circ\Psi)(z,r,u)|^2 = \frac 12 \| \pi _u ( (\nabla \varphi)(\Psi (z,r,u)))\|^2 
= \frac 12 \| \pi_H (\pi _u ( (\nabla \varphi )(\Psi (z,r,u))))\|^2
\end{equation}
since $u \in \sS$, so that $||u||=1$ and $u\in H$.

\vip

Finally, $\nabla_\sS (\varphi\circ \Psi)(z,r,u)=\sqrt r \pi_H(\pi_{u^\perp}(\nabla \varphi(\Psi (z,r,u) )))$.
Indeed, for all $\ell \in (\rr^2)^N$,  
\begin{align*}
d_u ((\varphi\circ \Psi) (z,r,\Phi _\sS (\cdot ) ))(\ell ) 
=& (d_{\Psi (z,r,u)}\varphi)(d_u [\Psi (z,r,\Phi _\sS (\cdot ))](\ell ))\\
=& \sqrt{r}(d_{\Psi (z,r,u)}\varphi) ( \pi_{u^\perp}(\pi_H (\ell)))\\ 
=& \sqrt r  \nabla \varphi(\Psi (z,r,u)) \cdot \pi_{u^\perp} (\pi _H (\ell ))\\
=&\sqrt r \pi _H (\pi _{u^\perp} (\nabla \varphi (\Psi (z,r,u)))) \cdot \ell,
\end{align*} 
and we conclude since 
$\nabla_\sS (\varphi\circ \Psi)(z,r,u)=\nabla_x((\varphi\circ \Psi)(z,r,\Phi_\sS(\cdot)))(u)$
by definition of $\nabla_\sS$, see \eqref{nablaS}.

\vip

This implies that \begin{equation}\label{aa3}
\frac 1 {2r}||\nabla_\sS (\varphi\circ\Psi)(z,r,u)||^2 
= \frac 12 \| \pi _H  (\pi _{u^\perp} ( \nabla \varphi(\Psi (z,r,u))))\|^2.
\end{equation}

Gathering \blue \eqref{aa1}, \eqref{aa2} and \eqref{aa3}, \bla we see that 
$I(z,r,u) = \frac 12 \| \nabla \varphi (\Psi (z,r,u))\|^2$, since
for $x \in (\rr^2)^N$, 
$$
\| \pi_{H^\perp} (x) \| ^2 +\| \pi_H (\pi _u (x)) \| ^2 + \| \pi_H (\pi _{u^\perp} (x)) \| ^2 
= \|x\|^2
$$
because $u\in \sS \subset H$.

\vip

Injecting the value of $I$ in \eqref{dirichletZ2} and using Lemma~\ref{changeofvariable}, we obtain 
\begin{align*}
\cE^Y (\varphi,\varphi) 
= \frac{1}{4} \int _{\rr^2 \times \rr_{+}^{*}\times \sS} \| \nabla \varphi(\Psi (z,r,u)) \|^2 \dd z r^{\nu } \dd r 
\beta (\dd u)
=\frac 12 \int _{(\rr ^2)^N} \| \nabla \varphi \|^2 \dd \mu.
\end{align*} 

{\it Step 5.} As a last technical step, we verify that $(\cE^Y, \cF ^Y)$ 
is a regular Dirichlet space on $L^2((\rr^2)^N,\mu)$ with core $C_c^\infty (\cX \cap E_{N})$,
i.e. that for all $\varphi \in \cF^Y$, there is $\varphi_n\in C_c^\infty (\cX \cap E_{N})$ such that
$\lim_n ||\varphi_n-\varphi||_{L^2((\rr^2)^N,\mu)}+\cE^Y(\varphi_n-\varphi,\varphi_n-\varphi)= 0$.

\vip

Recalling \eqref{noyauY} and using that $(\cE ^{(M,D,U_A)},\cF^{(M,D,U_A)})$ on 
$L^2(\rr^2 \times \rr _+ \times \sS , \dd z r^\nu \dd r \beta(\dd u))$ is
regular with core $C_c^\infty (\rr ^2 \times \rr ^*_+ \times \cU)$, there is 
$g_n \in C_c^\infty (\rr ^2 \times \rr ^*_+ \times \cU)$ such that
$$
||g_n-\varphi\circ \Psi||_{L^2(\rr^2 \times \rr _+ \times \sS , \dd z r^\nu \dd r \beta(\dd u))}
+\cE^{(M,D,U_A)}(g_n-\varphi\circ \Psi,g_n-\varphi\circ \Psi)\to 0.
$$
Setting $\varphi_n=g_n\circ \Psi^{-1}$, it holds that $\varphi_n\in C_c^\infty (\cX \cap E_{N})$ and
we have, by \eqref{dirichletZ},
$$
\cE^Y(\varphi_n-\varphi,\varphi_n-\varphi)
=\frac12\cE^{(M,D,U_A)}(g_n-\varphi\circ \Psi,g_n-\varphi\circ \Psi) \to 0,
$$
as well as, by Lemma~\ref{changeofvariable},
$$
||\varphi_n-\varphi||_{L^2((\rr^2)^N,\mu)} 
= \frac12 ||g_n-\varphi\circ \Psi||_{L^2(\rr^2 \times \rr _+ \times \sS , \dd z r^\nu \dd r \beta(\dd u))}\to 0.
$$

{\it Step 6.} By Steps 3, 4 and 5, we know that $\yY$ is a $\mu$-symmetric 
$(\cX \cap E_{N})\cup\{\triangle\}$-valued \blue diffusion \bla with regular Dirichlet space $(\cE^Y,\cF^Y)$
with core $C_c^\infty (\cX \cap E_{N})$ and with 
$\cE^Y(\varphi,\varphi)=\frac12\int_{(\rr^2)^N} ||\nabla \varphi ||^2 \dd \mu$
for $\varphi \in C_c^\infty (\cX \cap E_{N})$.

\vip

Now, applying Lemma~\ref{tuage} to $\mathbb{X}$ defined in Proposition~\ref{existenceXU}-(i) with 
the open set $\cX \cap E_{N}$, we see that $\xXs$, i.e. $\xX$ killed when getting outside $\cX \cap E_{N}$,
is a $\mu$-symmetric 
$(\cX \cap E_{N})\cup\{\triangle\}$-valued \blue diffusion \bla 
process with regular Dirichlet space $(\cE^{X^*},\cF^{X^*})$
with core $C_c^\infty (\cX \cap E_{N})$ and with 
$\cE^{ X^*}(\varphi,\varphi)=\frac12\int_{(\rr^2)^N} ||\nabla \varphi ||^2 \dd \mu$
for $\varphi \in C_c^\infty (\cX \cap E_{N})$.

\vip
This implies, as recalled in Subsection~\ref{ap1}, that $(\cE^{X^*},\cF^{X^*})=(\cE^Y,\cF^Y)$.
The conclusion follows by uniqueness, see \cite[Theorem 4.2.8 p 167]{f}.
\end{proof}

Actually, $(R_{\ig1,N\id}(X_t))_{t\ge 0}$ and $(S_{\ig1,N\id}(X_t))_{t\ge 0}$ are some independent 
squared Bessel process and  Brownian motion {\it until explosion}
(and not only until the first time where $R_{\ig1,N\id}(X_t)=0$, as shown in Proposition~\ref{deco}),
a fact that we shall often use.

\begin{lemma}\label{besbro}
We fix $N\geq 2$ and $\theta>0$ such that $N > \theta$ and we consider a 
$QKS(\theta,N)$-process $\xX=(\Omega^X,\cM^X,(X_t)_{t\geq 0},(\PP_x^X)_{x\in \cXt})$.
\blue Quasi-everywhere, \bla there are a $2D$-Brownian motion $(M_t)_{t\geq 0}$ with diffusion constant $N^{-1/2}$
issued from $S_{\ig1,N\id}(x)$
and a squared Bessel process $(D_t)_{t\geq 0}$ with dimension $d_{\theta,N}(N)$ issued from $R_{\ig1,N\id}(x)$
(killed when 
it gets out of $(0,\infty)$ if $d_{\theta,N}(N)\leq 0$) independent of $(M_t)_{t\ge 0}$ such that $\PP_x^X$-a.s., 
$S_{\ig1,N\id}(X_t)=M_t$ and $R_{\ig1,N\id}(X_t)=D_t$ for all $t\in [0,\zeta)$.
\end{lemma}

\begin{proof}
If $\theta \geq 2$, this follows from Proposition~\ref{deco}: setting 
$\tau=\inf \{ t>0 : R_{\ig 1,N \id}(X_t)\notin  (0, \infty ) \}$, we have $\tau = \zeta$.
Indeed, on $\{\tau<\zeta\}$, we have $X_\tau \notin E_N$, whence $X_{\tau}\notin \cX$ since
$\cX=E_{k_0}$ with $k_0 \le N$ (because $\theta \geq 2$),
which contradicts the fact that $\tau<\zeta$.
\vip

We now suppose that $\theta<2$, so that $k_0>N$ and thus $\cX=(\rr^2)^N$. We introduce the shortened
notation $R(x)=R_{\ig 1,N \id}(x)$, $S(x)=S_{\ig 1,N \id}(x)$ and split the proof in three parts.

\vip \blue

{\it Step 1.} First, one can show similarly (but much more easily) as in the proof of 
Proposition~\ref{deco} that there exists a $2D$-Brownian motion $(M_t)_{t\geq 0}$ independent of
$(X_t-\gamma(S(X_t)))_{t\geq 0}$, such that $S(X_t)=M_t$ for all $t\in [0,\zeta)$.
This moreover shows that $(M_t)_{t\geq 0}$ is independent of $(R(X_t))_{t\geq 0}$,
because $R(X_t)=\|X_t-\gamma(S(X_t))\|^2$.
\vip

{\it Step 2.} We consider some function $g_m \in C^\infty_c((\rr^2)^N)$ such that $g_m=1$
on $B(0,m)$ and $\sup_{\alpha \in (0,1]}\sup_{x\in (\rr^2)^N} |\cL^X_\alpha g_m(x)|<\infty$.
Such a function exists by Remark \ref{indiqXU2}.
For $\varphi \in C^\infty_c(\rr_+)$, we set $\psi (x)= \varphi(R(x))$
and show that $\psi g_m \in \cD_{\cA^X}$ and that for all $x \in B(0,m)$,
\begin{align}\label{abcdg}
\cA^X (\psi g_m) (x) =& 2 R(x) \varphi'' (R(x))+d_{\theta,N}(N)
\varphi' (R(x)).
\end{align}

To this end, we apply Remark~\ref{caracdomaine}. Since 
$\psi g_m \in C^\infty_c((\rr^2)^N)$ and since $\cX=(\rr^2)^N$, we have to show that
$\sup _{\alpha \in (0,1]}\sup_{x\in (\rr^2)^N} |\cL^X_\alpha (\psi g_m) (x) | < \infty$, and we will deduce
that $\cA^X (\psi g_m)=\cL^X(\psi g_m)$. By \eqref{lproduit}, 
we have $\cL^X_\alpha (\psi g_m) = g_m \cL^X_\alpha \psi+ \psi \cL^X_\alpha g_m+\nabla \psi \cdot\nabla g_m$.
The only difficulty consists in showing that $\sup _{\alpha \in (0,1]}\sup_{x\in (\rr^2)^N} |\cL^X_\alpha \psi (x) | 
< \infty$.
Using that $\nabla_{x^i}R(x)=2(x^i-S(x))$, we find
$\nabla _{x^i} \psi (x) =2(x^i - S(x)) \varphi'(R(x)).$
Hence by symmetry,
\begin{align}\label{singpart}
\frac \theta N \sum_{1\leq i\neq j \leq N}\frac{x^i-x^j}{\|x^i-x^j\|^2+\alpha}\cdot \nabla _{x^i} \psi (x) =& 
\frac{2\theta} N \varphi' (R(x)) 
\sum_{1\leq i\neq j \leq N}\frac{x^i-x^j}{\|x^i-x^j\|^2+\alpha}\cdot x^i \notag\\
=&\frac{\theta} N \varphi' (R(x))
\sum_{1\leq i\neq j \leq N}\frac{\|x^i-x^j\|^2}{\|x^i-x^j\|^2+\alpha}.
\end{align}
Besides, $\Delta _{x^i} \psi (x) = 4(1-1/N ) \varphi'(R(x))
+ 4 \|x^i - S(x)\|^2 \varphi'' (R(x)),$ whence
\begin{align}\label{laplapart}
\Delta \psi (x) =&  4(N-1) \varphi'(R(x)) +4 R(x) \varphi'' (R(x)).
\end{align}
We conclude by combining \eqref{singpart} and \eqref{laplapart} that
\begin{align*}
\cL^X_\alpha \psi (x) =&  2 R(x) \varphi'' (R(x))
+ \Big( 2(N-1)-\frac\theta {N}  \sum_{1\leq i\neq j \leq N} \frac{\|x^i-x^j\|^2}{\|x^i-x^j\|^2+\alpha} \Big) 
\varphi' (R(x)).
\end{align*}
We immediately deduce, since $\varphi$ is compactly supported, that 
$\sup _{\alpha \in (0,1]}\sup_{x\in (\rr^2)^N}  |\cL^X_\alpha \psi (x) | < \infty$, whence 
$\sup _{\alpha \in (0,1]}\sup_{x\in (\rr^2)^N}  |\cL^X_\alpha (\psi g_m) (x) | < \infty$.
Hence $\psi g_m \in \cD_{\cA^X}$ and $\cA^X (\psi g_m)=\cL^X (\psi g_m)$.
Moreover, recalling that $\cL^X\psi=\cL^X_\alpha \psi$ with $\alpha=0$ and that $g_m=1$ on $B(0,m)$, 
we conclude that $\cA^X (\psi g_m)(x)=\cL^X_0 \psi(x)$ for $x \in B(0,m)$, whence 
\eqref{abcdg}, because $2(N-1)-\theta (N-1)=d_{\theta,N}(N)$.

\vip
 
{\it Step 3.} We define $\zeta_m= \inf\{t>0 : X_t \notin B(0,m)\}$.
By Lemma~\ref{marting} and  Step 1, for all $\varphi \in C^\infty_c(\rr_+)$,
quasi-everywhere in $B(0,m)$,
$\varphi(R(X_{t\land \zeta_m}))- \varphi(R(x))-\int_0^{t\land \zeta_m} \cL^X \varphi(R(X_s)) \dd s$ 
is a $\PP_x^X$-martingale.
Recalling \eqref{abcdg}, we classically conclude that there is a Brownian motion $W$ such that
$R(X_t)=R(x)+2\int_0^t \sqrt{R(X_s)} \dd W_s + d_{\theta,N}(N) t$ during $[0,\zeta_n]$.
We recognize the S.D.E. of a squared Bessel process with dimension $d_{\theta,N}(N)$, 
see Revuz-Yor \cite[Chapter XI]{ry}. Since we know from Remark \ref{rklt} that $\zeta=\lim_m \zeta_m$,
the proof is complete. \bla
\end{proof}

\section{Some cutoff functions}\label{cutcut}

We will need several times to approximate some indicator functions by some smooth functions,
on which the generator $\cL^X$ (or $\cL^U$) is bounded. 
This does not seem obvious, due to the singularity
of $\cL^X$. We recall that $\cL^X_\alpha$ and $\cL^U_\alpha$ were defined in \eqref{dflxa} and \eqref{dflua}.

\blue
\begin{lemma}\label{indiqXU}
Fix $N \geq 2$, $\theta>0$, recall that $k_0= \lceil 2N/\theta\rceil$ and that $\cX = E_{k_0}$.
Consider a partition $\bK=(K_p)_{p\in \ig1,\ell\id}$ and define, for $\e\in [0,1]$, (with the convention that
$B(0,1/0)=(\rr^2)^N$),
$$
G_{\bK,\e}=\Big\{x \in \cX : \min_{1\leq p\neq q \leq \ell}\;\;\min_{i\in K_p,j\in K_q}||x^i-x^j||^2>\e  \Big\}
\cap B\Big(0,\frac1\e\Big).
$$

(i) For all $\e\in (0,1]$, there is a family of open relatively compact subsets $G^{n}_{\bK,\e}$ of 
$G_{\bK,0}$ such that
$$
\bigcup_{n\geq 1} G^{n}_{\bK,\e} \supset G_{\bK,\e}
\quad \hbox{and for each $n\geq 1$, } G^{n}_{\bK,\e}
\subset G^{n+1}_{\bK,\e},
$$
and some of $[0,1]$-valued functions $\Gamma_{\bK,\e}^{n} \in C^\infty_c(G_{\bK,0})$
such that for some $\eta\in (0,1]$, for all $n\geq 1$,
$$
{\rm Supp}\; \Gamma_{\bK,\e}^{n} \subset G_{\bK,\eta},\quad
\Gamma_{\bK,\e}^{n} =1 \; \hbox{ on }\; G_{\bK,\e}^{n}  \quad \hbox{and}\quad  
\sup_{\alpha\in(0,1]}\sup_{x\in (\rr^2)^N}\Big|\cL^X_\alpha \Gamma_{\bK,\e}^{n}  (x)\Big|<\infty.
$$

(ii) With the same sets $G^{n}_{\bK,\e}$ as in (i), there 
 is a family of functions $\Gamma_{\bK,\e}^{\sS,n} \in C^\infty_c(\sS \cap G_{\bK,0})$ with values in $[0,1]$
such that for all $n\geq 1$,
$$
\Gamma_{\bK,\e}^{\sS,n}=1 \; \hbox{ on }\; \sS \cap G^{n}_{\bK,\e} \quad \hbox{and}\quad  
\sup_{\alpha\in(0,1]}\sup_{ u \in \sS}\Big|\cL^U_\alpha \Gamma_{\bK,\e}^{\sS,n} (u)\Big|<\infty.
$$
\end{lemma}
\bla

The section is devoted to the proof of this lemma.
We start with the following technical result. 

\begin{lemma}\label{campingcar}
We define the family $(c_\ell)_{\ell \in \ig 1,N\id }$ by $c_0 =1$ and for all $\ell \in \ig 1,N-1 \id$, 
$c_{\ell+1} = \blue (2+4\ell)c_\ell$.
For all   $K \subsetneq \ig 1,N \id$,   all $\e \in (0,1]$, all $x \in (\rr ^2)^N$ such that 
$$
R_{K}(x) \le 2c_{|K|}\e  \quad \hbox{and}\quad   \min_{j\notin K} R_{K\cup \{j\} } (x) \ge c_{|K|+1} \e, 
$$ 
it holds that $\|x^{i}-x^{j}\|^{2} \ge c_{|K|}\e$ 
for all $i \in K$, all $j \notin K$.
\end{lemma}

\begin{proof}
We   fix $K \subsetneq \ig 1,N \id$, $\e\in (0,1]$ 
\blue and $ x \in (\rr ^2)^N$ as in the statement \bla and assume by 
contradiction that there are $i_{0} \in K$, $ j_{0} \notin K$ such that 
$\|x^{i_{0}}-x^{j_{0}}\|^{2} < c_{|K|} \e$.
Then for all $i \in K$,  
$$
\|x^{j_{0}}-x^{i}\|^{2} \le 2  \|x^{i_{0}}-x^{j_{0}}\|^{2} +2  \|x^{i_{0}}-x^{i}\|^{2}  \leq  2  \|x^{i_{0}}-x^{j_{0}}\|^{2} +
\blue 2 |K| \bla R_{K}(x) < \blue (2+4|K|)\bla c_{|K|}\e.
$$
This implies that \blue
$$ 
R_{K\cup \{j_{0} \}}(x) =\frac1{2(|K|+1)}\Big( 2 |K| R_{K}(x) + 2 \sum _{i\in K} \|x^{j_{0}}-x^{i}\|^{2} \Big)
\leq R_{K}(x)+ \frac{1}{|K|+1}  \sum _{i\in K} \|x^{j_{0}}-x^{i}\|^{2},
$$ 
whence
$$
R_{K\cup \{j_{0} \}}(x)< 2c_{|K|}\e + \frac{2+4|K|}{|K|+1} |K| c_{|K|}\e< (2+4|K|) c_{|K|} \e =  c_{|K|+1}\e,
$$
\bla
which is a contradiction.
\end{proof}

We are now ready to give the

\begin{proof}[Proof of Lemma~\ref{indiqXU}]
We introduce some nondecreasing $C^\infty$ function $\varrho : \rr _+ \to [0,1]$
such that $\varrho=0$ on $[0,1/2]$  and $\varrho=1$ on $[1,\infty)$.
We divide the proof in three steps.

\vip

\blue

{\it Step 1.} We fix $n\geq 1$ and define, for $K\subset \ig 1,N\id$, 
using the family  $(c_\ell)_{\ell \in \ig 1,N\id }$ of Lemma~\ref{campingcar},
$$
\tilde{E}_{K,n}=\Big\{x \in (\rr^2)^N : \forall \; L \supset K,\; R_L(x)>\frac{c_{|L|}}n\Big\}
\qquad \hbox{and} \qquad \tilde{\Gamma}_{K,n}(x)=\prod_{L\supset K} \varrho\Big(\frac{n R_L(x)}{c_{|L|}}\Big),
$$
where $\{L\supset K\}=\{L \subset \ig 1,N\id : K \subset L\}$.
We have 
\begin{equation}\label{ttf1}
\tilde{\Gamma} _{K,n}\in C^\infty ( (\rr^2)^N),\quad  \mbox{Supp } \tilde{\Gamma} _{K,n} \subset  
\tilde{E}_{K,2n}
\quad \mbox{and}\quad \tilde{\Gamma}_{K,n} =1 \quad \mbox{on} \quad \tilde{E}_{K,n}.
\end{equation}
Since $R_K(x)>0$ implies that $R_L(x)>0$ for all $L\supset K$, we also have
\begin{equation}\label{ttf2}
\cup_{n\geq 1} \tilde{E}_{K,n} = \tilde E_K, \quad \hbox{where}\quad \tilde E_K=\{x\in (\rr^2)^N :  R_K(x)>0\}.
\end{equation}
We now show, and this is the main difficulty of the step, that for all $A>0$, all $K \subset \ig1,N\id$
with $|K|\geq 2$, we have $\sup_{\alpha\in(0,1]}\sup_{x\in B(0,A)} |\cL^X_\alpha\tilde{\Gamma}_{K,n}(x)|<\infty$.
Since $\sup_{x\in B(0,A)} |\Delta \tilde{\Gamma}_{K,n} (x)|<\infty$, we only have to verify that 
$\sup_{\alpha \in (0,1]} \sup_{x\in B(0,A)}|I_{K,n,\alpha}(x)| <\infty$, where
$$
I_{K,n,\alpha}(x)=\sum_{1\leq i\neq j \leq N} \frac{x^i-x^j}{\|x^i-x^j\|^2} \cdot \nabla_{x^i} \tilde \Gamma_{K,n} (x)
=\sum_{L \supset K} f_{K,L,n}(x)
\sum_{1\leq i\neq j \leq N} \frac{x^i-x^j}{\|x^i-x^j\|^2} \cdot \nabla_{x^i} R_L(x),
$$
with 
$$
f_{K,L,n}(x)=\frac n{c_{|L|}}\varrho'\Big(\frac{n R_L(x)}{c_{|L|}}\Big)\prod_{M\supset K, M\neq L} 
\varrho\Big(\frac{n R_M(x)}{c_{|M|}}\Big).
$$ 
Using that $\nabla_{x^i} R_L(x)=2(x^i-S_L(x))\indiq_{\{i \in L\}}$, we now write 
$$
I_{K,n,\alpha}(x) =2\sum_{L \supset K} f_{K,L,n}(x)(A_{L,\alpha}(x) + B_{L,\alpha}(x)),
$$
where,
$$
A_{L,\alpha}(x)=\sum_{i,j\in L, i\neq j}  \frac{(x^i-x^j)\cdot (x^i-S_L(x))}{\|x^i-x^j\|^2+\alpha} 
\quad \hbox{and} \quad B_{L,\alpha}(x)=\sum_{i\in L, j \in L^c}  
\frac{(x^i-x^j)\cdot (x^i-S_L(x))}{\|x^i-x^j\|^2+\alpha}.
$$
We have $\sup_{\alpha \in (0,1]} \sup_{x\in B(0,A)}|f_{K,L,n}(x)A_{L,\alpha}(x)| <\infty$ 
because $f_{K,L,n}$ is bounded and because
\begin{align*}
A_{L,\alpha}(x)
=\sum_{i,j\in L, i\neq j}  \frac{(x^i-x^j)\cdot x^i}{\|x^i-x^j\|^2+\alpha}
=&\frac12\sum_{i,j\in L, i\neq j}  \frac{\|x^i-x^j\|^2}{\|x^i-x^j\|^2+\alpha}\in \Big[0,\frac{|L|  (|L|-1)}2\Big].
\end{align*}
Next, we assume that $L \subsetneq \ig 1, N\id$ (else $B_{L,\alpha}(x)=0$) and 
observe that $f_{K,L,n}(x) \neq 0$ implies that $R_L(x)<c_{|L|}/n$ (because $\varrho'=0$ on $[1,\infty)$) and that 
$\min_{i\notin L} R_{L\cup\{i\}}(x)>c_{|L|+1}/(2n)$ (because $\varrho=0$ on $[0,1/2]$). By Lemma~\ref{campingcar},
this implies that $\min_{i \in L, j \in L^c} ||x^i-x^j||^2 \geq c_{|L|}/(2n)$. We immediately conclude that
$\sup_{\alpha \in (0,1]} \sup_{x\in B(0,A)}|f_{K,L,n}(x)B_{L,\alpha}(x)| <\infty$.

\vip

{\it Step 2.}
We can now prove (i). We fix $\e \in (0,1]$ and a 
partition $\bK=(K_p)_{p\in \ig1,\ell\id}$ of $\ig 1,N\id$. 
For some $m \geq 1$ to be chosen later (as a function of $\e$), for each $n\geq 1$, we set
\begin{gather*}
G_{\bK,\e}^{n} = B(0,m) \cap \Big(\bigcap _{K\subset \ig 1,N\id : |K|=k_0} \tilde{E}_{K,n} \Big) \cap 
\Big( \bigcap_{1\leq p \neq q\leq \ell}\;\;\bigcap_{i\in K_p,j\in K_q} \tilde E_{\{i,j\},m}\Big),\\
\Gamma_{\bK,\e}^{n} (x) = g_m(x) \Big(\prod_{K\subset \ig 1,N\id:|K|=k_0}\tilde{\Gamma}_{K,n}(x) \Big)
\Big(\prod_{1\leq p \neq q\leq \ell}\;\;
\prod_{i\in K_p,j\in K_q} \tilde{\Gamma}_{\{i,j\},m}(x)\Big),
\end{gather*}
where $g_m(x) = \varrho (m/\|x\|)$ with the extension $g_m(0)=1$.

\vip

First, $G_{\bK,\e}^{n}$ is clearly included in $G_{\bK,\e}^{n+1}$ and relatively compact in
$G_{\bK,0}$. We deduce from \eqref{ttf2} that,
setting $H_{\bK,m}=B(0,m)\cap (\cap_{1\leq p \neq q\leq \ell}\;\;\cap_{i\in K_p,j\in K_q} \tilde E_{\{i,j\},m})$,
$$
\bigcup_{n\geq 1} G_{\bK,\e}^{n} = \Big(\bigcap _{K\subset \ig 1,N\id : |K|=k_0} \tilde E_K \Big)\cap
H_{\bK,m} = E_{k_0} \cap H_{\bK,m}= \cX \cap H_{\bK,m}.
$$
By \eqref{ttf2} again, we can choose $m$ large enough so that
$H_{\bK,m}$ contains $G_{\bK,\e}$. Next, by \eqref{ttf1}, it holds that
$\Gamma_{\bK,\e}^{n} \in C^\infty((\rr^2)^N)$, that $\Gamma_{\bK,\e}^{n} =1$ on $G_{\bK,\e}^{n}$ and that
$$
{\rm Supp}\; \Gamma_{\bK,\e}^{n} \subset B(0,2m) 
\cap \Big(\bigcap _{K\subset \ig 1,N\id : |K|=k_0} \tilde{E}_{K,2n} \Big) \cap 
\Big( \bigcap_{1\leq p \neq q\leq \ell}\;\;\bigcap_{i\in K_p,j\in K_q} \tilde E_{\{i,j\},2m}\Big),
$$
which is compact in $G_{\bK,0}$. Moreover, ${\rm Supp}\; \Gamma_{\bK,\e}^{n}
\subset H_{\bK,2m}$.
Since there exists $\eta\in (0,1]$ such that 
$H_{\bK,2m} \subset G_{\bK,\eta}$, we conclude that Supp $\Gamma_{\bK,\e}^{n}\subset G_{\bK,\eta}$.

\vip

It remains to show that $\sup_{\alpha\in (0,1]} \sup_{x\in (\rr^2)^N} |\cL^X_\alpha  \Gamma^{n}_{\bK,\e} (x)|<\infty$.
Introducing 
$$
\chi_{\bK,\e}^n(x)=\Big(\prod_{K\subset \ig 1,N\id:|K|=k_0}\tilde{\Gamma}_{K,n}(x) \Big)
\Big(\prod_{1\leq p \neq q\leq \ell}\;\;
\prod_{i\in K_p,j\in K_q} \tilde{\Gamma}_{\{i,j\},m}(x)\Big),
$$
which belongs to $C^\infty((\rr^2)^N)$ by Step 1, we have $\Gamma^{n}_{\bK,\e}=g_m \chi_{\bK,\e}^n(x)$ (with the chosen
value of $m$) and thus
by \eqref{lproduit}
$$ 
\cL^X_\alpha  \Gamma^{n}_{\bK,\e} (x) = g_m(x) \cL^X_\alpha \chi_{\bK,\e}^n(x) + \chi_{\bK,\e}^n\cL^X_\alpha g_m(x) + 
\nabla g_m(x) \cdot \nabla \chi_{\bK,\e}^n(x).
$$
The first term is uniformly bounded because $g_m$ is bounded and supported in $B(0,2m)$
and because $\sup_{\alpha\in (0,1]} \sup_{x\in B(0,2m)} |\cL^X \chi_{\bK,\e}^n (x)|<\infty$
by Step 1 and \eqref{lproduit}.
The third term
is also uniformly bounded, since $\chi_{\bK,\e}^n \in C^\infty((\rr^2)^N)$ and since $\nabla g_m$ is bounded
and supported in $B(0,2m)$. Finally, the middle term is bounded because $\chi_{\bK,\e}^n$ is bounded by $1$ and
because $\cL^X_\alpha g_m$ is uniformly bounded, as we now show: $\Delta g_m$ is obviously bounded
since $g_m \in C^\infty_c((\rr^2)^N)$
and, since $\nabla_{x^i} g_m (x) = -m \varrho'(m/||x||) x^i / \|x\|^3$,
\begin{align*}
\sum_{1\leq i, j \leq N} \frac{x^i-x^j}{\|x^i-x^j\|^2+\alpha}\cdot \nabla _{x^i}g_m (x) 
=& -\frac{m \varrho'(m/||x||)}{\|x\|^3}\sum_{1\leq i, j \leq N} \frac{x^i-x^j}{\|x^i-x^j\|^2+\alpha}\cdot x^i\\
=& -\frac{m\varrho'(m/||x||)}{2\|x\|^3}\sum_{1\leq i, j \leq N} \frac{\|x^i-x^j\|^2}{\|x^i-x^j\|^2+\alpha}.
\end{align*}
This last quantity is uniformly bounded,
since $\varrho'$ is bounded and vanishes on $[1,\infty)$.

\vip

{\it Step 3.} We now prove (ii), by showing that  the restriction
$\Gamma_{\bK,\e}^{\sS,n} = \Gamma_{\bK,\e}^{n}|_\sS$ satisfies the required conditions.
We obviously have $\Gamma_{\bK,\e}^{\sS,n} \in C^\infty_c(\sS\cap G_{\bK,0})$ and
$\Gamma_{\bK,\e}^{\sS,n} = 1$ on $ \sS \cap G_{\bK,\e}^{n}$. It remains to show that 
$\sup_{\alpha\in (0,1]}\sup_{u\in \sS} |\cL^U_\alpha \Gamma_{\bK,\e}^{\sS,n}|<\infty$,
recall \eqref{dflua}. Since $\Gamma_{\bK,\e}^{\sS,n} \in C^\infty(\sS)$, $\Delta_\sS \Gamma_{\bK,\e}^{\sS,n}$ 
is bounded.
We thus only have to verify that  $\sup_{\alpha\in (0,1]}\sup_{u\in \sS} |T_\alpha(u)|<\infty$, where
$$
T_\alpha(u)= -\frac\theta N \sum_{1\leq i, j \leq N} \frac{u^i-u^j}{\|u^i-u^j\|^2+\alpha}
\cdot (\nabla_\sS\Gamma_{\bK,\e}^{\sS,n}(u))^i
$$
Setting $b^i_\alpha(u)=-\frac\theta N \sum_{j=1}^N \frac{u^i-u^j}{\|u^i-u^j\|^2+\alpha}$
and using \eqref{tradgrad},
$$
T_\alpha(u) = b_\alpha(u)\cdot \nabla_\sS\Gamma_{\bK,\e}^{\sS,n}(u)
=b_\alpha(u) \cdot \pi_H (\pi_{u^\perp}(\nabla \Gamma_{\bK,\e}^{\sS,n}(u))).
$$
Since now $b(u) \in H$ and since
$\pi _H$ and $\pi_{u^\perp}$ are self-adjoint, as every orthogonal projection, we get
$$ 
T_\alpha(u) = \pi_{u^\perp} (b_\alpha(u)) \cdot  \nabla \Gamma_{\bK,\e}^{\sS,n}(u)
=b_\alpha(u)\cdot  \nabla \Gamma_{\bK,\e}^{\sS,n}(u) - ( b_\alpha(u)\cdot u )  (u \cdot  \nabla \Gamma_{\bK,\e}^{\sS,n}(u)).
$$
But $b_\alpha(u)\cdot  \nabla \Gamma_{\bK,\e}^{\sS,n}(u) = \cL^X_\alpha \Gamma_{\bK,\e}^{\sS,n}(u) 
- \frac12\Delta \Gamma_{\bK,\e}^{\sS,n}(u)$
is uniformly bounded by point (i) and since $\Delta \Gamma_{\bK,\e}^{\sS,n}(u)$ is bounded on $\sS$. Next,
$u \cdot  \nabla \Gamma_{\bK,\e}^{\sS,n}(u)$ is smooth and thus bounded on $\sS$. Finally,
$$
b_\alpha(u)\cdot u= -\frac{\theta}N \sum_{1\leq i, j \leq N} \frac{(u^i-u^j)\cdot u^i}{\|u^i-u^j\|^2+\alpha}= 
-  \frac{\theta }{2 N}  \sum_{1\leq i, j \leq N} \frac{\|u^i-u^j\|^2}{\|u^i-u^j\|^2+\alpha}
$$
is also uniformly bounded.
\end{proof}

\blue
\begin{rk}\label{indiqXU2}
We have proved in Step 2 that for each $m>0$, 
$g_m \in C^\infty_c((\rr^2)^N)$ satisfies $g_m=1$ on $B(0,m)$ and 
$\sup_{\alpha \in (0,1]} \sup_{x \in (\rr^2)^N} |\cL^X_\alpha g_m(x)|<\infty$.
\end{rk}
\bla

\section{A Girsanov theorem for the Keller-Segel particle system.}\label{gggir}

In this section, we prove a rigorous version of the intuitive argument presented in 
Subsection~\ref{girsanovbessel}.

\vip

For $x\in (\rr^2)^N$, all $K\subset \ig 1,N\id$, we denote by $x|_K=(x^i)_{i\in K}$.
For $\bK= (K_p)_{p\in \ig 1,\ell \id}$ a partition of $\ig 1,N \id$, for 
$y_1\in(\rr^2)^{|K_1|}, \dots,y_\ell\in(\rr^2)^{|K_\ell|}$, we abusively denote by $(y_p)_{p\in \ig 1,\ell\id}$
the element $y$ of $(\rr^2)^N$ such that for all $i\in \ig 1,\ell \id$, $y|_{K_i}=y_i$.

\vip

We adopt the convention that for any $\theta>0$, a $QKS(\theta,1)$-process is a $2$-dimensional Brownian motion.
This is natural in view of \eqref{EDS}.

\begin{prop}\label{girsanov}
Let $N\ge 2$, $\theta >0$ such that $N> \theta$ and set $k_0=\lceil 2N/\theta \rceil$. Fix 
some partition $\bK = (K_p)_{p\in \ig 1,\ell \id}$ of $\ig 1,N \id$ with $\ell \geq 2$. 
Consider the state spaces $\cX=E_{k_0}$ and, for each $p\in \ig 1,\ell\id$, 
$$
\cY_p=\Big\{y \in (\rr^2)^{|K_p|} : \forall K \subset \ig 1,|K_p|\id \hbox{ with $|K|\geq k_0$, } \sum_{i,j=1}^{|K_p|}
||y^i-y^j||^2>0\Big\}.
$$
Consider 

\vip

\noindent $\bullet$ $\xX = (\Omega^X,\cM^X,(X_t)_{t\geq 0},(\PP^X_x)_{x\in \cXt})$ a $QKS(\theta ,N)$-process,

\vip

\noindent $\bullet$ For all $p\in \ig 1,\ell \id$, $\yY^p = (\Omega^p,\cM^p,
(Y_{p,t})_{t\geq 0},(\PP^p_y)_{y\in \cY^{p}_\triangle})$ a 
$QKS(\theta|K_p|/N ,|K_p|)$-process.

\vip

We set $\Omega^Y=\prod_{p=1}^\ell \Omega^p$ and $Y_t=(Y_{p,t})_{p \in \ig 1,\ell\id}$, with the convention that
$Y_t=\triangle$ as soon as $Y_{p,t}=\triangle$ for some $p\in \ig1,\ell\id$.
We also introduce $\cM^Y=\sigma(Y_t : t\geq 0)$, as well as 
$\PP^Y_y= \otimes _{p=1}^\ell \PP^{p}_{y_p}$ for all $y=(y_p)_{p\in \ig 1,\ell\id} \in (\rr^2)^N$.

\vip

We fix $\e \in (0,1]$, \blue recall that
$$
G_{\bK,\e} = \Big\{ x \in \cX : \min_{1\leq p\neq q \leq \ell} \;\;\min _{i\in K_p, j\in K_q}\|x^i-x^j\|^2 > \e\Big\}
\cap B\Big(0,\frac 1\e\Big),
$$ 
and set \bla
\begin{gather*}
\tau_{\bK,\e}=\big\{ t\geq 0 : X_t \notin G_{\bK,\e} \} \quad \hbox{and} \quad 
\tilde{\tau}_{\bK,\e}=\big\{ t \geq 0 : Y_t \notin G_{\bK,\e} \}.
\end{gather*}

Fix $T>0$. \blue Quasi-everywhere in $G_{\bK,\e}$, \bla there is a probability measure 
$\mathbb{Q}_x^{T, \e, \bK}$ on $(\Omega^X,\cM^X)$, equivalent to  $\PP_x^X$, such that
the law of the process $(X_{t\land T \land \tau_{\bK,\e}})_{t\geq 0}$ under $\mathbb{Q}_x^{T, \e, \bK}$
is the same as that of 
$(Y_{t\land T \land \tilde{\tau}_{\bK,\e}})_{t\geq 0}$ 
on $(\Omega^Y,\cM^{Y})$ under $\PP^Y_x$.
\vip

Furthermore, the Radon-Nikodym density $\frac{\dd\mathbb{Q}^{T,\e,\bK}_x}{\dd\PP^X_{x}}$ 
is $\cM^X_{T\land \tau_{\bK,\e}}$-measurable, where as usual $\cM^X_t = {\sigma (X_s , s\le t)}$,  and there is a 
deterministic constant $C_{T,\e,\bK}>0$ such that \blue quasi-everywhere in $G_{\bK,\e}$, \bla
$$C_{T, \e,\bK}^{-1} \le\frac{d\mathbb{Q}^{T,\e,\bK}_x}{d\PP^X_{x}} \le C_{T, \e,\bK}.$$
\end{prop}

The \blue quasi-everywhere \bla notion refers to the process $\xX$.
Let us mention that for $\zeta$ the life-time of $\xX$, 
we have $\tau_{\bK,\e}\in [0,\zeta]$ when
$\zeta<\infty$ because $\triangle \notin G_{\bK,\e}$.
\blue Although this is not clear at this point of the paper, the event $\{\tau_{\bK,\e}=\zeta\}$ has a 
positive probability if $\max_{p=1,\dots,\ell} |K_p|\geq k_0$. \bla

\begin{proof}
\blue We only consider the case where $\ell =2$. The general case is
heavier in terms of notation but contains no additional difficulty.
We fix $\bK=(K_1,K_2)$ a non-trivial partition of $\ig 1,N\id$. 
The main idea is to apply Lemma~\ref{girsanovfuku} to $\xX$
with the function
\begin{equation}\label{rrrrr}
\varrho(x)=\exp(u(x)), \quad \hbox{where}\quad 
u(x) = \frac {\theta } N \sum _{i \in K_1, j\in K_2}\log ( \|x^i -x^j\|).
\end{equation}
Unfortunately, this is not licit because $u \notin \cF^X$.

\vip

{\it Step 1.} Set $\yY=(\Omega^Y,\cM^{Y},(Y_t)_{t\geq 0},
(\PP^Y_y)_{y \in (\cY_{1} \times \cY_{2})\cup \{\triangle\}})$ and fix $\e\in (0,1]$ and $n\geq 1$.
We first compute the Dirichlet space of 
$\yY$ killed when it
gets outside of $G_{\bK,\e}^{n}$, recall Lemma~\ref{indiqXU}. Consider the measures
$$
\mu_1(\dd y)=\prod_{i,j\in K_1,  i \neq j} ||y^i-y^j||^{-\theta/N}\dd y \quad \hbox{and}\quad
\mu_2( \dd y  )=\prod_{i,j\in K_2,  i \neq j} ||y^i-y^j||^{-\theta/N}\dd y
$$ 
on $(\rr^2)^{|K_1|}$ and $(\rr^2)^{|K_2|}$, with  $\mu_i(\dd y)=\dd y$ if $|K_i|=1$.
Recall that $\mu(\dd x)=\bm(x)\dd x$, see \eqref{mmu} and that
by definition, see \eqref{rrrrr},
$\varrho (x) = \prod_{i\in K_1, j\in K_2} \|x^i-x^j \|^{\theta/ N}$: we deduce that
$$
\mu_1\otimes \mu_2 = \varrho^2 \mu.
$$
By Proposition~\ref{existenceXU}, for $p=1,2$, $\yY^p$ is a $\cY^p_\triangle$-valued 
$\mu_p$-symmetric (since $(\theta|K_p|/N)/|K_p|=\theta/N$)
diffusion with regular Dirichlet space $(\cE_p,\cF_p)$ with core
$C^\infty_c(\cY_p)$ and, for $\varphi \in C^\infty_c(\cY_p)$, 
$\cE_p(\varphi,\varphi)=\frac12 \int_{(\rr^2)^{|K_p|}} ||\nabla \varphi||^2 \dd \mu_p$.
This also holds true if e.g. $|K_1|=1$, see \cite[Example 4.2.1 page 167]{f}, since then $\mu_1$
is nothing but the Lebesgue measure on $\rr^2$.
Since now $\mu_1\otimes\mu_2=\varrho^2 \mu$, by Lemma~\ref{concatenation},
$\yY$
is a $\varrho^2 \mu$-symmetric $\cX_\triangle$-valued diffusion with regular Dirichlet
space $(\cE^{Y},\cF^{Y})$ on $L^2(\cY_{1} \times \cY_{2}, \varrho ^2 \dd \mu)$ with core
$C^\infty_c(\cY_{1} \times \cY_{2})$ and, for $\varphi\in C^\infty_c(\cY_{1} \times \cY_{2})$,
\begin{align*}
\cE^{Y}(\varphi,\varphi)=&\int_{(\rr^2)^{|K_1|}} \hskip-0.3cm\cE_2(\varphi(y,\cdot),\varphi(y,\cdot))\mu_1(\dd y)
+ \int_{(\rr^2)^{|K_2|}} \hskip-0.3cm\cE_1(\varphi(\cdot,z),\varphi(\cdot,z))  \mu_2(\dd z)
=\frac{1}{2} \int_{(\rr^2)^N} \hskip-0.3cm\|\nabla  \varphi \|^2  \varrho^2 \dd \mu.
\end{align*}
Finally, we
apply Lemma~\ref{tuage} to $\yY$ with the open set $G_{\bK,\e}^{n} \subset \cX\subset  \cY_1\times\cY_2$,
to find that the resulting killed process 
$$
\yY^{n,\e}=\Big(\Omega^Y,\cM^{Y},(Y_t^{n,\e})_{t\geq 0},(\PP^Y_y)_{y\in G_{\bK,\e}^{n}\cup\{\triangle\}} \Big)
$$ 
is a
$ \varrho^2  \mu |_{G_{\bK,\e}^{n}}$-symmetric $G_{\bK,\e}^{n}\cup \{\triangle\}$-valued
diffusion with regular Dirichlet space
$({\cE}^{Y,n,\e},{\cF}^{Y,n,\e})$ with core $C_c^\infty (G_{\bK,\e}^{n})$ such that for all 
$\varphi \in C_c^\infty (G_{\bK,\e}^{n})$, 
$$
{\cE}^{Y,n,\e} (\varphi,\varphi) 
= \frac 12 \int _{G_{\bK,\e}^{n}} ||\nabla \varphi||^2 \varrho^2 \dd \mu.
$$

\vip

{\it Step 2.} We now fix $\e\in (0,1]$ and introduce, for each $n\geq 1$,
$u_{n,\e}(x)= u(x) \Gamma_{\bK,\e}^{n}(x)$, recall \eqref{rrrrr} and Lemma~\ref{indiqXU},
and $\varrho_{n,\e}=\exp(u_{n,\e})$. We check here that the functions $u_{n,\e}$ and $\varrho_{n,\e}$
satisfy the assumptions of Lemma~\ref{girsanovfuku} (to be applied to $\xX$), that
$\cA^X[\varrho_{n,\e}-1]=\cL^X\varrho_{n,\e}$ and that
\begin{equation}\label{theb}
\sup_{n\geq 1}\sup_{x\in \cX} |u_{n,\e}(x)|< \infty \qquad 
\hbox{and} \qquad 
\sup_{n\geq 1} \sup_{x \in G_{\bK,\e}^{n}} |\cL^X\varrho_{n,\e}(x)|<\infty.
\end{equation}
First, $u_{n,\e} \in \cF^X$ because $u_{n,\e} \in C^\infty_c(\cX)$, and $|u_{n,\e}|$ is bounded,
uniformly in $n\geq 1$, because $\Gamma_{\bK,\e}^{n}$ is bounded by $1$ and vanishes outside $G_{\bK,\eta}$ 
(see Lemma~\ref{indiqXU}), while $u$ is smooth on $G_{\bK,\eta}$.
To show that $\cA^X[\varrho_{n,\e}-1]=\cL^X\varrho_{n,\e}$, it suffices by Remark~\ref{caracdomaine} to verify that
$\varrho_{n,\e}-1 \in C^\infty_c(\cX)$, which is clear, and that
$\sup_{\alpha \in (0,1]}\sup_{x \in (\rr^2)^N}|\cL_\alpha^X\varrho_{n,\e}(x)|<\infty$. We have
$$
\cL^X_\alpha \varrho_{n,\e}(x) = e^{u_{n,\e}(x)}\cL^X_\alpha u_{n,\e} (x) + \frac 12 e^{u_{n,\e}(x)} \|\nabla u_{n,\e} (x) \|^2.
$$
Since $u_{n,\e} \in C_c^\infty((\rr^2)^N)$, the only difficulty is to check that 
$\sup_{\alpha \in (0,1]}\sup_{x \in (\rr^2)^N}|\cL^X_\alpha u_{n,\e}(x)|<\infty$.
By \eqref{lproduit},
$$ 
\cL^X_\alpha u_{n,\e} (x) = \Gamma_{\bK,\e}^{n}(x)\cL^X_\alpha u (x)+u(x) \cL^X_\alpha\Gamma_{\bK,\e}^{n}(x)
+ \nabla \Gamma_{\bK,\e}^{n}(x)\cdot \nabla u(x).
$$
Again, the only difficulty consists of the first term, because $\cL^X_\alpha \Gamma_{\bK,\e}^{n}$ 
is uniformly bounded by Lemma~\ref{indiqXU} and vanishes outside $G_{\bK,\eta}$, while
$u$ is smooth on $G_{\bK,\eta}$.
Since Supp $\Gamma_{\bK,\e}^{n}\subset G_{\bK,\eta}$,
we are reduced to show
that $\sup_{\alpha \in (0,1]}\sup_{x \in G_{\bK,\eta}}|\cL^X_\alpha u(x)|<\infty$. But
$$
\cL^X_\alpha u= \frac12 \Delta u - \frac{\theta}{N}S_\alpha, \quad \hbox{where}\quad 
S_\alpha(x) =\sum _{1\le i, j \le N}  \frac{x^i-x^j}{\|x^i-x^j\|^2+\alpha}\cdot \nabla_{x^i}{u}(x),
$$
and we only have to verify that $\sup_{\alpha \in (0,1]}\sup_{x \in G_{\bK,\eta}}|S_\alpha(x)|<\infty$.

\vip

For $k \in K_1$ and $\ell \in K_2$, we have
\begin{align*}
\nabla_{x^k} {u} (x) = \sum _{j\in K_2}  \frac{ \theta}{N} \frac{x^{k}-x^{j}}{\|x^{k}-x^{j}\|^{2}}
\quad \hbox{and}\quad \nabla_{x^\ell} {u} (x) = \sum _{i\in K_1}  \frac{ \theta}{N} \frac{x^{\ell}-x^{i}}
{\|x^{\ell}-x^{i}\|^{2}}.
\end{align*}
Hence
$ S_{\alpha} =S_{1,\alpha}+S_{2,\alpha}+S_{3,\alpha}+S_{4,\alpha}$, where
\begin{align*}
S_{1,\alpha}(x)=& \frac \theta N \sum_{i,j \in K_1} \frac{x^{i}-x^{j}}{\|x^{i}-x^{j}\|^{2}+\alpha}\cdot \sum _{k \in K_2} 
\frac{x^{i}-x^{k}}{\|x^{i}-x^{k}\|^{2}} ,\\
S_{2,\alpha}(x)=&\frac \theta N\sum _{i \in K_2, j\in K_1 }
\frac{x^{i}-x^{j}}{\|x^{i}-x^{j}\|^{2}+\alpha}\cdot \sum _{k \in K_1} \frac{x^{i}-x^{k}}{\|x^{i}-x^{k}\|^{2}},
\end{align*}
and $S_{3,\alpha}$ (resp. $S_{4,\alpha}$) is defined as $S_{1,\alpha}$ (resp. $S_{2,\alpha}$) exchanging the roles of 
$K_1$ and $K_2$.
First, $S_{2,\alpha}$ (and $S_{4,\alpha}$) is obviously uniformly bounded on $G_{\bK,\eta}$.
Next, by symmetry,
$$ 
S_{1,\alpha}(x) = \frac \theta { 2N }\sum _{i,j \in K_1} \frac{x^{i}-x^{j}}{\|x^{i}-x^{j}\|^{2}+\alpha} 
\sum _{k\in K_2} \Big( \frac{x^{i}-x^{k}}{\|x^{i}-x^{k}\|^{2}} - \frac{x^{j}-x^{k}}{\|x^{j}-x^{k}\|^{2}}\Big).
$$
Moreover, there is $C_\eta >0$ such that for all $x\in G_{\bK,\eta}$, all
$i,j \in K_1$ such that $i\neq j$, all $k\in K_2$,
\begin{align*}
\Big\|\frac{x^{i}-x^{k}}{\|x^{i}-x^{k}\|^{2}} -\frac{x^{j}-x^{k}}{\|x^{j}-x^{k}\|^{2}}\Big\| \le C_{\eta } \|x^{i}-x^{j} \|,
\end{align*}
so that $S_{1,\alpha}$ (and $S_{3,\alpha}$) is bounded on $G_{\bK,\eta}$, uniformly in $\alpha\in (0,1]$, as desired.
\vip
Finally, the above computations, together with the facts that $\Gamma_{\bK,\e}^{n}=1$ on $G_{\bK,\e}^{n}$,
also show that for $x \in G_{\bK,\e}^{n}$,
$$
\cL^X \varrho_{n,\e}(x)=e^{u(x)}\Big(\frac 12 \Delta u(x)- \frac\theta N S_\alpha(x)\Big) 
+ \frac12 e^{u(x)}||\nabla u(x)||^2,
$$
which is bounded on $G_{\bK,\eta}$. Since $G_{\bK,\e}^{n}\subset G_{\bK,\eta}$, this
implies that $\sup_{n\geq 1} \sup_{x \in G_{\bK,\e}^{n}} |\cL^X\varrho_{n,\e}(x)|$ and completes the step.

\vip

{\it Step 3.} We apply Lemma~\ref{girsanovfuku} to the process $\xX$ with
$u_{n,\e}$ and $\varrho_{n,\e}$ defined in Step 2. Recalling that $\cA^X [\varrho_{n,\e}-1]=\cL^X \varrho_{n,\e}$
and using the conventions $\varrho_{n,\e}(\triangle)=1$ and  $\cL^X \varrho_{n,\e}(\triangle)=0$, we set
\begin{equation}\label{eqlne}
L^{n,\e}_t
=\frac{\varrho_{n,\e} (X_t)}{\varrho_{n,\e} (X_0)} 
\exp \Big( -\int _0 ^t \frac{\cL^X \varrho_{n,\e}(X_s)}{\varrho_{n,\e}(X_s)} \dd s \Big).
\end{equation}
Set $\cM^X_t=\sigma (\{ X_s, s\le t \})$.
By Lemma~\ref{girsanovfuku}, there is a family of probability measures 
$(\QQ^{n,\e}_x)_{x\in \cX\cup\{\triangle\}}$ such that
$$
\QQ^{n,\e} _x = L_t^{n,\e} \cdot \PP_x^X \quad \hbox{on} \quad \cM^X_t
$$  
for all $t\ge 0$ and quasi-everywhere in $\cX\cup\{\triangle\}$,
and such that
$$
\xX^{n,\e}  =\Big(\Omega^X, \cM^X, (X_t)_{t\ge 0}, (\QQ^{n,\e}_x)_{x\in \cX_\triangle}\Big)
$$ 
is a $ \varrho_{n,\e}^2  \mu $-symmetric $\cX\cup \{\triangle \}$-valued 
diffusion with regular  Dirichlet space  $(\cE^{n,\e},\cF^{n,\e})$ 
with core $C_c^\infty (\cX)$ such that for all $\varphi \in C_c^\infty (\cX)$, 
$$
\cE^{n,\e} (\varphi,\varphi) = \frac 12 \int_{(\rr^2)^N} ||\nabla \varphi||^2 \varrho_{n,\e}^2 \dd \mu. 
$$
Next, we apply Lemma~\ref{tuage} to $\xX^{n,\e}$ with the open set $G_{\bK,\e}^{n}$: the 
resulting killed process 
$$
\xX^{*,n,\e}=\Big(\Omega^X, \cM^X, (X_t^{*,n,\e})_{t\ge 0}, (\QQ^{n,\e}_x)_{x\in G_{\bK,\e}^{n}\cup \{\triangle \}}\Big)
$$ 
is a
$ \varrho_{n,\e}^2  \mu |_{G_{\bK,\e}^{n}}$-symmetric $G_{\bK,\e}^{n}\cup \{\triangle\}$-valued
diffusion with regular Dirichlet space
$({\cE}^{*,n,\e},{\cF}^{*,n,\e})$ with core $C_c^\infty (G_{\bK,\e}^{n})$ such that for all 
$\varphi \in C_c^\infty (G_{\bK,\e}^{n})$, 
$$
{\cE}^{*,n,\e} (\varphi,\varphi) 
= \frac 12 \int _{G_{\bK,\e}^{n}} ||\nabla \varphi||^2 \varrho_{n,\e}^2 \dd \mu.
$$
Comparing this Dirichlet space with the one found in Step 1, using that 
$\varrho_{n,\e}=\varrho$ on $G_{\bK,\e}^{n}$ and a uniqueness argument, see \cite[Theorem 4.2.8 p 167]{f},
we conclude that quasi-everywhere in $G_{\bK,\e}^{n}$, the law of $X^{*,n,\e}$ under $\QQ^{n,\e}_x$ equals 
the law of $Y^{n,\e}$ under $\PP^Y_x$.

\vip

{\it Step 4.} We fix $T>0$ and $\e\in (0,1]$ and complete the proof.
Since $\QQ^{n,\e}_x=L^{n,\e}_T\cdot \PP_x^X$ on $\cM^X_T$, we know from Step 3 that for all $n\geq 1$,
quasi-everywhere
in $G_{\bK,\e}^{n}$,
for all continuous bounded $\Phi : C([0,T],\cX_\triangle) \to \rr$, (observe 
that $\bar G_{\bK,\e}^{n}\subset \cX\subset \cX_\triangle$)
$$
\E_x^X[ \Phi(X_{\cdot \land \tau_{\bK,n,\e}\land T}) L^{n,\e}_T]=\E_x^Y[ \Phi(Y_{\cdot \land \tilde\tau_{\bK,n,\e}\land T})],
$$
where $\tau_{\bK,n,\e} = \inf \{t>0 : X_t \notin G_{\bK,\e}^{n}\}\land \tau_{\bK,\e}$
and $\tilde \tau_{\bK,n,\e} = \inf \{t>0 : Y_t \notin G_{\bK,\e}^{n}\}\land \tilde\tau_{\bK,\e}$.
Since $(L^{n,\e}_{t})_{t\ge 0}$ is a $\PP_x^X$-martingale by Lemma \ref{girsanovfuku}, we deduce that
quasi-everywhere in $G_{\bK,\e}^{n}$,
\begin{equation}\label{newform}
\E_x^X[ \Phi(X_{\cdot \land \tau_{\bK,n,\e}\land T}) L^{n,\e}_{\tau_{\bK,n,\e}\land T}]
=\E_x^Y[ \Phi(Y_{\cdot \land \tilde \tau_{\bK,n,\e}\land T})].
\end{equation}
Recall that $G_{\bK,\e} \subset \cup_{n\geq 1 }G_{\bK,\e}^{n}$,
see Lemma~\ref{indiqXU}. Hence $\lim_n \tau_{\bK,n,\e}=\tau_{\bK,\e}$, $\lim_n \tilde\tau_{\bK,n,\e}=\tilde\tau_{\bK,\e}$, 
and 
for each $x\in G_{\bK,\e}$, there is $n_x\geq 1$ such that $x \in G_{\bK,\e}^{n}$ for all $n\geq n_x$.
We  deduce from \eqref{newform} that quasi-everywhere in $G_{\bK,\e}$,
the process $(L^{n,\e}_{\tau_{\bK,n,\e}\land T })_{n\ge n_x}$ is a $(\cM_{\tau_{\bK,n,\e}\land T}^X)_{n\ge n_x}$-martingale
under $\PP_x^X$.
Moreover, recalling the expression \eqref{eqlne} of $L^{n,\e}$, that $\varrho_{n,\e}=\exp(u_{n,\e})$
and the bound \eqref{theb}, we conclude that
there is a constant $C_{T,\e,\bK}>0$ such that quasi-everywhere in  $G_{\bK,\e}$,
$$
\hbox{$\PP^X_x$-a.s., for all $n\geq n_x$,} \quad
C_{T,\e,\bK}^{-1} \leq L^{n,\e}_{\tau_{\bK,n,\e}\land T } \leq C_{T,\e,\bK}.
$$
Hence the martingale $(L^{n,\e}_{\tau_{\bK,n,\e}\land T })_{n\ge n_x}$
is closed by some $\cM_{\tau_{\bK,\e}\land T}$-measurable random variable $J_{T,\e,\bK}$
that satisfies $C_{T,\e,\bK}^{-1} \leq J_{T,\e,\bK} \leq C_{T,\e,\bK}$, and 
\eqref{newform} implies that for all $n\geq n_x$,
\begin{equation*}
\E_x^X[ \Phi(X_{\cdot \land \tau_{\bK,n,\e}\land T}) J_{T,\e,\bK}]=\E_x^Y[ \Phi(Y_{\cdot \land \tilde\tau_{\bK,n,\e}\land T})].
\end{equation*}
Letting $n\to \infty$, we find that
quasi-everywhere in  $G_{\bK,\e}$, for $\Phi \in C_b(C([0,T],\cX_\triangle),\rr)$,
\begin{equation*}
\E_x^X[ \Phi(X_{\cdot \land \tau_{\bK,\e}\land T}) J_{T,\e,\bK}]=\E_x^Y[ \Phi(Y_{\cdot \land \tilde \tau_{\bK,\e}\land T})].
\end{equation*}
Setting $\QQ^{T,\e,\bK}_x=J_{T,\e,\bK}\cdot \PP_x^X$ completes the proof. \bla
\end{proof}

\section{Explosion and continuity at explosion}\label{explacon}

In this section we consider a $QKS(\theta,N)$-process 
$\xX$ with life-time $\zeta$. 
We show that $\zeta=\infty$ when $\theta\in (0,2)$ and that $\zeta<\infty$ when $\theta\geq 2$.
In the latter case, we also prove that $\lim_{t\to \zeta-} X_t$ a.s. exists, 
{\it for the usual topology of $(\rr^2)^N$}: 
the Keller-Segel process is continuous at explosion. 
This is not clear at all at first sight: we know that $\lim_{t\to \zeta-} X_t=\triangle$ a.s. for
the one-point compactification topology, which
means that the process escapes
from every compact of $\cX$, but it could either go to infinity, which is not 
difficult to exclude, or it could tend to the boundary of $\cX$ without converging, 
e.g. because it could alternate very fast
between having its particles labeled in $\ig 1,k_0\id$ very close and having its 
 particles labeled in $\ig 2,k_0+1\id$ very close.
The goal of the section is to prove
the following result.

\begin{prop}\label{cont}
Fix $\theta >0$ and $N\geq 2$ such that $N>\theta$, set $k_0=\lceil 2N/\theta\rceil$
and $\cX=E_{k_0}$
and consider a $QKS(\theta,N)$-process 
$\xX =(\Omega^X, \cM^X, (X_t)_{t\ge 0}, (\PP_x^X)_{x\in \cX \cup \{ \triangle \}})$ with life-time $\zeta$.
\vip
(i) If $\theta < 2$, then \blue quasi-everywhere, \bla $\PP^X_x(\zeta = \infty)=1$.
\vip
(ii) If $\theta\geq 2$, then \blue quasi-everywhere, \bla $\PP^X_x$-a.s., $\zeta<\infty$
and $X_{\zeta-}=\lim_{t\to \zeta} X_t$ exists for the usual topology of $(\rr^2)^N$ and does not 
belong to $E_{k_0}$.
\end{prop}

 We first show that the process does not explode in the subcritical case and 
cannot go to infinity at explosion in the supercritical case.

\begin{lemma}\label{Xborne}
(i) If $\theta < 2$ and $N\ge 2$, then \blue quasi-everywhere, \bla $\PP^X_x(\zeta = \infty)=1$.
\vip
(ii) If $\theta \ge 2$ and $N> \theta$, then \blue quasi-everywhere, \bla
$$
\PP^X_x\Big(\zeta < \infty \hbox{ and } \blue \sup_{t\in [0, \zeta )}\bla  \|X_t\| < \infty\Big)=1.
$$
\end{lemma}

\begin{proof}
The arguments below only apply \blue quasi-everywhere, \bla since we use Proposition
\ref{deco}.
In both cases, we have for all $i\in \ig 1,N \id$ and all $t \in [0,\zeta)$,
$$
||X_t||^2\leq 2 \sum_{i=1}^N (\|X^i_t-S_{\ig 1,N \id}(X_t)\|^2 + \| S_{\ig 1,N \id}(X_t) \|^2)
=2 R_{\ig 1,N \id}(X_t) + 2N \| S_{\ig 1,N \id}(X_t) \|^2.
$$
 By Lemma~\ref{besbro}, there are a Brownian motion
$(M_t)_{t\geq 0}$ and a squared Bessel process $(D_t)_{t\geq 0}$ with dimension $d_{\theta ,N}(N)$
(killed when it gets out of $(0,\infty)$ if $d_{\theta ,N}(N)\leq 0$), such that
$S_{\ig 1,N \id}(X_t)=M_t$ and $R_{\ig 1,N \id}(X_t)=D_t$ for all $t \in [0,\zeta)$.
These processes being locally bounded, we conclude that
\begin{align}\label{processborne}
\mbox{ a.s., for all } T>0, \quad \sup _{t \in [0,\zeta \land T)}\|X_t\| < \infty .
\end{align}

(i) When $\theta <2$ and $N\ge 2$, we have $k_0=\lceil 2N/\theta\rceil > N$, so that $\cX = (\rr^2)^N$.
Hence on the event $\{\zeta<\infty\}$, we necessarily have $\limsup_{t\to\zeta-}||X_t||=\infty$,
and this is incompatible with \eqref{processborne} with $T=\zeta$.

\vip
(ii) When $N>\theta \ge 2$, we have $d_{\theta,N}(N)\leq 0$, so that   $(D_t)_{t\ge 0}$   is killed at some finite time $\tau$.
It holds that $\zeta\leq \tau$. Indeed, on the event where $\tau<\zeta$, we have
$R_{\ig 1,N\id}(X_{\tau})=\lim_{t\to\tau-} R_{\ig 1,N\id}(X_{t})=\lim_{t\to\tau-} D_t=0$, so that $X_{\tau} \notin E_{k_0}$
(since $k_0\leq N$), which is not possible since $\tau<\zeta$.
Hence $\zeta$ is also a.s. finite and it holds that  
$\blue \sup_{t\in[0, \zeta )}\bla \|X_t\|< \infty$ a.s. by \eqref{processborne} with the choice $T=\zeta$.
\end{proof}

To show the continuity at explosion in the supercritical case, we need to prove
the following \blue delicate \bla lemma.

\begin{lemma}\label{lemtech}
Assume that $N>\theta\geq 2$. \blue Quasi-everywhere, \bla
for all $K\subset \ig 1,N\id$ with $|K|\ge 2$,
$$
\PP^X_x\hbox{-\blue a.s.\bla },\qquad \lim_{t \to \zeta-} R_K(X_t) = 0 \quad \mbox{ or } \quad \liminf_{t\to \zeta-} R_K(X_t) >0. 
$$
\end{lemma}

\begin{proof} 
We proceed by reverse induction on the cardinal of $K$. If first $K = \ig 1,N \id$, the result is clear 
because $(R_{\ig 1,N \id}(X_t))_{t\in [0,\zeta)}$ is a (killed) squared Bessel process on $[0,\zeta)$ by Lemma
\ref{besbro} (and since $\zeta \leq \tau$ exactly as in the proof of Lemma~\ref{Xborne}-(ii)),   
hence it has a limit in $\rr_+$ as $t\to\zeta$.
Then, we assume that the property is proved if $|K| \ge n$ where $n \in \ig 3,N \id$, we take 
$K \subset \ig 1,N \id$ such that $|K| = n-1$ and we show in several steps that a.s.,
either $\lim_{t\to \zeta-} R_K(X_t)=0$ or $\liminf_{t\to \zeta-} R_K(X_t)>0$.

\vip

{\it Step 1.} We fix $\e\in (0,1]$ and introduce $\tilde \sigma^\e_0=0$ and, for
$k\geq 1$,
$$
\sigma^\e_k=\inf\{t\in (\tilde\sigma^\e_{k-1},\zeta) : R_K(X_t)\leq \e\} \quad \hbox{and} 
\quad \tilde \sigma^\e_k=\inf\{t\in (\sigma^\e_k,\zeta) : R_K(X_t)\geq 2\e\},
$$
with the convention that $\inf \emptyset = \zeta$.
We show in this step that for all deterministic $A>0$,
there exists a constant $p_{A,\e}>0$ such that for all $k\geq 1$, \blue quasi-everywhere,
on $\{\sigma_k^\e<\zeta\}$,
\begin{align*}
\PP^X_x\Big(\{\tilde\sigma^\e_k\geq (\sigma^\e_k+A)\land \zeta \}\cup B_{k,\e} \Big\vert \cM^X_{\sigma^\e_k}\Big)
\geq p_{A,\e},
\end{align*}
where $\cM^X_t=\sigma(X_s : s\in [0,t])$, and where, setting $a_\e=c_{|K|+1} \e / c_{|K|}$ 
(recall Lemma~\ref{campingcar}),
$$
B_{k,\e}=\Big\{\blue \sup_{t\in[\sigma^\e_k,\tilde \sigma^\e_k)} ||X_t|| \geq 1/\e\; \hbox{ or } \;
\inf_{t\in[\sigma^\e_k,\tilde \sigma^\e_k)}
\min_{i\notin K} R_{K\cup\{i\}} (X_t)\leq a_\e\Big\}.
$$

By the strong Markov property of $\xX$, on $\{\sigma_k^\e<\zeta\}$,
$$
\PP^X_x\Big(\{\tilde\sigma^\e_k\geq (\sigma^\e_k+A)\land \zeta \}\cup B_{k,\e} \Big\vert \cM^X_{\sigma^\e_k}\Big)
=g(X_{\sigma _k^\e}),
$$
where
$$
g(y)=
\PP^X_{y}\Big(\{\tilde\sigma^\e_1\geq(\sigma^\e_1+ A) \land \zeta \}\cup B_{1,\e}\Big)=
\PP^X_{y}\Big(\{\tilde\sigma^\e_1\geq A \land \zeta \}\cup C_{1,\e}\Big)
$$
and 
$$
C_{1,\e}=\Big\{\blue \sup_{t\in[0,\tilde \sigma^\e_1)} ||X_t|| \geq 1/\e\;\; \hbox{ or } \;
\inf_{t\in[0,\tilde \sigma^\e_1)}
\min_{i\notin K} R_{K\cup\{i\}} (X_t)\leq a_\e\Big\}.
$$
We used that  $R_K(X_{\sigma_k^\e}) \leq \e$ on $\{\sigma_k^\e<\zeta\}$ by definition of $\sigma_k^\e$, 
so that $\sigma^\e_1=0$ under $\PP^X_{X_{\sigma_k^\e}}$.
Using again that $R_K(X_{\sigma_k^\e})\leq \e$ on $\{\sigma_k^\e<\zeta\}$,
it suffices to show that there is a constant $p_{A,\e}>0$ such that $g(y)\geq p_{A,\e}$ 
quasi-everywhere in $\{y \in \cX : R_K(y)\leq \e\}$.

\vip

If first $||y||\geq 1/\e$ or $\min_{i\notin K} R_{K\cup\{i\}} (y)\leq a_\e$, then clearly, $g(y)=1$. 

\vip

Otherwise,
$y \in G_{\bK,\e}$, where
$$
G_{\bK,\e}=\{ x \in \cX : \mbox{ for all } i\in K, \mbox{ all } j\notin K, \;
\|x^i-x^j\|^2 > \e \} \cap B(0,1/\e) 
$$ 
as in Proposition~\ref{girsanov} with $\bK=(K,K^c)$, because $||y|| < 1/\e$ and because 
$R_K(y)\leq \e <2\e$ and $\min_{i\notin K} R_{K\cup\{i\}} (y)> a_\e=c_{|K|+1} \e /c_{|K|}$ imply that
$||x^i-x^k||^2 > \e$ for all $i\in K$, $j\notin K$ by Lemma~\ref{campingcar}.
For the very same reasons and by definition of $\tilde\sigma_1^\e$, it holds that 
\begin{equation}\label{nn1}
C_{1,\e}^c \subset \{\hbox{for all } t\in [0,\tilde\sigma_1^\e), \;\; X_t \in G_{\bK,\e}\}.
\end{equation}

We now apply Proposition~\ref{girsanov} with $T=A$ (and $\e$) and we find that
quasi-everywhere in $G_{\bK,\e}$,
\begin{align}
g(y)\geq& C_{A,\e,\bK}^{-1}\QQ^{A,\e,\bK}_{y}(\{\tilde\sigma^\e_1\geq A \land \zeta \}\cup C_{1,\e})\notag\\
=&C_{A,\e,\bK}^{-1}\QQ^{A,\e,\bK}_{y}(\{\tilde\sigma^\e_1\geq A \land \zeta \}\cap C_{1,\e}^c)\label{nn2}
+C_{A,\e,\bK}^{-1}\QQ^{A,\e,\bK}_{y}(C_{1,\e}).
\end{align}
But we know from Proposition~\ref{girsanov} and Lemma~\ref{besbro}
that under $\QQ^{A,\e,\bK}_{y}$,  
$(R_K(X_t))_{t\in[0,\tau_{K,\e} \land A]}$ is a squared Bessel process
with dimension $d_{\theta |K| /N, |K|}(|K|) = d_{\theta ,N}(|K|)$, issued from $R_K(y)\leq\e$,
stopped at time $\tau_{\bK,\e} \land A$, where $\tau_{\bK,\e}=\inf\{t>0 : X_t \notin G_{\bK,\e}\}$. Hence there exists,
under $\QQ^{A,\e,\bK}_{y}$, a squared Bessel process $(S_t)_{t\ge 0}$ with dimension
$d_{\theta ,N}(|K|)$ such that $S_t = R_K(X_t)$ for all $t\in  [0,\tau_{\bK,\e}\land A]$.
We introduce $\kappa_\e=\inf\{t>0 : S_t \geq 2\e\}$ and we observe that
\begin{align*}
\{\kappa_\e \geq A \land \zeta\}\cap C_{1,\e}^c
= \{\tilde\sigma^\e_1 \geq A \}\cap C_{1,\e}^c.
\end{align*}
Indeed, we used that on $C_{1,\e}^c$, we have $\tau_{\bK,\e}\geq \tilde \sigma^\e_1$ by \eqref{nn1}
so that $R_K(X_t)=S_t$ for all $t\in [0,\tilde\sigma^\e_1\land A)$, from which we conclude 
that $\kappa_\e\geq A \land \zeta$ if and only $\tilde \sigma^\e_1 \geq A \land \zeta$. Coming back to
\eqref{nn2}, we get
\begin{align*}
g(y)\geq&C_{A,\e,\bK}^{-1}\QQ^{A,\e,\bK}_{y}(\{\kappa_\e\geq A \land \zeta \}\cap C_{1,\e}^c)
+C_{A,\e,\bK}^{-1}\QQ^{A,\e,\bK}_{y}(C_{1,\e})= C_{A,\e,\bK}^{-1} \QQ^{A,\e,\bK}_{y}(\kappa_\e\geq A \land \zeta).
\end{align*}
The step is complete, since $\QQ^{A,\e,\bK}_{y}(\kappa_\e\geq A)$ is the probability that a squared
Bessel process with dimension $d_{\theta ,N}(|K|)$ issued from $R_K(y)\leq \e$ remains below $2\e$ during $[0,A]$
and is thus strictly positive, uniformly in $y$ (such that $y\in G_{\bK,\e}$ and $R_K(y)\leq \e$). \bla

\vip

{\it Step 2.} We prove here that for all $\e\in (0,1]$, all $A >0$, \blue quasi-everywhere, \bla
$$
\PP^X_x\Big(\limsup_{t\to \zeta-}||X_t||\geq 1/\e  \;\; \hbox{ or }\;\; 
\liminf_{t\to \zeta-}\min_{i\notin K} R_{K\cup\{i\}} (X_t)\leq a_\e
\;\; \hbox{ or }\;\; \exists \; k\geq 1,\; \sigma^\e_k\geq \zeta\land A
\Big) =1.
$$
All the arguments below only hold  \blue quasi-everywhere, \bla even if we do not mention it explicitly
during this step. For $k\geq 1$, we introduce, \blue with $B_{k,\e}$ defined in Step 1,
$$
\Omega_{k+1}=\{\sigma^\e_{k+1}< \zeta\land A\}\cap B_{k,\e}^c
$$
\bla and we first show that $\PP^X_x(\liminf_k \Omega_k)=0$. To this end, it suffices to check that for all 
$\ell\geq 1$, $\PP^X_x(\cap_{k=\ell}^{\infty} \Omega_k)=0$.
Since $\Omega_k$ is $\cM_{\sigma^\e_{k}}$-measurable, for all $m\geq \ell\geq 1$,
$$
\PP^X_x(\cap_{k=\ell}^{m+1} \Omega_k)=\E^X_x[\indiq_{\cap_{k=\ell}^{m} \Omega_k}
\PP^X_x(\Omega_{m+1} |\cM_{\sigma^\e_{m}})].
$$
Since moreover $\cap_{k=\ell}^{m} \Omega_k\subset \{\sigma^\e_{m}<\zeta\}$
and since $\sigma^\e_{m+1}\geq \tilde\sigma^\e_m \geq \tilde\sigma^\e_m-\sigma_m^\e$,
we deduce that
on $\cap_{k=\ell}^{m} \Omega_k$, \blue
\begin{align*}
\PP^X_x(\Omega_{m+1} |\cM_{\sigma^\e_{m}}) =&
1- \PP^X_x(\{\sigma^\e_{m+1} \geq \zeta \land A \}\cup B_{m,\e} |\cM_{\sigma^\e_{m}})\\
\leq& 1-\PP^X_x(\{\tilde\sigma^\e_m\geq (\sigma^\e_m+A)\land \zeta \}\cup B_{m,\e} |\cM_{\sigma^\e_{m}}),
\end{align*}
\bla so that $\PP^X_x(\Omega_{m+1} |\cM_{\sigma^\e_{m}})\leq 1-p_{A,\e}$ by Step 1.
Hence we conclude that 
$$
\PP^X_x(\cap_{k=\ell}^{m+1} \Omega_k)
\leq (1-p_{A,\e})\PP^X_x(\cap_{k=\ell}^{m} \Omega_k)
$$  for all $m\geq \ell\geq 1$, so that
$\PP^X_x(\cap_{k=\ell}^{\infty} \Omega_k)=0$ as desired.
\vip

Hence $\PP^X_x(\liminf_k \Omega_k)=0$, so that a.s., an infinite number of $\Omega_k^c$ are realized.
\blue Recalling that 
$$
\Omega_{k+1}^c=\Big\{\sigma^\e_{k+1}\geq \zeta\land A\;\; \hbox{ or }\;\; \inf_{t\in[\sigma^\e_k,\tilde \sigma^\e_k)}
\min_{i\notin K} R_{K\cup\{i\}} (X_t)\leq a_\e\;\; \hbox{ or }\;\; \sup_{t\in[\sigma^\e_k,\tilde \sigma^\e_k)}||X_t||\geq 1/\e\Big\},
$$
we find the following alternative: \bla
\vip

\noindent $\bullet$ either there is $k\geq 1$ such that $\sigma^\e_{k}\geq \zeta \land A$;
\vip
\noindent $\bullet$ or for all $k\geq 1$, $\sigma^\e_{k}< \zeta$
and $\blue \inf_{t\in[\sigma^\e_k,\tilde \sigma^\e_k)} \bla \min_{i\notin K} R_{K\cup\{i\}} (X_t)\leq a_\e$ for infinitely many $k$'s, 
which implies
that $\liminf_{t\to\zeta-} \min_{i\notin K} R_{K\cup\{i\}} (X_t)\leq a_\e$ because necessarily,
$\lim_{\infty} \sigma^\e_{k}=\zeta$ by definition of the sequence $(\sigma_k^\e)_{k\geq 1}$ and by continuity
of $t\to R_K(X_t)$ on $[0,\zeta)$;
\vip
\noindent $\bullet$ or for all $k\geq 1$, $\sigma^\e_{k}< \zeta$ and there are infinitely many $k$'s for which 
$\sup_{t\in[\sigma^\e_k,\tilde \sigma^\e_k)}||X_t||\geq 1/\e$ and this implies that $\limsup_{t\to\zeta-} ||X_t|| \geq 1/\e$,
because $\lim_{\infty} \sigma^\e_{k}=\zeta$ as previously.

\vip

{\it Step 3.} We conclude. Applying Step 2,
we find that  \blue quasi-everywhere, $\PP^X_x$-a.s., for all 
$A \in \nn$ and all $\e \in \QQ\cap(0,1]$,\bla
$$
\limsup_{t\to \zeta-}||X_t||\geq 1/\e  \;\; \hbox{ or }\;\; 
\liminf_{t\to \zeta-}\min_{i\notin K} R_{K\cup\{i\}} (X_t)\leq a_\e
\;\; \hbox{ or }\;\; \exists \; k\geq 1,\; \sigma^\e_k\geq \zeta\land A.
$$
By Lemma~\ref{Xborne}-(ii), we know that $\zeta<\infty$, so that choosing $A=\lceil\zeta\rceil$, 
we conclude that  \blue quasi-everywhere, \bla $\PP^X_x$-a.s., for all $\e\in\QQ\cap(0,1]$
\begin{equation}\label{nbb}
\limsup_{t\to \zeta-}||X_t||\geq 1/\e  \;\; \hbox{ or }\;\; 
\liminf_{t\to \zeta-}\min_{i\notin K} R_{K\cup\{i\}} (X_t)\leq a_\e
\;\; \hbox{ or }\;\; \exists \; k\geq 1, \; \sigma^\e_k = \zeta.
\end{equation}
And by Lemma~\ref{Xborne}-(ii) again,
$\limsup_{t\to \zeta-}||X_t||\leq 1/\e_0$ for some (random) $\e_0\in (0,1]$.

\vip

On the event where $\liminf_{t\to \zeta-}\min_{i\notin K} R_{K\cup\{i\}} (X_t)=0$, there exists some (random)
$i_0\notin K$ such that $\liminf_{t\to \zeta-} R_{K\cup\{i_0\}} (X_t)=0$, whence 
$\lim_{t\to \zeta-} R_{K\cup\{i_0\}} (X_t)=0$ by induction assumption, and this obviously implies that
$\lim_{t\to \zeta-} R_{K} (X_t)=0$.

\vip

On the complementary event, we fix
$\e_1 \in (0,\e_0]$ such that $\liminf_{t\to \zeta-}\min_{i\notin K} R_{K\cup\{i\}} (X_t)> a_{\e_1}$
and we conclude from \eqref{nbb} and the fact that $\limsup_{t\to \zeta-}||X_t||\leq 1/\e_0$
that for all \blue $\e\in\QQ\cap(0,\e_1]$, \bla there exists $k_\e \geq 1$ such that
$\sigma_{k_\e}^\e = \zeta$. Recalling the definition of $(\sigma^\e_k)_{k\geq 1}$, 
we deduce that for all \blue $\e\in \QQ\cap(0,\e_1]$, \bla
$R_K(X_t)$  upcrosses the segment $[\e,2\e]$ a finite number of times
during $[0,\zeta)$. Hence for all \blue $\e\in (0,\e_1]\cap \QQ$, \bla there exists $t_\e\in [0,\zeta)$
such that either $R_K(X_t)>\e$ for all $t \in [t_\e,\zeta)$ or $R_K(X_t)<2\e$ for all $t \in [t_\e,\zeta)$.
If there is \blue $\e\in \QQ\cap(0,\e_1]$ \bla such that $R_K(X_t)>\e$ for all $t \in [t_\e,\zeta)$, then
$\liminf_{t\to\zeta-} R_K(X_t) \geq \e>0$. If next for all \blue $\e\in \QQ\cap(0,\e_1]$, \bla we have
$R_K(X_t)<2\e$ for all $t \in [t_\e,\zeta)$, then $\lim_{t\to\zeta-} R_K(X_t)=0$.

\vip
Hence in any case, we have either 
$\lim_{t\to\zeta-} R_K(X_t)=0$ or $\liminf_{t\to\zeta-} R_K(X_t)>0$.
\end{proof}

We finally give the

\begin{proof}[Proof of Proposition~\ref{cont}]
Point (i), which concerns the subcritical case, has already been checked in Lemma~\ref{Xborne}-(i).
Concerning point (ii), which concerns the supercritical case $\theta\geq 2$, we already know that
\blue quasi-everywhere, \bla $\PP^X_x(\zeta <\infty)=1$ by Lemma~\ref{Xborne}-(ii), and it remains to prove that 
$\PP^X_x$-a.s., 
$\lim _{t \rightarrow \zeta- } X_t$ exists and does not belong to $E_{k_0}$. We divide the proof in \blue four steps.

\vip 

{\it Step 1.} For $\bK=(K_p)_{p\in \ig1,\ell\id}$ a partition of $\ig 1,N \id$ and $\e\in (0,1]$,
we consider as in Proposition~\ref{girsanov}
$$
G_{\bK,\e} = \Big\{ x \in \cX : \min_{1\leq p\neq q \leq \ell} \;\;
\min _{i\in K_p, j\in K_q}\|x^i-x^j\|^2 > \e \Big\} \cap B\Big(0,\frac 1\e\Big)
$$ 
and $\tau_{\bK,\e}=\inf\{t\geq 0 : X_t \notin G_{\bK,\e}\}\in [0,\zeta]$. We show here for each $T>0$,
quasi-everywhere in $G_{\bK,\e}$, $\PP^X_x$-a.s., for all $T>0$, all $p\in \ig 1,\ell \id$,  $S_{K_p}(X_t)$ 
has a limit in $\rr^2$ as $t\to (\tau_{\bK,\e}\land T)-$.

\vip

If $\ell=1$, the result is obvious since $S_{\ig 1,N\id}(X_t)$ is a Brownian motion during $[0,\zeta)$
by Lemma~\ref{besbro}.
If next $\ell\geq 2$, Proposition~\ref{girsanov} and Lemma~\ref{besbro} tell us that 
under $\QQ_x^{T,\e,\bK}$, which is equivalent 
to $\PP^X_x$, the processes $S_{K_p}(X_t)$ are some Brownian motions on $[0,\tau_{\bK,\e}\land T)$,
and thus have some limits as $t\to (\tau_{\bK,\e}\land T)-$.

\vip

{\it Step 2.} 
For  $\e\in (0,1]$ and $\bK=(K_p)_{p\in \ig 1,\ell\id}$ a partition of $\ig 1,N \id$, 
we set $\tilde{\eta}_0^{\bK,\e}=0$ and, for $k\ge 0$,
$$
\eta _{k+1}^{\bK,\e} = \inf \{  t\ge \tilde{\eta}_k^{\bK,\e} : X_t \in G_{\bK,2\e} \}\quad \mbox{ and } 
\quad \tilde{\eta }_{k+1}^{\bK,\e} = \inf \{  t\ge \eta _{k+1} ^{\bK,\e} : X_t \notin G_{\bK,\e}\},
$$
with the convention that $\inf \emptyset=\zeta$. 
Using Step 1 and the strong Markov property, we conclude that
quasi-everywhere, $\PP^X_x$-a.s., for all $\e\in (0,1]\cap \QQ$, 
all $k\geq 1$, all $T \in \nn_+$, on $\{\eta_{k}^{\bK,\e}<\zeta\}$, for all $p\in \ig 1,\ell \id$,  $S_{K_p}(X_t)$ 
admits a limit in $\rr^2$ as  $t$ goes to $(\tilde\eta^{\bK,\e}_k\land T)-$.
Choosing $T = \lceil \zeta \rceil$, we conclude that 
quasi-everywhere, $\PP^X_x$-a.s., on $\{\eta_{k}^{\bK,\e}<\zeta\}$, for all $\e\in (0,1]\cap \QQ$, 
all $k\geq 1$, all $p\in \ig 1,\ell \id$,
\begin{equation*}
S_{K_p}(X_t) \mbox{ admits a limit in $\rr^2$ as } t \mbox{ goes to } \tilde\eta^{\bK,\e}_k-.
\end{equation*}

{\it Step 3.} We now check that quasi-everywhere, $\PP^X_x$-a.s., 
there is a partition $\bK=(K_p)_{p\in \ig 1,\ell\id}$ of $\ig 1,N \id$, some $\e \in (0,1]\cap \QQ$ and some $k\geq 1$ 
such that (i) $\eta_{k}^{\bK,\e}<\zeta$ and $\tilde\eta^{\bK,\e}_k=\zeta$ and (ii) for all $p \in \ig 1,\ell\id$,
$\lim_{t\to \zeta-}R_{K_p}(X_t)=0$.

\vip

By Lemma~\ref{lemtech}, we know that for all $K \subset \ig1,N\id$, we have the alternative
$\lim_{t\to\zeta-}R_K(X_t)=0$ or $\liminf_{t\to\zeta-}R_K(X_t)>0$. Hence
the partition $\bK=(K_p)_{p\in \ig 1,\ell\id}$ of $\ig 1,N \id$ consisting of the classes
of the equivalence relation defined
by $i \sim j$ if and only if $\lim_{t\to \zeta} R_{\{i,j\}}(X_t)=0$ satisfies that for all $p\in\ig 1,\ell\id$,
$\lim_{t\to\zeta-}R_{K_p}(X_t)=0$ and $\liminf_{t\to\zeta-} \min_{i\notin K_p}R_{K_p\cup\{i\}}(X_t)>0$.

\vip
Using moreover that $\limsup_{t\to \zeta-} ||X_t||<\infty$ according to Lemma~\ref{Xborne}, we deduce 
that there is $\alpha \in (0,\zeta)$ and $\e\in (0,1]\cap\QQ$ such that for all $t \in [\alpha,\zeta)$, 
$X_t$ belongs to $G_{\bK,2\e}$. Finally, we consider $k=\max\{m\geq 1 : \eta^{\bK,\e}_m\leq \alpha\}$,
which is finite by continuity of $t\mapsto X_t$ on $[0,\alpha]$, and
it holds that $\eta^{\bK,\e}_k\leq \alpha<\zeta$ and that $\tilde \eta^{\bK,\e}_k=\zeta$.

\vip

{\it Step 4.} We consider the (random) partition $\bK=(K_p)_{p\in \ig 1,\ell\id}$ introduced in Step 3.
By Step 2 and since $\eta_{k}^{\bK,\e}<\zeta$ and $\tilde\eta^{\bK,\e}_k=\zeta$, we know that
quasi-everywhere, $\PP^X_x$-a.s., 
for all $p\in \ig1,\ell\id$, $M_p=\lim_{t\to \zeta-} S_{K_p}(X_t)$ exists in $\rr^2$. By Step
3, we know that for all $p\in \ig1,\ell\id$, $\lim_{t\to \zeta} R_{K_p}(X_t)=0$.
We easily conclude that quasi-everywhere, $\PP^X_x$-a.s., 
for all $p \in \ig 1, \ell\id$, all $i \in K_p$, $\lim_{t\to \zeta-} X^i_t=M_p$.
This shows that quasi-everywhere, $\PP^X_x$-a.s., $X_{\zeta-}=\lim_{t\to \zeta-}X_t$
exists in $(\rr^2)^N$. Moreover, $X_{\zeta-}$ cannot belong to $\cX=E_{k_0}$,
because $\lim_{t\to \zeta-}X_t=\triangle$ when $E_{k_0}\cup\{\triangle\}$ is endowed with
the one-point compactification topology, see Subsection~\ref{ap1}.
\bla
\end{proof}

\section{Some special cases}\label{specialc}

During a $K$-collision, the particles labeled in $K$ are isolated from the other ones.
Thanks to Proposition~\ref{girsanov}, it will thus be possible to describe what happens in
a neighborhood of the instant of this $K$-collision, by studying a $QKS(\theta |K| / N, |K|)$-process.
In other words, we may assume that $|K|=N$, so that the following special cases,
which are the purpose of this section, will be crucial.

\begin{prop}\label{N-1toN} Let $N\geq 4$ and $\theta>0$ such that $N>\theta$. Consider a $QKS(\theta,N)$-process
$\xX$ as in Proposition~\ref{existenceXU}. Recall that $\zeta=\inf\{t\geq 0 : X_t=\triangle\}$
and set $\tau=\inf \{t \geq 0 : R_{\ig 1,N \id}(X_t) \notin (0,\infty)\}$
with the convention that $R_K(\triangle)=0$, so that $\tau\in [0,\zeta]$.

\vip

(i) If $d_{\theta,N}(N-1) \le 0$ and $d_{\theta,N}(N)<2$, then \blue quasi-everywhere, \bla
$$
\PP^X_x \Big(\inf _{t \in [0,\zeta)}R_{\ig 1,N \id}(X_t) >0 \Big)=1.
$$

(ii) If $d_{\theta,N}(N-1) \in (0,2)$ and $d_{\theta,N}(N)<2$, then \blue quasi-everywhere, \bla $\PP^X_x$-a.s,
for all $K\subset \ig 1,N\id $ with cardinal $|K|=N-1$,
there is $t \in [0,\tau)$ such that $R_K(X_t)=0$.

\vip 

(iii) If $0<d_{\theta,N}(N)<2\le d_{\theta,N}(N-1)$, then \blue quasi-everywhere, \bla 
$\PP^X_x$-a.s,
for all $K\subset \ig 1,N \id$ with cardinal $|K|=2$, 
there is $t \in [0,\tau)$ such that $R_K(X_t)=0$.
\end{prop}


The proof of this proposition is very long. First, we recall some notation
about the decomposition of $\xX$ obtained in Proposition~\ref{deco} and we study the involved
time-change. We then derive a formula describing
$R_K(U_t)$, valid on certain time intervals, for any $K \subset \ig 1,N\id$. This
formula is of course not closed, but it
allows us to compare 
$R_K(U_t)$, when it is close to $0$, to some process resembling a squared Bessel process,
of which one easily studies the behavior near $0$.
Finally, we prove Proposition~\ref{N-1toN}, unifying a little points (i) and (ii)
and treating separately point (iii).

\subsection{Notation and preliminaries}\label{rrrr}
We recall the decomposition of 
Proposition~\ref{deco}, which holds true \blue quasi-everywhere in $\cX\cap E_N$. \bla
Consider a Brownian motion $(M_t)_{t\geq 0}$ with diffusion coefficient $N^{-1/2}$ starting from $S_{\ig 1,N\id}(x)$, 
a squared Bessel process 
$(D_t)_{t\geq 0}$ starting from $R_{\ig 1,N\id}(x)>0$ killed when leaving $(0,\infty)$ with life-time
$\tau_D=\inf\{t \geq 0 : D_t=\triangle\}$
and a  $QSKS(\theta,N)$   -process $(U_t)_{t\geq 0}$ starting from $\Phi_\sS(x)$ with life-time
$\xi=\inf\{t \geq 0 : U_t=\triangle\}$, all these processes being independent.
For $t\in [0,\tau_D)$, we put   $A_t=\int_0^{t} \frac{\dd s}{D_s}$.  
We also consider the inverse $\rho:[0,A_{\tau_D})\to [0,\tau_D)$ of
$A$.

\begin{lemma}\label{explosionalinfini}
If $d_{\theta,N}(N)<2$, then $\tau_D <\infty$ and $A_{\tau_D}=\infty$ a.s.
\end{lemma}

\begin{proof}
Since   $(D_t)_{t\ge 0}$   is a (killed) squared Bessel process with dimension $d_{\theta,N}(N)<2$, 
we have $\tau_D < \infty$ a.s according to Revuz-Yor \cite[Chapter XI]{ry}. 
Moreover, 
there is a Brownian motion $(B_t)_{t\geq 0}$ such that $D_t=r+2\intot \sqrt{D_s}\dd B_s + d_{\theta,N}(N) t$ 
for all $t\in [0,\tau_D)$, where $r=R_{\ig 1,N\id}(x)>0$. A simple computation shows the existence of a Brownian 
motion $(W_t)_{t\ge 0}$ such that for all $t\in [0,A_{\tau_D})$,
$$ 
D_{\rho_t} = r+ 2 \int _0 ^t D_{\rho _s}\dd W_s + d_{\theta,N}(N) \int _0 ^t D_{\rho _s}\dd s.
$$
Hence for all $t\in [0, A_{\tau_D})$, $D_{\rho _t} = r \exp ( 2W_t + (d_{\theta,N}(N) -2)t)$.
On the event where $A_{\tau_D} < \infty$, we have $0=D_{\tau_D-}=\lim_{t\to A_{\tau_D}} D_{\rho_t}=
\exp ( 2W_{A_{\tau_D}} + (d_{\theta,N}(N) -2)A_{\tau_D} )>0$.
Hence $A_{\tau_D}=\infty$ a.s.
\end{proof}

From now on, we assume that $d_{\theta,N}(N)<2$. Hence $A:[0,\tau_D)\to[0,\infty)$ is an increasing bijection,
as well as $\rho:[0,\infty)\to [0,\tau_D)$.
By Proposition~\ref{deco}, \blue quasi-everywhere in $\cX\cap E_N$ \bla,
we can find a triple   $(M_t,D_t,U_t)_{t\ge 0}$   as above such that
for $\xX$ our $QKS(\theta,N)$ process starting from $x$, 
for all $t\in [0,\tau_D \land \rho_\xi)$, and actually for all $t\in [0,\rho_\xi)$
because $\rho_\xi\leq \tau_D$ since $\rho$ is $[0,\tau_D)$-valued,
$$
X_t= \Psi(M_t,D_t,U_{A_t}), \quad \hbox{i.e.} \quad M_t=S_{\ig 1, N\id}(X_t), \quad D_t=R_{\ig 1, N\id}(X_t)
\quad \hbox{and} \quad U_{A_t}=\Phi_\sS(X_t).
$$
We recall that $\Psi(m,r,u)=\gamma(m)+\sqrt r u$ if $(m,r,u)\in \rr^2\times (0,\infty)\times\cU$
and $\Psi(m,r,u)=\triangle$ if $(m,r,u)=\triangle$.
Observe that $\tau=\tau_D\land \rho_\xi=\rho_\xi$, where
$\tau=\inf\{t \geq 0 : R_{\ig 1,N\id}(X_t)\notin (0,\infty)\} \in [0,\zeta]$.

\vip

We note that if $\xi<\infty$, then $\rho_\xi<\tau_D$, because
$\rho$ is an increasing bijection from $[0,\infty)$ into $[0,\tau_D)$. 
 Hence, still if $\xi<\infty$, then 
  $X$   explodes at time $\rho_\xi$ strictly before $\tau_D$, whence
\begin{equation}\label{etoile3}
\{\xi<\infty\}  \subset \Big\{\inf_{t\in [0,\zeta)} R_{\ig 1,N\id}(X_t)>0\Big\}.
\end{equation}

Finally note that since   $U$   is $\sS$-valued, it cannot have a $\ig 1,N\id$-collision.
But for any $K \subset \ig 1,N \id$ with cardinal $|K|\leq N-1$, it holds that
\begin{gather}\label{etoile22}
\hbox{$  U   $ has a $K$-collision at $t \in [0,\xi)$ if and only if $  X  $ has a 
$K$-collision at $\rho_t \in [0,\tau)$},
\end{gather}
which follows from the facts that 

\vip

$\bullet$ for all 
$(m,r,u)\in \rr^2\times (0,\infty)\times\cU$, 
$R_K(\Psi(m,r,u))=0$ if and only if $R_K(u)=0$;

\vip

$\bullet$ $\rho$ is an increasing bijection from $[0,\xi)$ into $[0,\tau)$, because $\rho_\xi=\tau$.

\vip

We conclude this subsection with a remark about the \blue quasi-everywhere \bla notions of $\xX$
and $\uU$, in the case where they are related as above.
See Subsection~\ref{ap1} for a short reminder on this notion.

\begin{rk}\label{UtoX}
Fix $B\in \cM^U$ such that $\PP^U_u(B)=1$ \blue quasi-everywhere \bla (here \blue quasi-everywhere \bla 
refers to the Hunt process
$\uU$).
Then  $\PP^U_{\Phi_\sS(x)}(B)=1$ \blue quasi-everywhere \bla 
(here \blue quasi-everywhere \bla refers to the Hunt process $\xX^*$, which is $\xX$ killed when it 
gets outside $E_N$).
\end{rk}

\begin{proof}
By definition,
there exists $\cN^U$ a properly exceptional set relative to $\uU$ such that for all 
$u\in \cU \setminus \cN^U$, $\PP^U_u(B) = 1$. Thus  
for all $x \in \Phi_\sS ^{-1} (\cU \setminus \cN^U) $, $\PP^U_{\Phi_\sS (x)} ( B ) = 1$.

\vip

By Proposition~\ref{deco}, there exists $\cN^X$ a properly exceptional set relative to 
$\xX^*$, such that for all $x \in (\cX\cap E_N) \setminus \cN^X$, the law of $(X_t)_{t\ge 0}$ under $\PP^X_x$
is equal to the the law of $(Y_t=\Psi(M_t,D_t,U_{A_t}))_{t\ge 0}$ under 
$\QQ^Y_x=\PP^M_{\pi_{H^\perp}(x)}\otimes \PP^D_{\|\pi_H(x)\|^2} \otimes \PP^U_{\Phi_\sS (x)}$,
with some obvious notation.

\vip

Hence we only have to prove that $\cN=\Phi_\sS ^{-1} (\cN^U)\cup \cN^X$
 is properly exceptional for $\xX^*$. 
  
\vip
 
$\bullet$ First, we have $\PP^X_x(X_t^* \notin \cN $ for all $t\geq 0)=1$ for all $x\in \cX\setminus \cN$.
Indeed, since $x \in \cX\setminus \cN$, the law of $(X_t^*)_{t\ge0}$ under $\PP^X_x$
equals the law of $(Y_t)_{t\ge 0}$ under 
$\QQ^Y_x$.
Since $\PP^U_u(U_t \notin \cN^U $ for all $t\geq 0)=1$ for all $u \in \cU\setminus \cN^U$
and since $\Phi_\sS (Y_t)=U_{A_t}$, we have
$\QQ^Y_x(Y_t \notin \Phi_\sS^{-1}(\cN^U)$ for all $t\geq 0)=1$ for all $x \in \cX\setminus \Phi_\sS^{-1}(\cN^U)$.
Hence $\PP^X_x(X_t^* \notin \Phi_\sS^{-1}(\cN^U)$ for all $t\geq 0)=1$ for all $x \in \cX \setminus 
(\Phi_\sS^{-1}(\cN^U) \cup \cN^X)$. Finally,
$\PP^X_x(X_t^* \notin \Phi_\sS^{-1}(\cN^U)\cup \cN^X$ for all $t\geq 0)=1$ for all $x \in \cX \setminus 
(\Phi_\sS^{-1}(\cN^U) \cup \cN^X)$ because $\cN^X$ is properly exceptional for $\xX^*$.
 
 \vip
 
$\bullet$ We have $\mu (\cN) = 0$. Indeed, $\mu ( \cN^X) = 0 $ by definition and, using Lemma
\ref{changeofvariable},
\begin{align*}
\mu ( \Phi_\sS ^{-1} (\cN^U)) 
=\frac12 \int _{\rr^2 \times \rr^*_+ \times \sS} \indiq _{\{\Psi(z,r,u)\in \Phi_\sS ^{-1}(\cN^U)\}} r^\nu \dd z \dd r \beta (\dd u) 
= \frac12 \int _{\rr^2 \times \rr^*_+} \beta (\cN^U)r^\nu \dd z \dd r =0,
 \end{align*} 
because $\beta (\cN^U) = 0$. We used that $\Psi(z,r,u)\in \Phi_\sS ^{-1}(\cN^U)\Leftrightarrow
u\in \cN^U$, since $\Phi_\sS (\Psi(z,r,u))=u$.
\end{proof}

\subsection{An expression of dispersion processes on the sphere}

We now study the dispersion process   $(R_K(U_t))_{t\ge 0}$  , for $K \subset \ig 1,N\id$.
The equation below can be informally established if assuming that \eqref{EDS} rigorously holds true, 
after a time-change and \blue several It\^o computations. \bla

\begin{lemma}\label{formuleRKU}
Fix $N\geq 2$ and $\theta>0$ such that $N>\theta$ and recall that $k_0=\lceil 2N/\theta \rceil$. Consider
a $QSKS(\theta,N)$ -process $\uU$ with life-time $\xi$,
fix $K \subset \ig 1,N \id$ such that $|K|\geq 2$, and set $\bK=(K,K^c)$.
Recall that \blue $G_{\bK,\e}$ \bla was introduced in Lemma~\ref{indiqXU}, and 
observe that
$$
\blue G_{\bK,0}\cap \sS = \bla \Big\{u \in \cU : \min_{i \in K, j\notin K} ||u^i-u^j||>0\Big\}.
$$
\blue Quasi-everywhere in $G_{\bK,0}\cap \sS$, \bla enlarging the filtered probability space 
$(\Omega^U,\cM^U,(\cM^U_t)_{t\ge 0},\PP^U_u)$ if necessary, there exists a $1$-dimensional
$(\cM^U_t)_{t\ge 0}$-Brownian motion $(W_t)_{t\geq 0}$ under $\PP_u^U$ such that
\begin{align}\label{abcd}
R_K(U_t) =& R_K(u) + 2 \int_0^t \sqrt{R_K(U_s) (1 - R_K(U_s))}\dd W_{s} + d_{\theta ,N}(|K|) t \\
&- d_{\theta , N}(N) \int _0^t R_K(U_s) \dd s   
-\frac{2\theta} N \sum_{i\in K, j\notin K} \int_0^t \frac{U_s^i-U_s^j}{\|U_s^i-U_s^j\|^2} \cdot (U_s^i -S_K(U_s))\dd s
\notag
\end{align} 
for all $t \in [0,\kappa_K)$, where $\kappa_K= \inf\{t \ge 0 : U_t \notin \blue G_{\bK,0} \bla \}$.
\end{lemma}

 As usual, $\kappa_K \leq \xi$ because $\triangle \notin \blue G_{\bK,0}$. \bla
Note also that if $K=\ig1,N\id$, then $R_K(U_t)=1$ for all $t\in [0,\xi)$,
and that the constant process $1$ indeed solves \eqref{abcd}.

\begin{proof}
We divide the proof in several steps. The main idea is to compute $\cL^U R_K$ and $\cL^U (R_K)^2$ and
to use that $R_K(U_t)=R_K(u)+\intot \cL^U R_K(U_s)\dd s +M_t$, for some martingale
$(M_t)_{t\geq 0}$ of which we can compute the bracket. However, we need to 
regularize $R_K$ and to localize space in a zone where the last term of \eqref{abcd} is bounded.

\vip

{\it Step 1.} We fix $n\geq 1$ \blue and $\e\in (0,1]$ and recall 
$\Gamma_{\bK,\e}^{\sS, n}\in C^\infty(\sS)$, compactly supported in $G_{\bK,0}\cap \sS$, \bla 
was defined in Lemma~\ref{indiqXU}.
We want to apply Remark~\ref{caracdomaine} to
\blue $R_K \Gamma_{\bK,\e}^{\sS, n}$ \bla and
\blue$(R_K \Gamma_{\bK,\e}^{\sS, n})^2$. \bla
We thus have to show that \blue $R_K \Gamma_{\bK,\e}^{\sS, n}$ \bla and 
\blue$(R_K \Gamma_{\bK,\e}^{\sS, n})^2$ \bla belong to $C^\infty_c(\cU)$ for all $n\ge 1$, which is clear,
and that
$$
\sup_{\alpha \in (0,1]}\sup_{u\in\sS} \Big(|\cL^U_\alpha [R_K \blue \Gamma_{\bK,\e}^{\sS, n}]\bla (u)|+
|\cL^U_\alpha [(R_K \blue \Gamma_{\bK,\e}^{\sS, n}\bla )^2](u)|\Big)<\infty
$$ 
 for all $n\ge 1$. Since
\begin{align}\label{lproduitU}
\cL^U_\alpha (fg) = f \cL^U_\alpha g  + g \cL^U_\alpha f + \nabla_\sS f \cdot \nabla_\sS g
\end{align}
for all $f,g \in C^\infty (\sS)$ and recalling that
$\sup_{\alpha \in (0,1]}\sup_{u\in\sS}|\cL^U_\alpha \blue \Gamma_{\bK,\e}^{\sS, n}\bla (u)|<\infty$ by Lemma~\ref{indiqXU}
and that \blue $\Gamma_{\bK,\e}^{\sS, n}$ is compactly supported in $G_{\bK,0}\cap \sS$, the
only issue is to verify that, for $A$ compact in  $G_{\bK,0}\cap \sS$,\bla
\begin{equation}\label{o1}
\sup_{\alpha \in (0,1]}\sup_{u\in A} |\cL^U_\alpha R_K(u)|<\infty.
\end{equation}

{\it Step 2.} Here we prove that
\begin{align}\label{LURKa}
 \cL^U_\alpha R_K (u) =& 2(|K|-1)-2(N-1)R_K(u)+ \frac\theta N R_K(u)  
\sum_{1\leq i, j \leq N}\frac{\|u^i-u^j\|^2}{\|u^i-u^j\|^2+\alpha} \\
&-\frac\theta N \sum_{i\in K, j\in K} \frac{\|u^i-u^j\|^2}{\|u^i-u^j\|^2+\alpha} 
-\frac{2\theta }N \sum_{i\in K, j\notin K} \frac{u^i-u^j}{\|u^i-u^j\|^2+\alpha}\cdot (u^i-S_K(u)), \notag
\end{align}
and this will imply \eqref{o1}:
the first four terms are obviously uniformly bounded on $\sS$, and the last one is uniformly bounded
on \blue $A$ (because $A$ is compact in $G_{\bK,0}\cap \sS$). \bla
\vip
This will also imply, taking $\alpha=0$ and observing that
$2(|K|-1)-\frac\theta N |K|(|K|-1)=d_{\theta,N}(|K|)$ and $2(N-1)- \frac\theta N N(N-1)=d_{\theta,N}(N)$,
that for all $u\in \sS\cap E_2$,
\begin{align}\label{lurk}
\cL^U R_K (u) = d_{\theta,N}(|K|) - d_{\theta,N}(N) R_K(u) - 
\frac{2\theta }N \sum_{i\in K, j\notin K} \frac{u^i-u^j}{\|u^i-u^j\|^2}\cdot (u^i-S_K(u)).
\end{align}
\vip
{\it Step 2.1.} We first verify that for all $u\in \sS$,
\begin{gather}
(\nabla_\sS R_K(u))^i=2 (u^i - S_K(u)) \indiq_{\{ i\in K\}} 
-2R_K(u)u^i,\qquad i\in \ig1,N\id, \label{fi}\\[3pt]
\Delta_\sS R_K(u)=4(|K|-1)-4(N-1)R_K(u). \label{si}
\end{gather}
First, a simple computation shows that for $x\in (\rr^2)^N$, for $i\in \ig 1,N\id$,
\begin{align}\label{deriveRK}
\nabla _{x^i} R_K (x) = 2 (x^i - S_K(x))\indiq_{\{ i\in K\}}\quad
\hbox{and} \quad \Delta_{x^i} R_K(x)=\frac{4(|K|-1)}{|K|} \indiq_{\{i\in K\}},
\end{align}
so that in particular $\nabla R_K (x) \in H$ and
\begin{align}\label{supc}
\nabla R_K (x) \cdot x = 2\sum _{i\in K} (x^i - S_K (x))\cdot x^i
=2\sum _{i\in K} (x^i - S_K (x))\cdot (x^i - S_K(x))= 2R_K(x).
\end{align}
 
Next, proceeding as in \eqref{tradgrad}, we get
$\nabla[R_K\circ\Phi_\sS](x)=||\pi_H (x)||^{-1}\pi_H(\pi_{(\pi_H(x))^\perp}(\nabla R_K(\Phi_\sS (x))))$ 
for all $x\in E_N$, so that
$$
\nabla[R_K\circ\Phi_\sS](x)=\frac{\pi_H\Big(\nabla R_K(\Phi_\sS (x)) 
- \frac{\pi_H(x)\cdot \nabla R_K(\Phi_\sS (x))}{||\pi_H(x)||^2}\pi_H(x)  \Big)}{||\pi_H (x)||}
=\frac{\nabla R_K(x)-2R_K(x)\frac{\pi_H(x)}{||\pi_H(x)||^2}}{||\pi_H(x)||^{2}}.
$$
We used that $\nabla R_K ( \Phi_\sS (x)) = \nabla R_K (x) / \|\pi_H(x)\|$ thanks to \eqref{deriveRK},
that  $\nabla R_K (x) \in H$ by \eqref{deriveRK}
and that $\pi_H(x)\cdot \nabla R_K (x)=x \cdot \nabla R_K (x)=2R_K(x)$
by \eqref{supc}.
\vip
We first conclude that for $u\in \sS$, since 
$\pi_H(u)=u$ and $||u||=1$,
\begin{equation}\label{supd}
\nabla_\sS R_K(u)=\nabla[R_K\circ\Phi_\sS](u)=   \nabla R_K(u)-2R_K(u)u,
\end{equation}
which implies \eqref{fi} by \eqref{deriveRK}. 
\vip
Second, we deduce that for $x\in E_N$,
\begin{align*}
\Delta[R_K\circ\Phi_\sS](x)=&\frac1{||\pi_H(x)||^{2}}\Big(\Delta R_K(x)
-2\nabla R_K(x)\cdot \frac{\pi_H(x)}{||\pi_H(x)||^{2}} - 
2 R_K(x) \frac{\ddiv \pi_H(x)}{||\pi_H(x)||^{2}} + \frac{4 R_K(x) }{||\pi_H(x)||^{2}}\Big)\\
& - \frac{2\pi_H(x)}{||\pi_H(x)||^4} \cdot  \Big(\nabla R_K(x)-2R_K(x)\frac{\pi_H(x)}{||\pi_H(x)||^{2}}\Big).
\end{align*}
Using that $\ddiv \, \pi_H (x)=2(N-1)$, we conclude that for $u\in \sS$, since 
$\pi_H(u)=u$, $||u||=1$ and $u\cdot\nabla R_K(u)=2R_K(u)$ by \eqref{supc},
$$
\Delta_\sS R_K (u)= \Delta[R_K\circ\Phi_\sS](u)=\Delta R_K(u)-4R_K(u)-4(N-1)R_K(u)+4R_K(u).
$$
Since finally $\Delta R_K(u)=4(|K|-1)$ by \eqref{deriveRK}, this leads to \eqref{si}.
 
\vip
{\it Step 2.2.}
We fix $u\in \sS$ and show that setting 
$I_\alpha(u)= -\frac\theta N \sum _{1\leq i, j\leq N} \frac{u^i-u^j}{\|u^i-u^j\|^2+\alpha}\cdot (\nabla_\sS R_K (u))^i$,
 it holds that
\begin{align}\label{LUsing}
I_\alpha(u) =& -\frac\theta N \sum_{i\in K, j\in K} \frac{\|u^i-u^j\|^2}{\|u^i-u^j\|^2+\alpha} 
+ \frac\theta N R_K(u)  \sum_{1\leq i, j \leq N}\frac{\|u^i-u^j\|^2}{\|u^i-u^j\|^2+\alpha}\\
&-\frac{2\theta }N \sum_{i\in K, j\notin K} \frac{u^i-u^j}{\|u^i-u^j\|^2+\alpha}\cdot (u^i-S_K(u)).  \notag
\end{align}
By \eqref{fi}, we may write $I_\alpha=I_{1,\alpha}+I_{2,\alpha}$, where
\begin{gather*}
I_{1,\alpha}(u) = -\frac{2\theta }N \sum_{i\in K, j\in \ig 1,N \id} \frac{u^i-u^j}{\|u^i-u^j\|^2+\alpha}
\cdot (u^i-S_K(u)),\\
I_{2,\alpha}(u)=\frac{2\theta}N R_K (u) \sum_{1\leq i, j \leq N}\frac{u^i-u^j}{\|u^i-u^j\|^2+\alpha}\cdot u^i.
\end{gather*}
First, by symmetry, 
\begin{align*}
I_{1,\alpha}(u)=& -\frac{2\theta }N \sum_{i\in K, j\in K} \frac{u^i-u^j}{\|u^i-u^j\|^2+\alpha}\cdot (u^i -S_K(u))
-\frac{2\theta }N \sum_{i\in K, j\notin K} \frac{u^i-u^j}{\|u^i-u^j\|^2+\alpha}\cdot (u^i -S_K(u))\\
=&-\frac{2\theta }N \sum_{i\in K, j\in K} \frac{u^i-u^j}{\|u^i-u^j\|^2+\alpha}\cdot u^i
-\frac{2\theta }N \sum_{i\in K, j\notin K} \frac{u^i-u^j}{\|u^i-u^j\|^2+\alpha}\cdot (u^i -S_K(u))\\
=& -\frac\theta N \sum_{i\in K, j\in K} \frac{\|u^i-u^j\|^2}{\|u^i-u^j\|^2+\alpha} 
-\frac{2\theta }N \sum_{i\in K, j\notin K} \frac{u^i-u^j}{\|u^i-u^j\|^2+\alpha}\cdot (u^i-S_K(u)).
\end{align*}
Second, by symmetry,
$$
I_{2,\alpha}(u)=\frac \theta N  R_K(u)
\sum_{1\leq i, j \leq N}\frac{\|u^i-u^j\|^2}{\|u^i-u^j\|^2+\alpha}.
$$

{\it Step 2.3.}  Since $\cL^U_\alpha R_K (u)= \frac12\Delta_\sS R_K(u)+I_\alpha(u)$, 
\eqref{LURKa} follows from \eqref{si} and \eqref{LUsing}.

\vip

{\it Step 3.} \blue By Steps 1 and 2, we can apply 
 Remark~\ref{caracdomaine} and Lemma~\ref{marting}:  quasi-everywhere, for all $n\ge 1$, 
there exist two 
$(\cM^U_t)_{t\ge 0}$-martingales $(M^{1,n,\e}_t)_{t\ge 0}$ and $(M^{2,n,\e}_t)_{t\ge 0}$ under $\PP^U_u$, 
such that
\begin{gather*}
\blue (R_K \Gamma^{\sS,n}_{\bK,\e}) (U_t) = (R_K \blue \Gamma^{\sS,n}_{\bK,\e}  ) (u) + M^{1,n,\e}_t + 
\int_0^t \cL^U (R_K \Gamma^{\sS,n}_{\bK,\e}) (U_s)\dd s, \\
(R_K \Gamma^{\sS,n}_{\bK,\e})^2 (U_t) = (R_K \Gamma^{\sS,n}_{\bK,\e})^2 (u) + M^{2,n,\e}_t 
+ \int_0^t \cL^U (R_K \Gamma^{\sS,n}_{\bK,\e})^2 (U_s)\dd s
\end{gather*}
for all $t\ge 0$. We recall that $\kappa_K=\inf \{ t\ge 0: U_t \notin G_{\bK,0}^{n} \}$ and introduce 
$$
\kappa_{K,n,\e} = \inf \{ t\ge 0: U_t \notin \blue G_{\bK,\e}^{n} \} \land \kappa_K.
$$
Since $\cup_{n \geq 1} G_{\bK,\e}^{n} \supset G_{\bK,\e}$ and since
$G_{\bK,\e}$ increases to $G_{\bK,0}$ as $\e\to 0$, see Lemma~\ref{indiqXU},
we conclude that $\lim_{\e \to 0} \lim_{n\to \infty}\kappa_{K,n,\e}=\kappa_K$.  Next, since
$\Gamma^{\sS,n}_{\bK,\e}=1$ on $G_{\bK,\e}^{n}\cap \sS$, we have, for all $t \in [0,\kappa_{K,n,\e}]$,
\begin{gather}\label{martU1}
R_K(U_t) = R_K(u) + M^{1,n,\e}_t +
\int_0^t \cL^U R_K (U_s)\dd s, \\
(R_K(U_t))^2 = (R_K (u))^2 + M^{2,n,\e}_t + \int_0^t \cL^U (R_K^2) (U_s)\dd s.\label{martU2}
\end{gather}
Applying the It\^o formula to compute $(R_K(U_t))^2$ from \eqref{martU1},
recalling from \eqref{lproduitU} that $\cL^U (R_K^2)=2 R_K \cL^U R_K + ||\nabla_\sS R_K||^2$
and comparing to \eqref{martU2},
we obtain that for $t \in [0,\kappa_{K,n,\e}]$,
$$
\langle M^{1,n,\e}\rangle _{t} = \int _0 ^{t} \|\nabla_\sS R_K (U_s)\|^2 \dd s.
$$
Hence, enlarging the probability space if necessary, we can find a Brownian motion
$(W_t)_{t\geq 0}$, which is defined by $W_t=\int_0^t \|\nabla_\sS R_K (U_s)\|^{-1}\dd M^{1,n,\e}_s$
for $t \in [0,\kappa_{K,n,\e}]$ and which is then extended to $\rr_+$, such that 
$M^{1,n,\e}_t=\int_0^t \|\nabla_\sS R_K (U_s)\| \dd W_s$ during  $[0,\kappa_{K,n,\e}]$. Hence,
still for $t\in[0,\kappa_{K,n,\e}]$, \bla
\begin{equation}\label{trez}
R_K(U_t)= R_K(u)+\int_0^t \|\nabla_\sS R_K (U_s)\| \dd W_s + \int_0^t \cL^U R_K (U_s)\dd s.
\end{equation}
But $\nabla_\sS R_K(u)=\nabla R_K(u)-2R_K(u) u$ by \eqref{supd}, whence
$$
\|\nabla_\sS R_K (u) \|^2 = \| \nabla R_K (u)\|^2 -4R_K (u) \nabla R_K (u) \cdot u + 4 (R_K(u))^2.
$$ 
Since $||\nabla R_K(u)||^2=4R_K(u)$ by \eqref{deriveRK} and $\nabla R_K(u)\cdot u = 2R_K(u)$ by 
\eqref{supc},
$$
\| \nabla _\sS R_K (u) \|^2 = 4 R_K (u) - 4 (R_K (u))^2 = 4R_K(u)(1-R_K(u)).
$$
\blue Inserting this, as well as the expression \eqref{lurk} of $\cL^U R_K$, in \eqref{trez}, shows that 
$R_K(U_t)$ satisfies the desired equation on $[0,\kappa_{K,n,\e}]$.
Since
$\lim_{\e\to 0}\lim_{n\to \infty}\kappa_{K,n,\e}=\kappa_K$ a.s., the proof is complete. \bla
\end{proof}

\subsection{A squared Bessel-like process}

The equation obtained in the previous lemma will be studied by comparison with the 
process we now introduce. This process behaves, near $0$, like a squared Bessel processes.

\begin{lemma}\label{touche}
Fix $\delta \in \rr$, $a>0$ and $b>0$ such that $\delta+a\sqrt{b}<2$. 
For $(W_t)_{t\ge 0}$ a $1$-dimensional Brownian motion
and for $x \in [0,1)$, consider the unique solution $(S_t)_{t\geq 0}$ of
\begin{align}\label{tbc}
S_{t} = x + \int _{0} ^{t} 2\sqrt{|S_s(1-S_s)|} \dd W_{s} + \delta t + a\int _{0}^{t} \sqrt{b+|S_s|} \dd s.
\end{align}
 For $z \in \rr$, set  $\tau_z=\inf\{t> 0 : S_t=z\}$.
For all $y\in(x,1)$, it holds that $\PP( \tau_0<\tau_y)>0$.
\end{lemma}

\begin{proof}
This equation is classically well-posed, since the diffusion coefficient is $1/2$-H\"older continuous
and the drift coefficient is Lipschitz continuous, see Revuz-Yor \cite[Theorem 3.5 page 390]{ry}.
As in Karatzas-Shreve \cite[(5.42) page 339]{kshreve}, we introduce the scale function
$$
f(z)=\int_{1/2}^z \exp\Big(-\int_{1/2}^u \frac{\delta+a\sqrt{b+|v|}}{2 |v(1-v)|} \dd v\Big) \dd u.
$$
This function is obviously continuous
on $(0,1)$ and one gets convinced, for example approximating 
$(\delta+a\sqrt{b+|v|})/(2 |v(1-v)|)$ by $(\delta+a\sqrt{b})/(2 |v|)$,
that it is also continuous at $0$ because $\delta+a\sqrt{b}<2$. By \cite[(5.61) page 344]{kshreve}, we have
\begin{equation}\label{sc}
\PP(\tau_0<\tau_y)= \frac{f(y)-f(x)}{f(y)-f(0)}.
\end{equation}
for all $y \in (x,1)$. This last quantity is nonzero (which would not be the case if $\delta +a\sqrt{b}\geq 2$, 
since then $f(0)=-\infty$).
\end{proof}

\subsection{Collisions of large clusters}

We are now ready to give the

\begin{proof}[Proof of Proposition~\ref{N-1toN}-(i)-(ii)]
We fix $N\geq 4$, $\theta>0$ such that $N>\theta$. We always assume that
$d_{\theta,N}(N)<2$ and we use the notation of Subsection~\ref{rrrr}.
\vip
{\it Step 1.} We consider $\e\in (0,1]$ and $K\subset \ig 1,N\id $ such that $|K|\in \ig 2,N-1\id$ 
and $d_{\theta,N}(|K|)<2$.
We introduce the constant $a_K= c_{|K|+1}/(2c_{|K|})$ with $(c_\ell)_{\ell \in \ig 1,N \id}$ defined in Lemma 
\ref{campingcar}. 
We prove in this step that there are some constants $p_{K,\e}>0$ and $T_{K,\e}>0$ such that, setting 
\begin{gather*}
\tilde\sigma^{K,\e} = \inf \Big\{ t> 0 : R_K(U_t) \ge \e \; \mbox{ or } \;\min_{i\notin K} R_{K\cup \{ i\}}(U_t) 
\le a_K \e\Big\}\land T_{K,\e},
\end{gather*}
with the convention that $\inf \emptyset = \xi$, it holds that 
\blue quasi-everywhere on $\{u \in \cU : R_K(u) \le \e /2\}$, \bla
$$
\PP^U_{u} \Big(\tilde\sigma^{K,\e} = \xi \;
\mbox{ or }  \inf_{t\in [0, \tilde \sigma^{K,\e})}\min_{i\notin K} R_{K\cup \{ i\}}(U_t) \le 2a_K\e  \hbox{ or } 
R_K(U_t) = 0 \mbox{ for some } t\in [0, \tilde \sigma^{K,\e}) \Big)\ge p_{K,\e}.
$$

\blue We introduce $Z_{K,\e}=\inf_{t\in [0, \tilde \sigma^{K,\e})}\min_{i\notin K} R_{K\cup \{ i\}}(U_t)$. \bla
We note that for all $t \in [0, \tilde\sigma^{K,\e})$, 
$R_K(U_t)\le \e $ and $\blue Z_{K,\e}\bla \ge a_K\e$ so that
$\min_{i\in K,j\notin K} \|U^i_t-U^j_t\| \ge \e/2$ thanks to the definition of $a_K$ and to Lemma~\ref{campingcar}. 
This implies that $\tilde\sigma^{K,\e} \leq \kappa_K$,
where we recall that $\kappa_K= \inf\{t \ge 0 : U_t \notin \blue G_{\bK,0} \bla\}$
was defined in Lemma~\ref{formuleRKU}, and that
$\blue G_{\bK,0}\cap \sS\bla=\{u \in \cU : \min_{i\in K,j\notin K} ||u^i-u^j||>0\}$.

\vip

By the Cauchy-Schwarz inequality, and since $R_K$ is bounded on $\cU$,
there is a deterministic constant $C_{K,\e}>0$, allowed to change from line to line,
such that for all $t\in [0, \tilde\sigma^{K,\e})$, we have
\begin{align*}
&-d_{\theta,N}(N)R_K(U_t)- \frac{2\theta} N  \sum_{i\in K, j\notin K} \frac{U_t^i-U_t^j}{\|U_t^i-U_t^j\|^2} 
\cdot (U_t^i-S_K(U_t)) \\
\leq & C_{K,\e} \sqrt{R_K(U_t)}+ 
C_{K,\e} \Big(\sum_{i\in K} \|U_t^i-S_K(U_t)\|^2  \Big)^{1/2}\\
\leq & C_{K,\e} \sqrt{R_K(U_t)}\\
\leq& C_{K,\e} \sqrt{b+R_K(U_t)}
\end{align*}
where $b>0$ is chosen small enough so that $d_{\theta,N}(|K|)+C_{K,\e} \sqrt{b}<2$. Actually, $b$ is only introduced
to make the drift coefficient of \eqref{tbc} Lipschitz continuous.

\vip

Recalling that $R_K(U_0)\leq \e /2$, the formula describing $R_K(U_t)\in [0,1]$ for 
$t\in [0,\kappa_K) \supset [0,\tilde\sigma^{K,\e})$, see Lemma~\ref{formuleRKU},
considering the process $(S_t)_{t\geq 0}$ solution to \eqref{tbc} with $x=\e / 2$, 
$\delta=d_{\theta,N}(|K|)$, $a=C_{K,\e}$
and with $b$ introduced a few lines above, driven by the same Brownian motion $(W_t)_{t\geq 0}$,
and using the comparison theorem, we conclude that $R_K(U_t) \leq S_t$ for all  
$t\in [0,\tilde\sigma^{K,\e})$.

\vip

Setting $\tau_z=\inf\{t\geq 0 : S_t=z\}$ for $z \in \rr$
and recalling the definition of $\tilde \sigma^{K,\e}$, we conclude that
\blue $\{Z_{K,\e} > 2a_K\e \}
\subset \{\tilde\sigma^{K,\e} \geq \tau_{\e}\land T_{K,\e}\}.$ \bla
Indeed, on $\{\inf_{t\in [0, \tilde \sigma^{K,\e})}\min_{i\notin K} R_{K\cup \{ i\}}(U_t) > 2a_K\e\}$,
either $\tilde\sigma^{K,\e}=T_{K,\e}$, or   $(R_K(U_t))_{t\ge 0}$   reaches $\e$ at time 
$\tilde\sigma^{K,\e}$ and we then have
$\tau_\e \leq \tilde\sigma^{K,\e}$. In both cases, $\tilde\sigma^{K,\e} \geq \tau_{\e}\land T_{K,\e}$. 
Hence, using again that $R_K(U_t) \leq S_t$ for  all $t\in [0,\tilde\sigma^{K,\e})$,
\begin{align*}
&\Big\{\tilde \sigma^{K,\e} < \xi \; \hbox{ and } \;\blue Z_{K,\e} \bla > 2a_K\e \;\mbox{ and }
\;S_t = 0 \mbox{ for some } t\in [0, \tau_{\e} \land T_{K,\e}]\Big\} \\
\subset& \Big\{\tilde\sigma^{K,\e} < \xi\; \mbox{ and }\; \blue Z_{K,\e} \bla > 2a_K\e \; \hbox{ and }\;R_K(U_t) = 0 \mbox{ for some } t\in [0, \tilde\sigma^{K,\e})\Big\}.
\end{align*}
\blue
But $A^c \cap B' \subset A^c\cap B$ gives $\PP(A \cup B)=\PP(A)+\PP(A^c \cap B) \geq
\PP(A)+\PP(A^c \cap B')=\PP(A\cup B')$. Hence
\begin{align*}
&\PP^U_{u} \Big(\tilde\sigma^{K,\e} = \xi \;\mbox{ or } \; \blue Z_{K,\e} \bla \le 2a_K\e \;\hbox{ or } \;
R_K(U_t) = 0 \mbox{ for some } t\in [0, \tilde \sigma^{K,\e}) \Big)\\
\geq& \PP^U_u\Big(\tilde\sigma^{K,\e} = \xi \;\hbox{ or }\;\blue Z_{K,\e} \bla \le 2a_K\e \;\mbox{ or } 
\; S_t = 0 \mbox{ for some } t\in [0, \tau_{\e} 
\land T_{K,\e})\Big)\\
\geq & \PP^U_u\Big(S_t = 0 \mbox{ for some } t\in [0, \tau_{\e} \land T_{K,\e})\Big).
\end{align*}
\bla
This last quantity equals $\PP(\tau_0<\tau_{\e}\land T_{K,\e})$ and does not depend on $u$ such that 
$R_K(u)\leq \e /2$. But $\PP(\tau_0<\tau_{\e})>0$ by Lemma~\ref{touche}
and since $d_{\theta,N}(|K|)+C_{K,\e} \sqrt{b}<2$. Hence
there exists 
$T_{K,\e}>0$ so that $\PP(\tau_0<\tau_{\e}\land T_{K,\e})>0$
and this completes the step.

\vip

{\it Step 2.} We prove (ii), i.e. that when $d_{\theta,N}(N-1)\in (0,2)$, for any $K \subset \ig 1,N \id$
with cardinal $|K|=N-1$, \blue quasi-everywhere, \bla $\PP^X_x$-a.s.,
$R_K(X_t)$ vanishes during $[0,\zeta)$.
By \eqref{etoile22}  and Remark~\ref{UtoX}, and since $\PP^U_u(\xi=\infty)=1$ \blue quasi-everywhere \bla
by Lemma~\ref{visite}-(ii), it suffices to check that \blue quasi-everywhere, \bla $\PP^U_u$-a.s.,
$(R_K(U_t))_{t\geq 0}$ vanishes at least once during $[0,\infty)$.

\vip

We fix $K\subset \ig 1,N\id $ with $|K| = N-1$, set $\e_0=1/(4a_{K} )$
and introduce $\tilde\tau^K_0 = 0$ and for all $k\ge 0$,
\begin{gather*}
\tau^K_{k+1} = \inf \{ t\ge \tilde\tau^K_k : R_K(U_t) \le \e_0 /2 \},
\\
\tilde\tau^K_{k+1} = \inf \{ t\ge \tau^K_{k+1} : R_K(U_t) \ge \e _0  \}
\land (\tau^K_{k+1}+T_{K,\e _0}).
\end{gather*}
with $T_{K,\e _0}$ defined in Step 1.
All these stopping times are finite since $(\cE^U,\cF^U)$ is recurrent by Lemma~\ref{visite}-(ii).
We also put, for $k\geq 1$,
$$
\Omega_k^K=\{R_K(U_t) = 0 \hbox{ for some } t\in [\tau^{K}_k, \tilde\tau^{K}_k]\}.
$$ 
We now prove that $\PP^U_{u } ( \cap _{k\ge 1} (\Omega_k^{K})^c) = 0$ \blue quasi-everywhere,
and this will complete the proof of (ii). \bla
\vip
For $\ell \ge 1$, 
since $\cap_{k=1}^{\ell} (\Omega_k^K)^c $ is $\cM^U_{\tau^K_{\ell+1}}$-measurable, the strong Markov property
tells us that
$$ 
\PP^U_{u } \Big( \cap_{k=1}^{\ell +1} (\Omega_k^K)^c \Big) 
= \E^U_u\Big[ \Big(\prod_{k=1}^{\ell} \indiq_{(\Omega_k^K)^c} \Big)\PP^U_{U_{\tau^K_{\ell +1} }} ( (\Omega_{1}^K)^c)\Big].
$$

We now prove that $\PP^U_u(\Omega_{1}^K) \geq p_{K,\e_0}$ \blue quasi-everywhere on 
$\{u\in \cU : R_K(u)\leq \e_0/2\}$. \bla
For such a $u$, we have $\tau^K_1=0$. Moreover, for all $i\notin K$, we have $R_{K\cup\{i\}}(u)=R_{\ig 1, N\id}(u)=1> 2a_{K}\e_0$ 
thanks to our choice of $\e_0$.
Hence $\tilde \tau^K_1=\tilde \sigma^{K,\e_0}$, recall Step 1. Since finally
$\tilde \sigma^{K,\e_0}<\infty=\xi$ and since $R_{K\cup\{i\}}(U_t)=R_{\ig 1, N\id}(U_t)=1> 2a_{K}\e_0$
for all $t\geq 0$ and all $i \notin K$,
\begin{align*}
\Omega_1^K=&\Big\{ R_K(U_t) = 0 \hbox{ for some } 
t\in [0, \tilde \sigma^{K,\e_0}]\Big\}\\
=&\Big\{\tilde \sigma^{K,\e_0}=\xi \; \hbox{ or } 
 \inf_{t\in[0,\tilde\sigma^{K,\e_0})}\min_{i\notin K}R_{K\cup\{i\}}(U_t)\leq 2a_K\e_0
\; \hbox{ or } \;R_K(U_t) = 0 \hbox{ for some } 
t\in [0, \tilde \sigma^{K,\e_0}]\Big\}.
\end{align*}
Hence Step 1 tells us that $\PP^U_u(\Omega_{1}^K) \geq p_{K,\e_0}$ \blue quasi-everywhere on $\{u\in \cU :
R_K(u)\leq \e_0/2\}$. \bla

\vip

Since $R_K(U_{ \tau^K_{\ell +1} })\leq \e_0/2$, we have proved that for all $\ell\geq 1$,
$$
\PP^U_{u } \Big( \cap_{k=1}^{\ell +1} (\Omega_k^K)^c \Big) \le (1-p_{K,\e _0}) 
\PP^U_{u} \Big( \cap_{k=1}^{\ell} (\Omega_k^K)^c \Big).
$$
This allows us to conclude that indeed, $\PP^U_{u} ( \cap_{k=1}^{\infty} (\Omega_k^K)^c)=0$.

\vip 
{\it Step 3.} We prove (i), i.e. that if $d_{\theta,N}(N-1)\leq 0$, 
then $\PP^X_x(\inf_{[0,\zeta)}R_{\ig 1,N\id}(X_t)>0)=1$ \blue quasi-everywhere. \bla
By Remark~\ref{UtoX} 
and \eqref{etoile3}, it suffices to show that \blue quasi-everywhere, \bla $\PP^U_u(\xi<\infty)=1$.

\vip

For all $K \subset \ig 1,N\id$, all $\e\in (0,1]$, we introduce
$\tilde\sigma^{K,\e}_0 = 0$ and for all $k\ge 0$,
\begin{gather*}
\sigma^{K,\e}_{k+1} = \inf \Big\{ t\ge \tilde\sigma^{K,\e}_k :  R_K(U_t) \le \e /2 \; \mbox{ and } \;
\min_{i\notin K} R_{K\cup \{ i\}}(U_t) \ge 2a_K \e \Big\},\\
 \tilde\sigma_{k+1}^{K,\e} = \inf \Big\{ t\ge \sigma_{k+1}^{K,\e} : R_K(U_t) \ge \e \; \mbox{ or } \;
\min_{i\notin K} R_{K\cup\{i\}}(U_t) \le a_K \e \Big\} 
\land (\sigma^{K,\e}_{k+1}+T_{K,\e}),
\end{gather*}
with $T_{K,\e}$ defined in Step 1 and with the convention that $\inf \emptyset = \xi$.
\vip
{\it Step 3.1.} We fix $\e\in (0,1]$ and assume that $|K|\geq k_0$, so that
$d_{\theta,N}(|K|)\leq 0$ by Lemma~\ref{dimensions}. We prove here that
\blue quasi-everywhere, \bla $\PP^U_u$-a.s., either there is $t\in [0,\xi)$ such that $R_K(U_t)=0$
or there is $k\geq 1$ such that $\sigma^{K,\e}_{k+1}=\xi$ or there is $k\ge 1$ such that 
$\inf_{t\in [\sigma^{K,\e}_k, \tilde\sigma^{K,\e}_k)}\min_{i\notin K} R_{K\cup\{i\}}(U_t)  \le 2a_K\e$. 
\vip
It suffices to prove that $\PP^U_u ( \cap _{k\ge 1} (\Omega_k^{K,\e})^c ) = 0$, where
\begin{align*}
&\Omega^{K,\e}_k = \Big\{\sigma^{K,\e}_{k+1} = \xi  \,\mbox{ or }
\inf_{t\in [\sigma^{K,\e}_k,\tilde\sigma^{K,\e}_k)} \min_{i\notin K}R_{K\cup\{i\}}(U_t) \le 2a_K\e \\
&\hskip5cm\mbox{ or } \,
R_K(U_t)=0  \hbox{ for some }  t\in 
[ \sigma^{K,\e}_k, \tilde \sigma^{K,\e}_k)\Big\}.
\end{align*}
But for all $\ell \ge 1$, $\cap_{k= 1}^\ell (\Omega_k^{K,\e})^c $ is $\cM^U_{\sigma^{K,\e}_{\ell+1}}$-measurable, whence,
by the strong Markov property,
$$
\PP^U_u \Big( \cap_{k= 1}^{\ell+1} (\Omega_k^{K,\e})^c \Big) = 
\E^U_u \Big[ \Big(\prod_{k=1}^\ell \indiq_{(\Omega_k^{K,\e})^c}\Big) 
\PP^U_{U_{\sigma^{K,\e}_{\ell+1}}}\Big((\Omega_{1}^{K,\e})^c\Big)\Big]
\leq (1-p_{K,\e}) \PP^U_u \Big( \cap_{k=1 }^{\ell} (\Omega_k^{K,\e})^c \Big).
$$
We used Step 1, that $R_K(U_{\sigma^{K,\e}_{\ell+1}})\leq \e /2$ on the event
$(\Omega_\ell^{K,\e})^c\subset\{\sigma^{K,\e}_{\ell+1}<\xi\}$, as well as the inclusion
$\{\tilde \sigma^{K,\e}_{k} = \xi\}\subset \{\sigma^{K,\e}_{k+1} = \xi\}$. One easily concludes.

\vip

{\it Step 3.2.} For all $K\subset \ig 1,N\id$ such that $|K|\geq k_0$, \blue quasi-everywhere, \bla
$\PP^U_u$-a.s., there is no $t\in [0,\xi)$ such that $R_K(U_t)=0$.
Indeed, on the contrary event, there is $t\in [0,\xi)$ such that $U_t \notin E_{k_0}$,
whence $U_t \notin \cU$, which contradicts the fact that $t \in [0,\xi)$.

\vip

{\it Step 3.3.}
We show by decreasing induction that
$$\cP (n) : \hbox{ \blue quasi-everywhere, \bla $\PP^U_u$-a.s.
on the event $\{ \xi = \infty \}$, $b_n=\min_{\{|K|=n\}}\inf_{t\ge 0} R_K (U_t)>0$}
$$
holds true for every $n\in \ig k_0 ,N\id$.

\vip

The result is clear when $n=N$, because for all $t\in [0,\xi )$, $R_{\ig 1,N \id}(U_t) = 1$.

\vip

We next assume $\cP (n)$ for some $n\in \ig k_0+1,N \id$
and we show that $\cP(n-1)$ is true. 
We fix $K\subset \ig 1,N \id$ with cardinal $|K|=n-1$ and we apply Step 3.1 with $K$ and with some
$\e \in (0,b_{n}/(4a_K))$ 
($b_n$ is random but we may apply Step 3.1 simultaneously for all $\e \in \QQ_+^*\cap (0,1]$)
and Step 3.2, we find that on the event $\{ \xi = \infty \}$, there either
exists $k\ge 1$ such that $\sigma^{K,\e}_{k+1} = \infty$ or $k\ge 1$ such that 
$\blue\inf_{t\in [\sigma^{K,\e}_k, \tilde\sigma^{K,\e}_k)}\bla \min_{i\notin K} R_{K\cup \{i\}}(U_t) \le 2a_K\e$.
This second choice is not possible, since by induction assumption,
$R_{K\cup \{i\}}(U_t)\geq b_n$ for all $t>0$ and all $i \notin K$. Hence there is $k\geq 1$
such that $\sigma^{K,\e}_{k+1} = \infty$.

\vip

By definition of  $\sigma^{K,\e}_{k+1}$, this implies that, still on the event where $\xi=\infty$,
there exists $t_0 \ge 0$ such that for all $t\ge t_0$, 
either $R_K(U_t) \ge \e/2$ or $\min_{i\in K} R_{K\cup \{i\}}(U_t) \le 2a_K\e$. 
Using again the induction assumption, we get that the second choice is never possible, so that actually,
$R_K(U_t) \ge \e/2$ for all $t\geq t_0$. Since $(R_K(U_t))_{t\ge 0}$ is continuous
and positive   on $[0,t_0]$   according to Step 3.2, this completes the step.

\vip

{\it Step 3.4.} We conclude from Step 3.3 that  \blue quasi-everywhere, \bla $\PP^U_u$-a.s. 
on the event $\{ \xi = \infty \}$, $U_t \in \cK$ for all $t\geq 0$, where
$$
\cK = \{u \in \cU : \hbox{for all $n\in \ig k_0,N\id$, all $K \subset \ig 1,N\id$ with $|K|=n$, 
$R_K (u)\geq b_n$} \}.
$$
This (random) set is compact in $\cU$, so that Lemma~\ref{visite}-(i) tells us, 
both in the case where $(\cE^U,\cF^U)$ is recurrent and in the case where $(\cE^U,\cF^U)$ is transient,
that this happens with probability $0$. Hence \blue quasi-everywhere, \bla
$\PP^U_u(\xi=\infty)=0$ as desired.
\end{proof}

\subsection{Binary collisions}

We finally give the

\begin{proof}[ Proof of Proposition~\ref{N-1toN}-(iii)] We assume that $N\geq 4$, that
$0<d_{\theta,N}(N)<2\leq d_{\theta,N}(N-1)$ and observe that $\theta < 2$ and $k_0>N$,
so that $\cX=(\rr^2)^N$ and $\cU=\sS$. The $QKS(\theta,N)$-process $\xX$ is non-exploding
by Proposition~\ref{cont}-(i), and the $QSKS(\theta,N)$-process $\uU$ is irreducible recurrent
by Lemma~\ref{visite}-(ii). In particular, $\zeta=\xi=\infty$ a.s.
We divide the proof in 4 steps.
First, we prove that $\xX$ may have some binary collisions
with positive probability. Then we check that this implies that $\uU$ also may have some binary collisions
with positive probability.
Since $\uU$ is recurrent, it will then necessarily be a.s. subjected to (infinitely many)
binary collisions. Finally, we conclude using \eqref{etoile22}.

\vip

{\it Step 1.} We set $\bK = (\{ 1,2\}, \{ 3 \}, \dots, \{ N\})$ and \blue
$$
\cK = \Big\{ x \in B(0, C ) : \|x^1-x^2\| < 1 \; 
\mbox{ and }\; \min_{i\in \ig 1,N \id, j\in \ig 3,N\id, i \neq j } \|x^i-x^j\|> 10  \Big\},
$$
with $C$ large enough so that $\mu(\cK)>0$.
We show in this step that $\PP^X_x(A)>0$  quasi-everywhere in $\cK$, where
$$
A = \Big\{X^1_t = X^2_t \mbox{ for some } t\in [0,1] \; \mbox{ and } \; 
\min_{t\in [0,1]} R_{\ig 1,N \id}(X_t)>0 \Big\}.
$$

To this end, we fix $x\in \cK$ and introduce the set
\begin{gather*}
O= \Big\{y \in (\rr^2)^2 : R_{\{1,2\}}(y)<2, \;\;
\Big\|\frac{y^1+y^2}2-\frac{x^1+x^2}2\Big\|<1\Big\},
\end{gather*}
and $B_i=\{ y \in \rr^2 : ||y-x^i||^2<1\}$ for $i \in \ig3,N\id$. 
Clearly, there is some $\e\in (0,1]$ such that
$$
L=\Big\{y\in (\rr^2)^N : (y^1,y^2)\in  O \;\hbox{ and } \;y^i \in B_i 
\hbox{ for all $i\in \ig3,N\id$}\Big\}
\subset G_{\bK,\e},
$$
where as usual
$G_{\bK,\e}=\{y \in B(0,1/\e) : \forall\; i \in \ig 1,N\id,
\;\forall\; j\in \ig 3,N\id\setminus\{i\},\; ||y^i-y^j||^2>\e\}$, recall that $\cX=(\rr^2)^N$
because $k_0>N$.

\vip

Since $G_{\bK,\e}$ is obviously included in $\{y \in (\rr^2)^N : R_{\ig1,N\id}(y)>0\}$, we conclude that
\begin{align*}
\PP^X_x(A) \geq& \PP^X_x\Big(X^1_t = X^2_t \mbox{ for some } t\in [0,1] \; \mbox{ and } \; 
X_t\in L \mbox{ for all } t\in [0,1] \Big)\\
\geq & C_{1,\e,\bK}^{-1}\QQ^{1,\e,\bK}_x\Big(X^1_t = X^2_t \mbox{ for some } t\in [0,1] \; \mbox{ and } \; 
X_t\in L \mbox{ for all } t\in [0,1] \Big)
\end{align*}
by Proposition~\ref{girsanov} with $T=1$. We now set 
$\tau_{\bK,\e}=\inf\{t>0: X_t \notin G_{\bK,\e}\}$. Proposition~\ref{girsanov} tells us that,
quasi-everywhere in $\cK\subset G_{\bK,\e}$, the law of $(X_t)_{t\in [0, \tau_{\bK,\e}]}$  under 
$\QQ^{1,\e,\bK}_x$
equals the law of 
$Y_t=(Y^1_t,\dots, Y^N_t)_{t\in [0, \tilde\tau_{\bK,\e}]}$ where $(Y^1_t,Y^2_t)_{t\ge 0}$ is a
$QKS(2\theta /N, 2)$-process issued from $(x^1,x^2)$, where for all $i\in \ig 3,N \id$, $(Y^i_t)_{t\ge 0}$ is a 
$QKS(\theta /N, 1)$-process, i.e. a $2$-dimensional Brownian motion, issued from $x^i$, 
and where all these processes are independent. We have set $\tilde \tau_{\bK,\e}=\inf\{t>0: Y_t \notin G_{\bK,\e}\}$.
This implies, together with the fact that $\{X_t\in L \mbox{ for all } t\in [0,1] \}\subset \{\tau_{\bK,\e}>1\}$,
that 
$$
\PP^X_x(A) \ge  C_{1,\e,\bK}^{-1} \; p \prod_{i=3}^N q_{i}
$$
quasi-everywhere in $\cK$, where
\begin{align*}
p=\PP\Big(\min_{s\in[0,1]} R_{\{ 1,2\}}((Y^1_s,Y^2_s))= 0 \;\mbox{ and } \; (Y^1_t,Y^2_t)\in O
\mbox{ for all } t\in [0,1] \Big),
\end{align*}
and where $q_{i}=\PP(Y^i_t\in B_i$ for all $t\in [0,1])$. Of course, $q_i>0$ for all $i\in \ig3,N\id$, 
since $(Y^i_t)_{t\geq 0}$ is a Brownian motion issued from
$x^i$. Moreover, we know from Lemma~\ref{besbro} that
$(M_t=(Y^1_t+Y^2_t)/2)_{t\ge 0}$ is a $2$-dimensional Brownian motion with diffusion coefficient $2^{-1/2}$
issued from $m=(x^1+x^2)/2$,
that $(R_t=R_{\{1,2\}}((Y^1_t,Y^2_t)))_{t\ge 0}$ is a squared Bessel process of dimension 
$d_{2\theta/N, 2}(2) = d_{\theta,N}(2)$ issued from $r=||x^1-x^2||^2/2\in (0,1/2)$, 
and that these processes are independent. Hence, recalling the definition of $O$,
$$
p= \PP\Big(\min_{s\in[0,1]} R_s= 0 \;\mbox{ and } 
\max_{s\in [0,1]} R_s < 2\Big)\PP\Big(\max_{s\in [0,1]} ||M_t-m||< 1 \Big).
$$
This last quantity is clearly positive,  because a squared Bessel process
with dimension $d_{\theta,N}(2)\in (0,2)$, see Lemma~\ref{dimensions}, does hit zero, see Revuz-Yor
\cite[Chapter XI]{ry}.
\bla
\vip

{\it Step 2.} We now deduce from Step 1 that the set $F=\{u \in \cU : u^1=u^2\}$ is not exceptional for $\uU$.
Indeed, if it was exceptional, we would have $\PP^U_u(\exists\;t\geq 0 : U_t\in F)=0$ \blue quasi-everywhere. \bla
By \eqref{etoile22} and 
Remark~\ref{UtoX}, this would imply that \blue quasi-everywhere, \bla
$\PP^X_x(\exists\;t\in[0,\tau) : X_t\in G)=0$, where $G=\{x \in \cX : x^1=x^2\}$
and $\tau=\inf\{t>0 : R_{\ig1,N\id}(X_t)=0\}$.
But on the event $A$ defined in Step 1, there is $t\in [0,1]$ such that $X_t \in G$ and it holds that $\tau>1$. 
As a conclusion, $\PP^X_x(\exists\;t\in[0,\tau) : X_t\in G)>0$ \blue quasi-everywhere in $\cK$, \bla
whence a contradiction, since $\mu(\cK)>0$.

\vip

{\it Step 3.} Since $(\cE^U,\cF^U)$ is irreducible-recurrent and since $F$ is not exceptional, 
we know from Fukushima-Oshima-Takeda
\cite[Theorem 4.7.1-(iii) page 202]{f} that \blue quasi-everywhere, \bla
$$
\PP^U_u(\forall\; r>0, \exists \; t\geq r\; : \; U_t \in F)=1.
$$

{\it Step 4.} Using again \eqref{etoile22} and Remark~\ref{UtoX} and recalling that $\xi=\infty$
and that $\rho$ is an increasing bijection from $[0,\infty)$ to $[0,\tau)$,
we conclude that \blue quasi-everywhere, \bla $\PP^X_x$-a.s.,
$X_t$ visits $F$ (an infinite number of times) during $[0,\tau)$.
Of course, the same arguments apply when replacing $\{1,2\}$ by any subset of $\ig1,N\id$
with cardinal $2$, and the proof is complete.
\end{proof}

\section{\blue Quasi-everywhere conclusion \bla}\label{conccc}

Here we prove that the conclusions of Theorem~\ref{theoremecollisions} hold \blue quasi-everywhere. \bla

\begin{proof}[Partial proof of Theorem~\ref{theoremecollisions}]
We assume that $\theta\geq 2$ and $N>3\theta$, so that $k_0=\lceil 2N/\theta\rceil \in \ig 7,N\id$,
and consider a $\cX_\triangle$-valued $QKS(\theta,N)$-process $\xX$ with life-time $\zeta$
as in Proposition~\ref{existenceXU}, where $\cX=E_{k_0}$.
\vip
{\it Preliminaries.}
For $K \subset \ig1,N\id$ and $\e\in(0,1]$, we
write $\tau_{K,\e}=\inf\{t>0 : X_t \notin G_{K,\e}\}\in [0,\zeta]$ and 
$G_{K,\e}=\{x \in \cX : \min_{i\in K, j \notin K} ||x^i-x^j||^2>\e\}\cap B(0,1/\e)$
instead of $\tau_{\bK,\e}$ and $G_{\bK,\e}$ with $\bK=(K,K^c)$ as in Proposition~\ref{girsanov}.
We also write $\QQ^{T,\e,K}_x$ instead of $\QQ^{T,\e,\bK}$ and recall that it 
is equivalent to $\PP^X_x$ on $\cM^X_T=\sigma(X_s : s\in [0,T])$.

\vip

Setting $X^K_t=(X^i_t)_{i\in K}$ and $X^{K^c}_t=(X^i_t)_{i\in K^c}$,
we know  that \blue quasi-everywhere in $G_{K,\e}$, \bla the law of 
$(X^K_{t},X^{K^c}_t)_{t\in [0,   \tau_{K,\e}  \land T]}$ under $\QQ^{T,\e, K}_x$
is the same as the law of $(Y_t,Z_t)_{t\in [0, \tilde \tau_{K, \e} \land T]}$,
where $(Y_t)_{t\ge 0}$ is a $QKS(|K|\theta/N, |K|)$-process issued from
$x|_K$ and $(Z_t)_{t\geq 0}$ is a $QKS(|K^c|\theta/N, |K^c|)$-process issued from $x|_{K^c}$,
these two processes being independent, and where 
$\tilde\tau_{K,\e} = \inf\{ t>0 : (Y_t,Z_t) \notin G_{K, \e}\}$.
We denote by $\zeta^Y$ and $\zeta^Z$ the life-times of $(Y_t)_{t\ge 0}$ and $(Z_t)_{t\ge 0}$.
The life-time of $(Y_t,Z_t)_{t\geq 0}$ is given by $\zeta'=\zeta^Y\land\zeta^Z$ and it holds that 
$\tilde \tau_{K,\e} \in [0,\zeta']$.
\vip

{\it No isolated points.} Here we prove that for all $K \subset \ig1,N\id$ with $d_{\theta,N}(|K|)\in (0,2)$,
\blue quasi-everywhere, \bla
we have $\PP^X_x(A_K)=0$, where $A_K=\{\cZ_K$ has an isolated point$\}$ and
$$\cZ_K=\{t \in (0,\zeta) : \hbox{ there is a $K$-collision in the configuration } X_t\}.$$

On $A_K$, we can find $u,v \in \QQ_+$ such that $u<v<\zeta$ and such that there is a unique 
$t\in (u,v)$ with $R_K(X_t)=0$ and $\min_{i\notin K}R_{K\cup\{i\}}(X_t)>0$. By continuity, we deduce that
on $A_K$, there exist
$r,s \in \QQ_+$ and $\e\in \QQ\cap (0,1]$ such that $r<s<\zeta$,
$X_t \in G_{K,\e}$ for all $t \in [r,s]$ and such that $\{t\in(r,s) :  R_K(X_t)=0\}$ has an isolated point.
It thus suffices that for all $r<s$ and all $\e\in(0,1]$, that we all fix from now on, \blue quasi-everywhere, \bla
$\PP^X_x(A_{K,r,s,\e})=0$, where
\begin{align*}
&A_{K,r,s,\e}=\Big\{ \hbox{$X_t \in G_{K,\e}$ for all $t \in (r,s)$ 
and $\{t\in(r,s) :  R_K(X_t)=0\}$ has an isolated point} \Big\}.
\end{align*}
By the Markov property, it suffices that $\PP^X_x(A_{K,0,s,\e})=0$ \blue quasi-everywhere in $G_{K,\e}$ \bla and, 
by equivalence, that $\QQ^{s,\e,K}_{  x  }(A_{K,0,s,\e})=0$ \blue quasi-everywhere in $G_{K,\e}$. \bla 
We write, recalling the preliminaries,
\begin{align*}
\QQ^{s,\e,K}_x(A_{K,0,s,\e})=&\QQ^{s,\e,K}_x\Big(\tau_{K,\e} \geq s
\hbox{ and $\{t\in(0,s) :  R_K(X_t)=0\}$ has an isolated point}\Big)\\
=&\PP\Big(\tilde \tau_{K,\e} \geq s
\hbox{ and $\{t\in(0,s) :  R_K(Y_t)=0\}$ has an isolated point}\Big)\\
\leq & \PP\Big(\hbox{ $\{t\in(0,s) :  R_K(Y_t)=0\}$ has an isolated point}\Big).
\end{align*}
But $(Y_t)_{t\geq 0}$ is a $QKS(|K|\theta/N, |K|)$-process,
so that we know from Lemma~\ref{besbro} that $(R_{  K  }(Y_t))_{t\geq 0}$
is a   squared   Bessel process with dimension $d_{|K|\theta/N, |K|}(|K|)=d_{\theta,N}(|K|)\in (0,2)$.
Such a process has no isolated zero, see Revuz-Yor \cite[Chapter XI]{ry}.

\vip

{\it Point (i).} We have already seen in Proposition~\ref{cont}-(ii) that \blue quasi-everywhere, \bla
$\PP^X_x$-a.s.,
$\zeta<\infty$ and $X_{\zeta-}=\lim_{t\to \zeta-}X_t$ exists in $(\rr^2)^N$ and does not belong to $E_{k_0}$.

\vip

{\it Point (ii).} We want to show that \blue quasi-everywhere, \bla $\PP^X_x$-a.s.,
there is $K_0 \subset \ig 1,N \id$ with $|K_0|=k_0$ such that
there is a $K_0$-collision   and no $K$-collision with $|K| > k_0$   in the configuration $X_{\zeta-}$.
We already know that $X_{\zeta-} \notin E_{k_0}$,
so that there is $K \subset \ig 1,N \id$ with $|K|\geq k_0$ such that
there is a $K$-collision in the configuration $X_{\zeta-}$.
Hence the goal is to verify that \blue quasi-everywhere, \bla
for all $K\subset \ig 1,N \id $ with $|K|> k_0$, $\PP^X_x(B_K)=0$, where
$$
B_K=\{\hbox{There is a $K$-collision in the configuration } X_{\zeta-}\}.
$$
On $B_K$, there is $\e\in \QQ\cap(0,1]$ such that $X_{\zeta-}\in G_{K,2\e}$.
By continuity, there also exists, still on $B_K$, 
some $r \in \QQ_+\cap [0,\zeta)$ such that $X_t\in G_{K,\e}$ for all $t \in [r,\zeta)$.
Hence we only have to prove that for all $\e \in \QQ\cap(0,1]$, all $t\in \QQ_+$, all $T \in \QQ_+$ such that
$T>r$, \blue quasi-everywhere, \bla $\PP^X_x(B_{K,r,T,\e})=0$, where
$$
B_{K,r,T,\e}=\{\zeta \in (r,T], \; X_t \in G_{K,\e} \hbox{ for all $t \in [r,\zeta)$ and } R_K(X_{\zeta-})=0\}.
$$
By the Markov property, it suffices that  $\PP^X_x(B_{K,0,T,\e})=0$ \blue quasi-everywhere in $G_{K,\e}$, \bla 
for all $\e\in \QQ\cap(0,1]$ and all $T\in \QQ_+^*$.
We now fix $\e\in \QQ\cap(0,1]$ and $T\in \QQ_+^*$. By equivalence, it suffices to prove that 
$\QQ^{T,\e,K}_x(B_{K,0,T,\e})=0$. Using the notation introduced in the preliminaries, we write
\begin{align*}
\QQ^{T,\e,K}_x(B_{K,0,T,\e})=&\QQ^{T,\e,K}_x\Big(\zeta \leq T, \; \tau_{K,\e} = \zeta \hbox{ and } R_K(X_{\zeta-})=0\Big)\\
=&\PP\Big(\zeta'\leq T, \; \tilde \tau_{K,\e}=\zeta' \hbox{ and } R_K(Y_{\zeta'-})=0\Big)\\
\leq & \PP\Big(\inf_{t\in [0,\zeta^Y)}R_K(Y_t)=0\Big).
\end{align*}
But $(Y_t)_{t\geq 0}$ is a $QKS(|K|\theta/N, |K|)$-process with $|K|>k_0\geq 7$ and with
$d_{|K|\theta/N, |K|}(|K|-1)=d_{\theta,N}(|K|-1)\leq 0$ by Lemma~\ref{dimensions}
because $|K|-1\geq k_0$. We also have $d_{|K|\theta/N, |K|}(|K|)=d_{\theta,N}(|K|)\leq 0$.
Hence  Proposition~\ref{N-1toN}-(i) tells us that
$\PP(\inf_{t\in [0,\zeta^Y)}R_K(Y_t)=0)=0$. 

\vip

{\it Point (iii).} We recall that $k_1=k_0-1$ and we fix $L\subset K \subset \ig 1,N\id$ with
$|K|=k_0$ and $|L|=k_1$. We want to prove that \blue quasi-everywhere, \bla 
$\PP^X_x$-a.s., if $R_{K}(X_{\zeta-})=0$, then
for all $t \in [0,\zeta)$, the set 
$\cZ_L \cap (t,\zeta)$ is infinite and has no isolated point. But since
$d_{\theta,N}(k_1) \in (0,2)$, see Lemma~\ref{dimensions}, we already
know that $\cZ_L$ has no isolated point. It thus suffices to check that \blue quasi-everywhere, \bla 
for all $r \in \QQ_+$, we have 
$\PP^X_x(C_{K,L,r})=0$, where 
$$
C_{K,L,r}=\{ \zeta>r,\; R_K(X_{\zeta-})=0, \hbox{ and } R_L(X_t)>0 \hbox{ for all }t\in (r,\zeta)\}.
$$
We used that since $|L|=k_1  = k_0-1  $, for all $x\in \cX=E_{k_0}$, there is a $L$ collision
in the configuration $x$ if and only if $R_L(x)=0$. 
\vip
On $C_{K,L,r}$,   thanks to point (ii) , there are $\e \in \QQ \cap (0,1]$, $T\in \QQ_+$ and 
$s\in \QQ_+^*\cap[r,\zeta)$ such that
$\zeta \in (s,T]$ and 
$X_t \in G_{K,\e}$ for all $t \in [s,\zeta)$. Thus it suffices to prove that for all $s<T$ and
all $\e\in (0,1]$, that we now fix, \blue quasi-everywhere, \bla 
$\PP^X_x(C_{K,L,s,T,\e})=0$, where
$$
C_{K,L,s,T,\e}=\{ \zeta\in (s,T],\; R_K(X_{\zeta-})=0, \; X_t \in G_{K,\e}
\hbox{ and } R_L(X_t)>0 \hbox{ for all }t\in [s,\zeta)\}.
$$
By the Markov property, it suffices that  $\PP^X_x(C_{K,L,0,T,\e})=0$ \blue quasi-everywhere in $G_{K,\e}$ \bla and,
by equivalence, that $\QQ^{T,\e,K}_x(C_{K,L,0,T,\e})=0$. Recalling the preliminaries, we write
\begin{align*}
\QQ^{T,\e,K}_x(C_{K,L,0,T,\e})=&\QQ^{T,\e,K}_x\Big(\zeta \leq T, \; R_K(X_{\zeta-})=0,\;
\tau_{K,\e} = \zeta \hbox{ and } R_L(X_t)>0
\hbox{ for all } t\in [0,\zeta)\Big)\\
=&\PP\Big(\zeta'\leq T, \; R_K(Y_{\zeta'-})=0,\;
\tilde \tau_{K,\e} = \zeta' \hbox{ and } R_L(Y_t)>0
\hbox{ for all } t\in [0,\zeta')\Big).
\end{align*}
Setting $\sigma_K=\inf\{t>0 : R_K(Y_t)=0\}$, we observe that
$\sigma_K=\zeta^Y$. Indeed,
$|K|=k_0$ and
$(Y_t)_{t\geq 0}$ is a $QKS(|K|\theta/N, |K|)$-process, of which the state space is given by $\cY_\triangle=
\cY\cup\{\triangle\}$,
where $\cY=\{y \in (\rr^2)^{|K|} : R_{ M}(y)>0$ for all $ M \subset \ig 1,N\id$ such that
$| M |\geq k_0\}$, because $\lceil 2|K|/(|K|\theta/N)\rceil=\lceil 2N/\theta\rceil=k_0$. Hence
$\{R_K( Y_{\zeta'-})=0\}\subset \{\zeta'=\sigma_K\}$, so that
\begin{align*}
\QQ^{T,\e,K}_x(C_{K,L,0,T,\e})\leq \PP(R_L(Y_t)>0 \hbox{ for all } t\in [0,\sigma_K)).
\end{align*}
This last quantity equals zero by Proposition~\ref{N-1toN}-(ii), since 
$d_{|K|\theta/N, |K|}(|K|-1)=d_{\theta,N}(|K|-1)=d_{\theta,N}(k_0-1)\in (0,2)$ by Lemma~\ref{dimensions}
and since $|L|=k_1=|K|-1$ and since $d_{|K|\theta/N, |K|}(|K|)=d_{\theta,N}(|K|)=d_{\theta,N}(k_0)\leq 0 <2 $.

\vip

{\it Point (iv).} We assume that $k_2=k_0-2$, i.e. that $d_{\theta,N}(k_0-2)\in (0,2)$.
We fix $L\subset K \subset \ig 1,N\id$ with
$|K|=k_1$ and $|L|=k_2$. We want to prove that \blue quasi-everywhere, \bla 
$\PP^X_x$-a.s., for all $t\in [0,\zeta)$, if there is a $K$-collision
in the configuration $X_t$, then 
for all $r \in [0,t)$, the set 
$\cZ_L \cap (r,t)$ is infinite and has no isolated point. We already
know that $\cZ_L$ has no isolated point. 
It thus suffices to check that \blue quasi-everywhere, \bla
for all $r \in \QQ_+$, we have $\PP^X_x(D_{K,L,r})=0$, where
\begin{align*}
&D_{K,L,r}=\{ \zeta>r \hbox{ and there is $t\in (r,\zeta)$ such that there is a $K$-collision at time $t$}\\
&\hskip9cm \hbox{but no $L$-collision during $(r,t)$}\}.
\end{align*}
We set $\sigma_{K,r}=\inf\{t>r :$ there is a $K$-collision in the configuration $X_t\}$.
It holds that
$$
D_{K,L,r}=\{\zeta>r, \; \sigma_{K,r}<\zeta \hbox{ and there is no $L$-collision during }
u\in [r,\sigma_{K,r})\}.
$$
On $D_{K,L,r}$, there exists $\e\in \QQ\cap(0,1]$ such that $X_{\sigma_{K,r}}\in G_{K,2\e}$, so that by continuity, there
exists $v\in \QQ_+\cap [r,\sigma_{K,r})$ such that $X_{u}\in G_{K,\e}$ for all $u\in [v,\sigma_{K,r}]$.
Observe that $\sigma_{K,v}=\sigma_{K,r}$ and that for all $t\in [v,\sigma_{K,v})$,
there is a $L$-collision at time $t$ if and only if $R_L(X_t)=0$, by definition of $\sigma_{K,v}$
and since $X_t\in G_{K,\e}$.
All in all, it suffices to prove that for all $v\in \QQ_+$, all $\e\in \QQ\cap (0,1]$, all $T\in \QQ_+^*$,
$\PP^X_x(D_{K,L,v,T,\e})=0$ \blue quasi-everywhere, \bla where
$$
D_{K,L,v,T,\e}=\{\zeta\in(v,T],\; \sigma_{K,v}<\zeta,\; X_{u}\in G_{K,\e} \hbox{ and }
R_L(X_u)>0 \hbox{ for all } u\in [v,\sigma_{K,v})\}.
$$
By the Markov property, it suffices to prove that $\PP^X_x(D_{K,L,0,T,\e})=0$ 
\blue quasi-everywhere in $G_{K,\e}$ \bla
and, by equivalence, we may use $\QQ^{T,\e,K}_x$ instead of $\PP^X_x$. But recalling the preliminaries,
\begin{align*}
\QQ^{T,\e,K}_x(D_{K,L,0,T,\e})=&\QQ^{T,\e,K}_x\Big(\zeta \leq T, \; \sigma_{K,0}<\zeta,\; \tau_{K,\e}\geq \sigma_{K,0}
\;\hbox{ and }\; R_L(X_t)>0 \hbox{ for all } t\in [0,\sigma_{K,0})\Big)\\
=&\PP\Big(\zeta'\leq T, \; \tilde \sigma_{K,0}<\zeta',\; \tilde \tau_{K,\e}\geq \tilde \sigma_{K,0}
\;\hbox{ and }\; R_L(Y_t)>0 \hbox{ for all } t\in [0,\tilde\sigma_{K,0})\Big)\\
\leq& \PP\Big(R_L(Y_t)>0 \hbox{ for all } t\in [0,\tilde\sigma_{K,0})\Big),
\end{align*}
where we have set $\tilde\sigma_{K,0}=\inf\{t>0 : R_K(Y_t)=0\}$.
Finally, $\PP(R_L(Y_t)>0 \hbox{ for all } t\in [0,\tilde\sigma_{K,0}))=0$ by 
 Proposition~\ref{N-1toN}-(ii),
because $(Y_t)_{t\geq 0}$ is a $QKS(|K|\theta/N, |K|)$-process, because $|L|=k_2=|K|-1$,
because $d_{|K|\theta/N, |K|}(|K|-1)=d_{\theta,N}(|K|-1)=d_{\theta,N}(k_2)\in (0,2)$
and because $d_{|K|\theta/N, |K|}(|K|)=d_{\theta,N}(|K|)=d_{\theta,N}(k_1)\in (0,2)$.

\vip

{\it Point (v).} We fix $K\subset \ig 1,N\id$ with cardinal $|K|\in \ig3,  k_2-1   \id$, so that
$d_{\theta,N}(|K|)\geq 2$. We want to prove that \blue quasi-everywhere, \bla $\PP^X_x$-a.s.,
for all $t\in [0,\zeta)$, there is no $K$-collision in the configuration $X_t$.
We introduce $\sigma_K=\inf\{t>0 : $ there is a $K$-collision in the configuration $X_t\}$,
with the convention that $\inf\emptyset = \zeta$, and we have to verify that \blue quasi-everywhere, \bla
$\PP^X_x(\sigma_K<\zeta)=0$.

\vip

On the event $\{\sigma_K<\zeta\}$, there exist $\e \in \QQ\cap(0,1]$ and $r\in \QQ_+^*\cap[0,\sigma_K)$
such that $X_t \in G_{K,\e}$ for all $t \in [r,\sigma_K]$. Hence it suffices to check that
for all $\e \in \QQ\cap(0,1]$, all $r\in \QQ_+^*$ and all $T \in \QQ_+^*\cap(r,\infty)$, which we now fix, 
\blue quasi-everywhere, \bla $\PP^X_x(F_{K,r,T,\e})=0$, where
$$
F_{K,r,T,\e}=\{\sigma_K \in (r,\zeta \land T) 
\hbox{ and $X_t \in G_{K,\e}$ for all $t \in [r,\sigma_K]$}\}.
$$
By the Markov property, it suffices that  $\PP^X_x(F_{K,0,T,\e})=0$ \blue quasi-everywhere in $G_{K,\e}$ \bla and,
by equivalence, that $\QQ^{T,\e,K}_x(F_{K,0,T,\e})=0$. Recalling the preliminaries, we write
\begin{align*}
\QQ^{T,\e,K}_x(F_{K,0,T,\e})=&\QQ^{T,\e,K}_x\Big(\sigma_K \in ( 0,\zeta\land T) 
\hbox{ and } \tau_{K,\e}\geq \sigma_K\Big) \\
=&\PP\Big(\tilde \sigma_K \in ( 0 ,\zeta'\land T) \hbox{ and } \tilde \tau_{K,\e}\geq \tilde \sigma_K\Big)\\
\leq & \PP\Big(\inf_{t\in [0,T]} R_K(Y_t)=0\Big),
\end{align*}
where we have set $\tilde \sigma_K=\inf\{t>0 : $ there is a $K$-collision in the configuration $(Y_t,Z_t)\}$.
Since $(Y_t)_{t\geq 0}$ is a $QKS(|K|\theta/N, |K|)$-process, we know from Lemma~\ref{besbro}
that $(R_K(Y_t))_{t\geq 0}$ is a squared Bessel process with dimension
$d_{|K|\theta/N, |K|}(|K|)=d_{\theta,N}(|K|)\geq 2$. Such a process does a.s. never
reach $0$.

\vip

{\it Point (vi).} The proof is exactly the same as that of (iv), replacing everywhere $k_1$
by $k_2$ and $k_2$ by $2$, and using Proposition~\ref{N-1toN}-(iii) instead of Proposition~\ref{N-1toN}-(ii), 
which is licit because 
${0<d_{k_2\theta/N, k_2}(k_2)<2\leq d_{k_2\theta/N, k_2}(k_2-1)}$, since
$d_{k_2\theta/N, k_2}(k_2)=d_{\theta,N}(k_2)$ and $d_{k_2\theta/N, k_2}(k_2-1)=d_{\theta,N}(k_2-1)$ 
and by Lemma~\ref{dimensions}.
\end{proof}

\section{Extension to all initial conditions in $E_2$}\label{rq}

We first prove Proposition~\ref{densite}: we can build a $KS(\theta,N)$-process, i.e. a 
$QKS(\theta,N)$-process such that $\PP_x^X\circ X_t^{-1}$ is absolutely continuous for all $x \in E_2$ and all
$t>0$. We next conclude the
proofs of Proposition~\ref{reloumaisafaire} and of 
Theorem~\ref{theoremecollisions}.

\subsection{Construction of a $KS(\theta,N)$-process}\label{existenceKS}

We fix $\theta>0$ and $N\geq 2$ such that $N>\theta$ during the whole subsection.
For each $n\in \nn^*$, we introduce $\phi_n \in C^\infty ( \rr _+, \rr _+^*)$ such that 
$\phi_n (r) =r$ for all $r\geq 1/n$ and we set, for $x \in (\rr^2)^N$,
$$
\bm _n (x) = \prod _{1\leq i\neq j \leq N} [\phi_n(\|x^i-x^j\|^2)]^{-\theta/N} 
\qquad \hbox{and} \qquad \mu_n(\dd x)= \bm _n (x)\dd x.
$$
We then consider the $(\rr^2)^N$-valued S.D.E
\begin{align}\label{EDSn}
X^{n}_t = x + B_t + \intot \frac{\nabla \bm _n (X^n_s)}{2\bm _n(X^n_s)} \dd s,
\end{align}
which is strongly well-posed, for every initial condition, since the drift coefficient is smooth and bounded. 
We denote by $\xX^n = (\Omega^{n}, \cM^{n}, (X^n_t)_{t\ge 0}, (\PP^{n}_x)_{x \in (\rr^2)^N})$ the corresponding 
Markov process.

\begin{lemma}\label{dirichletn}
For all $n\ge 1$, $\xX^n$ is a $\mu_n$-symmetric $(\rr^2)^N$-valued \blue diffusion \bla with regular 
Dirichlet space $(\cE^n,\cF^n)$ with core $C_c^\infty ((\rr^2)^N)$ such that for 
all $\varphi\in C_c^\infty ((\rr^2)^N)$,
$$
\cE^n(\varphi,\varphi) =\frac 12 \int_{(\rr^2)^N} \|\nabla \varphi \|^2 \dd \mu_n.
$$
Moreover $\PP^n_x \circ (X_t^n)^{-1}$ has a density with respect to the Lebesgue measure on $(\rr^2)^N$
for all $t>0$ and all $x \in (\rr^2)^N$.
\end{lemma}

\begin{proof}
Classically, $\xX^n$ is a $\mu_n $-symmetric \blue diffusion \bla and
its (strong) generator $\cL^n$ satisfies that for all $\varphi \in C^\infty_c((\rr^2)^N)$, all $x\in (\rr^2)^N$,
$\cL^n \varphi(x)= \frac 12 \Delta \varphi(x) + \frac{\nabla \bm_n(x)}{2\bm_n(x)} \cdot \nabla \varphi(x)$.
Hence, see Subsection~\ref{ap1}, one easily shows that for $(\cE^n,\cF^n)$ the Dirichlet
space of $\xX^n$, we have $C_c^\infty ((\rr^2)^N)\subset \cF^n$ and, for $\varphi \in C_c^\infty ((\rr^2)^N)$,
$\cE^n(\varphi,\varphi) =\frac 12 \int_{(\rr^2)^N} \|\nabla \varphi \|^2 \dd \mu_n$.
Since $(\cE^n,\cF^n)$ is closed, we deduce that
$$
\overline{C_c^\infty ((\rr^2)^N)}^{\cE^n _1}  \subset \cF^n,
$$
where $\cE^n_1 (\cdot, \cdot) = \cE^n(\cdot, \cdot ) + \|\cdot \|^2_{L^2((\rr^2)^N, \mu_n)}$. 
But thanks to \cite[Lemma 3.3.5 page 136]{f},
$$
\cF^n \subset \{ \varphi\in L^2((\rr^2)^N, \mu_n) : \nabla \varphi \in L^2((\rr^2)^N, \mu_n)\},
$$
where $\nabla $ is understood in the sense of distributions. Since finally
$$
\overline{C_c^\infty ((\rr^2)^N)}^{\cE^n _1}=\{ \varphi \in L^2((\rr^2)^N, \mu_n) : 
\nabla \varphi \in L^2((\rr^2)^N, \mu_n)\},
$$
$\xX^n$ has the announced Dirichlet space. Finally, the absolute continuity of $\PP^n_x \circ (X_t^n)^{-1}$,
for $t>0$ and $x\in (\rr^2)^N$, immediately follows from the (standard) Girsanov theorem,
 since the drift coefficient is bounded. 
\end{proof}

For all $x\in E_2$ we set $d_x = \min_{i\neq j} \|x^i-x^j\|^2$. For $n\ge 1$,
we introduce the open set
\begin{equation}\label{e2n}
E_2^{n} = \Big\{   x   \in (\rr^2)^N: d_{  x  } > \frac 1n \; \hbox{ and } \;  
||  x  ||<n\Big\}.
\end{equation}
We also fix a $QKS(\theta,N)$-process $\xX=(\Omega ^X,\cM^X, (X_t)_{t\ge 0}, (\PP^X_x)_{x\in \cX_\triangle})$ 
for the whole subsection.

\begin{lemma}\label{mmloi}
There exists an exceptional set $\cN_0\subset E_2$ with respect to $\xX$ such that for all
$n\geq 1$, for all $x \in E_2^n\setminus \cN_0$, the law of $(X^n_{t\land \tau_n})_{t\geq 0}$ under $\PP^n_x$ 
equals the law of $(X_{t\land\sigma_n})_{t\geq 0}$ under $\PP^X_x$, where
$$
\tau_n = \inf \{ t>0 : X_t^n \notin E^n_2 \} \qquad \hbox{and} \qquad
\sigma_n = \inf \{ t>0 : X_t \notin E^n_2 \}.
$$
\end{lemma}

\begin{proof}
We fix $n\geq 1$.
Applying Lemma~\ref{tuage} to $\xX^n$ and $\xX$ with the open set $E^n_2$, using that $\bm _n = \bm$ on $ E_2^n$
and Lemma~\ref{dirichletn}, we find that the processes $\xX^n$ and $\xX$ killed when leaving $E^n_2$
have the same Dirichlet space.
By uniqueness, see \cite[Theorem 4.2.8 page 167]{f},
there exists an exceptional set $\cN_n$ such that for all $x\in E^n_2 \setminus \cN_n$, 
the law of $(X^n_t)_{t\geq 0}$ killed when leaving $E^n_2$ under $\PP^n_x$ equals the law
of $(X_t)_{t\ge 0}$ killed when leaving $E^n_2$ under $\PP^X_x$.
We conclude setting $\cN_0 = \cup_{n\ge 1}\cN_n$. 
\end{proof}

\begin{lemma}\label{quasilaw}
For all exceptional set $\cN$ with respect to $\xX$, all $n\geq 1$ and all $x\in E^n_2$,
we have $\PP^n_x(X^n_{\tau_n}\notin \cN) = 1$.
\end{lemma}

\begin{proof}
We fix $\cN$ an exceptional set with respect to $\xX$,
$n\geq 1$ and $x\in E^n_2 $. For $\e \in (0,1]$, we write
$$
\PP^n_x ( X^n_{\tau_n} \in \cN ) \le \PP^n_x ( \tau_n \leq \e ) 
+ \PP^n_x ( \tau_n>\e, X^n_{\tau_n} \in \cN )
=\PP^n_x ( \tau_n \leq \e ) +\E^n _x [ \indiq _{\{\tau_n>\e\}} \PP^n_{X_\e^n} (X^n_{\tau_n} \in \cN )]
$$
by the Markov property. But by Lemma~\ref{mmloi}, for all $y \in E^n_2 \setminus \cN_0$,
the law of 
$(X^n_{t\land\tau_n})_{t\ge 0}$ under $\PP^n_y$ is equal to the law of $(X_{t\land\sigma_n})_{t\ge 0}$ under $\PP^X_y$.
Since $\cN_0\cup \cN$ is exceptional for $\xX$, we can find $\cN'\supset \cN_0\cup \cN$ properly exceptional for $\xX$
 (see Subsection~\ref{ap1}).
Hence for all $y\in E^n_2\setminus \cN'$, 
$$
\PP^n_y ( X^n_{\tau_n} \in \cN) \leq 
\PP^n_y ( X^n_{\tau_n} \in \cN') = \PP^X_y( X_{\sigma_n} \in \cN') = 0.
$$
Since $\PP^n_x\circ (X^n_\e)^{-1}$ has a density by Lemma~\ref{mmloi}, we conclude that
$\PP^n_x(X^n_\e \in \cN')=0$ and thus that  $\PP^n_x$-a.s., we have $\PP^n_{X_\e^n} (X^n_{\tau_n} \in \cN )=0$.
All in all, we have proved that $\PP^n_x ( X^n_{\tau_n} \in \cN ) \le \PP^n_x ( \tau_n \leq \e )$,
and it suffices to let $\e\to 0$, since $\PP^n_x(\tau_n>0)=1$ by continuity and since $x\in E_2^n$.
\end{proof}

Using Lemmas~\ref{mmloi} and \ref{quasilaw},
it is slightly technical but not difficult to build from $\xX$ and the family $(\xX^n)_{n\geq 1}$ 
a $\cXt$-valued \blue diffusion \bla
$\tilde\xX = (\tilde \Omega ^X, \tilde\cM^X, (\tilde X_t)_{t\ge 0}, (\tilde\PP^X_x)_{x\in \cX_\triangle})$
such that 
\vip
$\bullet$ for all $x\in \cXt\setminus \cN_0$, the law of $(\tilde X_t)_{t\geq 0}$ under $\tilde\PP^X_x$
equals the law of $(X_t)_{t\geq 0}$ under $\PP^X_x$,
\vip
$\bullet$ for all $x\in \cN_0$, setting $n=1+\lfloor \max(1/d_x,||x||) \rfloor$ (so that $x\in E^n_2$), the law of 
$(\tilde X_{t\land \tilde\sigma_n})_{t\geq 0}$ under $\tilde\PP^X_x$ is the same as that of 
$(X^n_{t\land \tau_n})_{t\geq 0}$ under $\PP^n_x$ and the law of 
$(\tilde X_{\tilde\sigma_n+t})_{t\geq 0}$ under $\tilde\PP^X_x$ conditionally on $\tilde \cM^X_{\tilde\sigma_n}$ 
equals the law of $(X_{t})_{t\geq 0}$ under 
$\PP^X_{\tilde X_{\sigma_n}}$. We have used the notation
$\tilde\sigma_n=\inf\{t>0 : \tilde X_t\notin E^n_2\}$ and $\tilde\cM^X_t=\sigma(\tilde X_s : s \in [0,t])$.

\begin{rk}\label{rrd}
For all $x \in E_2$, setting $n=1+\lfloor \max(1/d_x,||x||) \rfloor$, 
the law of $(\tilde X_{t\land \tilde\sigma_n})_{t\geq 0}$ 
under $\tilde\PP^X_x$ is the same as that of 
$(X^n_{t\land \tau_n})_{t\geq 0}$ under $\PP^n_x$.
\end{rk}

\begin{proof}
This follows from Lemma~\ref{mmloi} when $x \in E_2\setminus \cN_0$ and from the definition 
of $\tilde\xX$ otherwise.
\end{proof}

We can finally give the

\vip

\begin{proof}[Proof of Proposition~\ref{densite}]
We fix $N\geq 2$ and $\theta>0$ such that $N>\theta$ and we 
prove that $\tilde\xX$ defined above is a $KS(\theta,N)$-process. First, it is clear that $\tilde\xX$ is a
$QKS(\theta,N)$-process because $\tilde\xX$ is a $\cXt$-valued \blue diffusion \bla and since 
for all $x\in \cXt \setminus \cN_0$, the law of $(\tilde X_t)_{t\geq 0}$ under $\tilde\PP^X_x$
equals the law of $(X_t)_{t\geq 0}$ under $\PP^X_x$, with $\cN_0$ exceptional for $\xX$.
It remains to prove that for all $x \in E_2$, all $t>0$ and all Lebesgue-null $A\subset (\rr^2)^N$,
we have $\tilde \PP^X_x ( \tilde X_t \in A)=0$. We set $n=1+\lfloor \max(1/d_x,||x||)\rfloor$ and write, 
for any 
$\e\in(0,t)$,
$$
\tilde \PP^X_x ( \tilde X_t \in A) \leq \tilde \PP^X_x ( \tilde\sigma_n >\e , \tilde X_t \in A) + 
\tilde \PP^X_x (\tilde \sigma_n\le \e)=
\tilde \E ^X_x [ \indiq _{\{\tilde \sigma_n >\e\}} \tilde \PP^X_{\tilde X_\e} ( \tilde X_{t-\e} \in A)]
+\tilde \PP^X_x (\tilde \sigma_n\le \e).
$$
Since $\tilde\xX$ is $\mu$-symmetric (because it is a $QKS(\theta,N)$-process),
since $\tilde P_{t-\e}1 \leq 1$, where $\tilde P_t$ is the semi-group of $\tilde\xX$
and since $A$ is Lebesgue-null,
$$
\int _{(\rr^2)^N} \tilde \PP _y ( \tilde X_{t-\e} \in A ) \mu (\dd y ) \leq \mu(A) = 0.
$$
Hence there is a Lebesgue-null subset $B$ of $(\rr^2)^N$ (depending on $t-\e$) such that 
$\tilde \PP_y ( \tilde X_{t-\e} \in A )=0$ for every $y \in (\rr^2)^N\setminus B$. We conclude that
$$
\tilde \PP^X_x ( \tilde X_t \in A) \leq \tilde \PP^X_x ( \tilde\sigma_n >\e , \tilde X_\e \in B)
+\tilde \PP^X_x (\tilde \sigma_n\le \e)
 = \PP^n_x ( \tau_n >\e , X_\e^n \in B)
+\tilde \PP^X_x (\tilde \sigma_n\le \e), 
$$
 where we finally used Remark~\ref{rrd}.
Since $B$ is Lebesgue-null, we deduce from Lemma~\ref{dirichletn} that $\PP^n_x ( \tau_n >\e , X_\e^n \in B)=0$.
  Thus
$\tilde \PP^X_x ( \tilde X_t \in A) \leq \tilde \PP^X_x (\tilde \sigma_n\le \e)$,
which tends to $0$ as $\e\to 0$ because $\tilde\PP^X_x(\tilde \sigma_n>0)=1$ by continuity.
\end{proof}

\subsection{Final proofs}

We fix $\theta > 0$, $N\ge 2$ such that $N>\theta$ and a $KS (\theta,N)$-process $\xX$, 
which exists thanks to Subsection~\ref{existenceKS}. We recall that $E^n_2$ was introduced in
\eqref{e2n} and define, for all $n\ge 1$,
$\sigma_n = \inf \{t \ge 0:  X_t \notin E^n_2\}$, as well as the $\sigma$-field
$$
\cG=\cap_{n\ge 1} \sigma (  X_{ \sigma_n + t}, t\ge 0).
$$
 
\begin{lemma}\label{quasitotout}
Fix $A \in \cG$. If $\PP^X_x(A)=0$ \blue quasi-everywhere, \bla then $\PP^X_x(A) = 0$ for all $x \in E_2$.
\end{lemma}

\begin{proof}
We fix $A \in \cG$ such that $\PP^X_x(A)=0$ \blue quasi-everywhere. \bla
There is an exceptional set $\cN$ such that for all 
$x \in E_2 \setminus \cN$, $\PP^X_x(A) = 0$. We now fix $x \in E_2$ and set 
$n= 1 + \lfloor \max(1/d_x,||x||)\rfloor$.
For any $\e \in (0,1]$,
$$
\PP^X_x( A) \le  \PP^X_x (\sigma_n\leq \e )+  \PP^X_x [\sigma_n>\e, A].
$$
By the Markov property and since $A \in \cG\subset \sigma (  X_{ \sigma_n + t}, t\ge 0)$, we 
get 
$$
\PP^X_x [\sigma_n>\e, A]=\E^X_x[\indiq _{\{\sigma_n>\e\}} \PP^X_{X_\e} (A)].
$$ 
But the law of $X_\e$ under $\PP^X_x$ has a density, so that 
$\PP^X_x(X_\e \in \cN)=0$, whence $\PP^X_x(\PP^X_{X_\e} (A)=0)=1$. 
Hence $\PP^X_x [\sigma_n>\e, A]=0$ and we end with 
$\PP^X_x( A) \le  \PP^X_x (\tau_{n}\leq \e )$. As usual, we conclude that $\PP^X_x( A)=0$ by letting
$\e\to 0$.
\end{proof}

We are now ready to give the

\begin{proof}[Proof of Proposition~\ref{reloumaisafaire}]
Let $\theta \in (0,2)$ and $N\ge 2$. Since our $KS (\theta,N)$-process $\xX$ is a $QKS(\theta,N)$-process,
we know from Proposition~\ref{cont}-(i) that $\PP^X_x(\zeta=\infty)=1$  \blue quasi-everywhere. \bla
We want to prove that $\PP^X_x(\zeta=\infty)=1$ for all $x\in E_2$.
By Lemma~\ref{quasitotout}, it thus suffices to check that $\{\zeta=\infty\}$ belongs to $\cG$, which is not
hard since for each $n\geq 1$,
$$
\{\zeta=\infty\}=\{X_t\in \cX \hbox{ for all } t\geq 0\}=\{X_t\in \cX \hbox{ for all } t\geq \sigma_n\}
\in  \sigma (  X_{ \sigma_n + t}, t\ge 0).
$$
For the second equality, we used that $X_t \in \bar E^n_2\subset \cX$ for all $t\in [0,\sigma_n]$
by definition.
\end{proof}

\begin{proof}[Proof of Theorem~\ref{theoremecollisions}]
Let $\theta \geq 2$ and $N>3\theta$. Since our $KS (\theta,N)$-process $\xX$ is a $QKS(\theta,N)$-process,
we know from Section~\ref{conccc} that all the conclusions of Theorem~\ref{theoremecollisions} hold  
\blue quasi-everywhere. \bla
In other words, $\PP_x^X(A)=1$  \blue quasi-everywhere, \bla where $A$ 
is the event on which we have $\zeta<\infty$, $X_{\zeta-}=\lim_{t\to \zeta-}X_t \in (\rr^2)^N$, there is
$K_0 \in \ig 1,N\id$ with cardinal $|K_0|=k_0$ such that there is a $K_0$-collision in the configuration
$X_{\zeta-}$, etc. We want to prove that $\PP^X_x(A)=1$ for all $x\in E_2$.
By Lemma~\ref{quasitotout}, it thus suffices to check that $A$ belongs to $\cG$.
But for each $n\geq 1$, $A$ indeed belongs to $\sigma (  X_{ \sigma_n + t}, t\ge 0)$, because 
no collision (nor explosion) may happen before getting out of $E^n_2$.
\end{proof}

We end this section with the following remark (that we will not use anywhere).

\begin{rk}\label{nsar}
Fix $\theta \ge 0$ and $N\ge 2$ such that $N> \theta$.
Consider a $KS(\theta,N)$ process $\xX$ and define $\sigma=\inf\{t\geq 0 : X_t \notin E_2\}$.
For all $x \in E_2$, there is some $(\cM^X_t)_{t\geq 0}$-Brownian motion
$((B^i_t)_{t\geq 0})_{i\in \ig 1,N\id}$ (of dimension $2N$) under $\PP^X_x$  
such that for all $t\in [0,\sigma)$, all $i\in \ig1,N\id$,
\begin{equation}\label{eeddss}
X_t^i = x^i + B^i_t - \frac{\theta}{N}\sum_{j\neq i} \intot \frac{X^i_s-X^j_s}{||X^i_s-X^j_s||^2} \dd s.
\end{equation}
\end{rk}

\begin{proof}
It of course suffices to prove the result during $[0,\sigma_n)$, 
where $\sigma_n=\inf\{t\geq 0 : X_t \notin E_2^n\}$. 
For any $x\in E_2^n$ and for a given Brownian motion, the solutions to \eqref{eeddss} and \eqref{EDSn} 
 classically coincide while they remain $E_2^n$,  because their drift coefficients
coincide and are smooth inside $E_2^n$.
Hence, recalling the notation of Subsection~\ref{existenceKS},
it suffices to prove that the semi-groups $P_t(x,\cdot)$ and $P_t^n(x,\cdot)$ of the Markov processes 
$\xX$ and $\xX^n$ killed when getting out of $E_2^n$ coincide for all $x\in E_2^n$.
\vip
By Lemma~\ref{mmloi}, there is an exceptional set $\cN_0$ such that $P_t(x,\cdot)=P_t^n(x,\cdot)$ for all
$x\in E_2^n\setminus \cN_0$. We next fix $x \in E_2^n$.
For any $\e \in (0,t)$, using that $P_\e(x,\cdot)$ has a density and that $\cN_0$ is Lebesgue-null, we easily
deduce that $P_t(x,\cdot)= (P_\e P_{t-\e})(x,\cdot) = (P_\e P^n_{t-\e})(x,\cdot)$. It is then
not difficult, using that $P^n_t$ is Feller, to let $\e\to 0$ and conclude that
indeed, $P_t(x,\cdot)= P^n_{t}(x,\cdot)$.
\end{proof}

\appendix

\section{A few elementary computations}\label{ppre}
\setcounter{thm}{0}
\setcounter{equation}{0}

We recall that
$d_{\theta ,N}(k)=(k-1)(2-\theta k/N)$ for $k\geq 2$ and give the 

\begin{proof}[Proof of Lemma~\ref{dimensions}]
First, \eqref{hh}, which says that $d_{\theta ,N}(k)>0$ if and only if 
$k<k_0=\lceil 2N/\theta\rceil$,
is clear.
We next fix $N >3\theta\geq 6$, so that $k_0 \in \ig 7,N \id$ and 
$d_{\theta,N}(2)=2-2\theta/N \in (4/3,2)$. By concavity of $x\to (x-1)(2-\theta x/N)$, 
it only remains
to check that (i) $d_{\theta,N}(3)\geq 2$, (ii)  $d_{\theta,N}(k_0-3)\geq 2$,  and (iii)  $d_{\theta,N}(k_0-1)<2$.
We introduce $a=2N/\theta > 6$ and observe that $d_{\theta ,N}(k)=2a^{-1}(k-1)(a-k)$
and that $k_0=\lceil a \rceil$.

\vip

For (i), we write $d_{\theta,N}(3)=4a^{-1}(a-3)=4-12 a^{-1}> 2$ since $a>6$.

\vip

For (ii), we have $d_{\theta,N}(k_0-3)=2a^{-1}(\lceil a \rceil -4)(a-\lceil a \rceil+3)$
and we need $(\lceil a \rceil -4)(a-\lceil a \rceil+3) \geq a$.
Writing $a=n+\alpha$ with an integer $n\geq 6$
and $\alpha \in (0,1]$, we need that $(n-3)(2+\alpha)\geq n+\alpha$,
and this holds true
because $2(n-3)\geq n$ and $(n-3)\alpha\geq \alpha$.

\vip
For (iii), we write $d_{\theta,N}(k_0-1)=
2a^{-1}(\lceil a \rceil -2)(a-\lceil a \rceil+1)\leq 2a^{-1}(\lceil a \rceil -2) <2$.
\end{proof}

We next study the reference measure of the Keller-Segel particle system.

\begin{prop}\label{radon}
Let $N\geq 2$ and $\theta>0$ be such that $N> \theta$. Recall that 
$k_0=\lceil 2N/\theta \rceil$  and
the definition \eqref{mmu} of $\mu(\dd x)=\bm(x)\dd x$.
\vip
(i) The measure $\mu$ is Radon on $E_{k_0}$.
\vip
(ii) If $k_0 \leq N$, then $\mu$ is not Radon on $E_{k_0+1}$.
\end{prop}

\begin{proof}
(i)  To show that $\mu $ is radon on $E_{k_0}$, we have to check that 
for all $x=(x^1,\dots,x^N) \in E_{k_0}$, which we now fix, there is an open set 
$O_x \subset E_{k_0}$ such that $x\in O_x$ and $\mu(O_x)<\infty$.
We choose $O_x=\prod _{i=1}^{N} B(x^{i},d_x)$, where the balls are subsets of $\rr^2$ and
where 
$$
d_x = 1 \land  \min \Big\{ \frac{\|x^{i}-x^{j}\|}{3}  :  i,j \in \ig 1,N \id 
\mbox{ such that } 
x^i \neq x^j  \Big\}>0.
$$
We consider the partition $K_1,\dots , K_\ell $ of $\ig 1,N \id$ such that for all 
$p \neq q$ in $\ig 1,\ell\id$, for all $i,j \in K_p$ and all $k \in K_q$, $x^i = x^j$ and $x^i \neq x^k$. 
Since $x \in E_{k_0}$,
it holds that $\max_{p \in \ig 1,\ell \id} |K_p| \leq k_0-1$. By definition of $O_x$ and $d_x$, we see that 
for all $y \in O_x$, for all $p \neq q$ in $\ig 1,\ell\id$, for all $i \in K_p$, all $j\in K_q$,
$$
\|y^i - y^j \| \ge \|x^i-x^j\| -\|x^i-y^i\|-\|x^j-y^j\| \geq \|x^i-x^j\|-2d_x\geq d_x. 
$$
This implies that for some finite constant $C$ depending on $x$, 
for all $y \in O_x$,
$$
\bm(y)=\prod_{1\leq i\ne j\leq N}||y^i-y^j||^{-\theta/N} \leq C \prod_{p=1}^\ell 
\Big(\prod_{i,j \in K_p,i\ne j} 
||y^i-y^j||^{-\theta/N} \Big).
$$
Recall now that $\mu(\dd y)=\bm(y)\dd y$ and that we want to show that $\mu(O_x)<\infty$.
Since 
$x^i=x^j$ for all $i,j\in K_p$ and all $p\in\ig 1, \ell \id$, since $|K_p| \leq k_0-1$,
$d_x\leq 1$ and by a translation argument, we are reduced to show that for any 
$n \in \ig   2  ,k_0-1\id$,
(when $k_0>N$, one could study only $n \in \ig   2  ,N \id$)
$$
I_n=\int _{(B(0,1))^{n}} \Big(\prod _{1 \le i \ne j \le n} \|y^{i}-y^{j}\|^{- \theta/N}\Big)
\dd y^1 \dots \dd y^n <\infty.
$$
 We fix $n \in \ig   2  ,k_0-1\id$ and show that $I_n<\infty$.
Since $\|u\|^2 \ge |u_1u_2|$ for all $u=(u_1,u_2) \in \rr ^2$,  we have 
$I_n \leq J_n^2$, where
$$
J_n=\int _{[-1,1]^{n}} \Big(\prod _{1 \le i\neq j \le n} |t^{i}-t^{j}|^{- \theta/(2N) }\Big) 
\dd t^1\dots\dd t^n.
$$ 
But for all $t^1,\dots,t^n \in \rr$, 
\begin{align*}
\prod _{1\le i \neq j \le n} |t^i-t^j|^{-\theta/(2N)} = 
\prod _{i=1} ^{n} \Big( \prod _{j=1, j\neq i}^{n} |t^i-t^j|^{-\theta/(2N)} \Big) 
\le \frac{1}{n} \sum _{i=1}^{n} \prod _{j=1, j\neq i}^{n} |t^{i}-t^{j}|^{-\theta n/(2N)}
\end{align*}
by the inequality of arithmetic and geometric means. Thus by symmetry,
$$
J_n \le \int_{[-1,1]^{n}} \Big(\prod _{j=2}^{n} |t^{1}-t^{j}|^{- \theta n/(2N) }\Big)\dd t^1\dots\dd t^n
=\int_{-1} ^{1} \Big( \int _{-1}^{1} 
|t^{1}-t^{2}|^{- \theta n/(2N) }\dd t^{2} \Big) ^{n-1}  \dd t^{1}.
$$
Consequently,
$$
J_n \leq \int_{-1}^{1} \Big( \int _{-2}^{2} |s|^{- \theta n/(2N) }\dd s \Big) ^{n-1}  \dd t^{1}.
$$
Since $n\leq k_0-1=\lceil 2N/\theta\rceil-1 < 2N/\theta $, we have $\theta n/(2N)<1$, 
so that $J_n<\infty$, whence 
$I_n<\infty$.

\vip

(ii) We next assume that $k_0 \in \ig   2 ,N\id$. 
To prove that $\mu$ is not radon on $E_{k_0+1}$, we show that $\mu (K) = \infty$
for the  compact subset
$$
K = \prod _{i=1}^{k_0}\overline{B}(0,1) \times \prod _{k = k_0 +1} ^N \overline{B}( (2k,0), 1/2)
$$ 
of $E_{k_0+1}$. All the balls in the previous formula are balls of $\rr^2$. 
For $x=(x^1,\dots,x^N) \in K$,
it holds that $x^{k_0+1},\dots,x^N$ are far from each other and far from $x^1,\dots,x^{k_0}$,
which explains that $K$ is indeed compact in $E_{k_0+1}$. There is a positive constant $c>0$
such that for all $x \in K$,
$$
\bm(x) =\prod_{1\leq i\ne j\leq N}||x^i-x^j||^{-\theta/N} \geq c \prod_{1\le i \neq j \le k_0}
||x^i-x^j||^{-\theta/N},
$$
whence, the value of $c>0$ being allowed to vary, 
$$ 
\mu ( K ) \ge c \int _{(B(0,1))^{k_0}} \Big(\prod _{1\le i \neq j \le k_0} \|x^i-x^j\|^{-\theta/N}
\Big)\dd x^1\dots \dd x^{k_0}.
$$
We now observe that
$$
A=\{x=(x^1,\dots,x^{k_0}) : x^1,x^2 \in B(0,1/3), \; 
\forall i \notin \{1,2 \},\; x^{i} \in B(x^{1}, \| x^{1} - x^{2} \|)\} \subset (B(0,1))^{k_0}
$$
and that for $x \in A$, we have $||x^i-x^j||\leq ||x^i-x^1||+||x^j-x^1||\leq 2||x^1-x^2||$ 
for all $i,j=1,\dots,k_0$, from which
$$
\prod _{1\le i \neq j \le k_0} \|x^i-x^j\|^{-\theta/N} \geq c \|x^1-x^2\|^{-k_0(k_0-1)\theta/N}.
$$ 
As a conclusion, 
\begin{align*}
\mu ( K ) \ge& c \int_{(B(0,1/3))^2} \|x^1-x^2\|^{-k_0(k_0-1)\theta/N} \dd x^1 \dd x^2
\int_{(B(x_1, \| x^{1} - x^{2} \|))^{k_0-2}} \dd x^3\dots\dd x^{k_0}\\
\geq & c \int_{(B(0,1/3))^2} \|x^1-x^2\|^{-k_0(k_0-1)\theta/N+2(k_0-2)} \dd x^1 \dd x^2\\
\geq & c \int_{B(0,1/3)} \|u\|^{-k_0(k_0-1)\theta/N+2(k_0-2)} \dd u,
\end{align*}
where we finally used the change of variables $u = x^{1}-x^{2}$ and $v = x^{1} + x^{2}$.
This last integral diverges, because
$ -k_0(k_0-1)\theta/N+2(k_0-2)=d_{\theta,N}(k_0)-2 \leq -2$, recall
that $d_{\theta,N}(k_0)=(k_0-1)(2-k_0 \theta/N)\leq 0$ by definition of $k_0$.
\end{proof}

We need a similar result on the sphere $\sS$ defined in Section~\ref{nota}, where
$\gamma : \rr^2 \to (\rr^2)^N$ and $\Psi :  \rr^2 \times \rr^*_+ \times \sS \to E_N\subset (\rr^2)^N$
were also introduced. First, we show an explicit link between $\mu(\dd x)=\bm(x)\dd x$ 
and $\beta(\dd u)=\bm(u)\sigma(\dd u)$ 
defined in \eqref{mmu} and \eqref{beta}, that we use several times.

\begin{lemma}\label{changeofvariable}
We fix $N\ge 2$, $\theta >0$ and set $\nu= d_{\theta ,N}(N)/2 -1$.
For all Borel $\varphi : (\rr^2)^N\to \rr_+$,
$$ 
\int _{(\rr ^2)^N} \varphi(x) \mu(\dd x) = \frac12\int_{\rr ^2\times \rr^*_+\times \sS} 
\varphi( \Psi (z,r,u) ) r^\nu  \dd z  
\dd r \beta (\dd u).
$$
\end{lemma}

\begin{proof}
Since $H=\{y=(y^1,\dots,y^N) \in (\rr^2)^N : \sum_1^N y^i=0\}$ and since $\bm$ is 
translation invariant,
$$
\int _{(\rr ^2)^N} \varphi(x) \mu(\dd x)= \int_{(\rr^2)^N} \varphi(x)\bm(x)\dd x 
= \int_{\rr^2\times H} \varphi(\gamma(z)+y)\bm(y) \dd z \dd y.
$$
We next note that $\sS$ is the (true) unit
sphere of the $(2N-2)$-dimensional Euclidean space $H$ and proceed to the
substitution $(\ell ,u) = (\|y\|,y/||y||)$:
$$
\int _{(\rr ^2)^N} \varphi(x) \mu(\dd x)=\int_{\rr^2\times \rr_+\times\sS} 
\varphi(\gamma(z)+\ell u) \bm(\ell u) \ell^{2N-3} 
\dd z \dd \ell \sigma(\dd u).
$$
We finally substitute $\ell=\sqrt r$ and obtain
$$
\int _{(\rr ^2)^N} \varphi(x) \mu(\dd x)= \frac12\int_{\rr^2\times \rr_+\times\sS} 
\varphi(\gamma(z)+\sqrt r u) \bm(\sqrt r u) r^{N-2} 
\dd z \dd r \sigma(\dd u).
$$
But $\bm(\sqrt r u) r^{N-2} = r^{N-2-\theta(N-1)/2}\bm(u)$ by \eqref{mmu} and $\beta(\dd u)=\bm(u)\sigma(\dd u)$,
whence
$$
\int _{(\rr ^2)^N} \varphi(x) \mu(\dd x)= \frac12\int_{\rr^2\times \rr_+\times\sS} 
\varphi(\Psi(z,r,u)) r^{N-2-\theta(N-1)/2} 
\dd z \dd r \beta(\dd u).
$$
Since finally $\nu = d_{\theta ,N}(N)/2 -1=N-2-\theta (N-1)/2$,
the conclusion follows.
\end{proof}

We can now study the measure $\beta$ on $\sS$.

\begin{prop}\label{radonsphere}
Let $N\geq 2$ and $\theta>0$ such that $N> \theta$. 
Recall that $k_0=\lceil 2N/\theta\rceil$.
\vip
(i) The measure $\beta$ is Radon on $\sS \cap E_{k_0}$.
\vip
(ii) If $k_0 \ge N$, then $\beta (\sS ) < \infty$  .  
\end{prop}

\begin{proof}
We start with (i). For $\e \in (0,1]$, we introduce
$$
\cK_\e = \{ x \in (\rr^2)^N : \forall K \subset \ig 1,N \id \mbox{ such that } 
|K| \ge k_0, \mbox{ we have } R_K(x) \ge \e \} \quad \hbox{and} \quad \cL_\e=\cK_\e\cap\sS.
$$
Since $\cK_\e\cap \overline{B}(0,1)$ is compact in $E_{k_0}$, with here $B(0,1)$ 
the unit ball of $(\rr^2)^N$,  
we know from Proposition~\ref{radon}-(i) that $\mu(\cK_\e \cap B(0,1))<\infty$.
Now by Lemma~\ref{changeofvariable},
$$
\mu(\cK_\e \cap B(0,1))=   \frac12  \int_{\rr^2\times\rr_+\times\sS} 
\indiq_{\{\gamma(z)+\sqrt r u \in \cK_\e \cap B(0,1)  \}} 
r^\nu\dd z \dd r \beta(\dd u).
$$
But for $(z,r,u)\in\rr^2\times\rr_+\times\sS$, 
$$
\hbox{$\gamma(z)+\sqrt r u \in \cK_\e \cap B(0,1)$ \; \; if and only
if  \; \; $u \in \cL_{\e/r}$ \, and \, $N||z||^2+r< 1$.} 
$$
Indeed,
$R_K(\gamma(z)+\sqrt r u)=r R_K(u)$ for all $K\subset \ig 1,N\id$ and
$||\gamma(z)+\sqrt r u ||^2 = \sum_{1}^N||z+\sqrt r u^i||^2=N||z||^2+r$ 
because $\sum_1^N u^i=0$ and $\sum_1^N ||u^i||^2=1$. Thus
$$
\mu(\cK_\e \cap B(0,1))= 
\int_{\rr^2\times\rr_+} \indiq_{\{ N ||z||^2+r<1 \}} r^\nu \beta(\cL_{\e/r}) \dd z \dd r.
$$
 All this implies that for all $\e\in (0,1]$, for almost all $r\in (0,1)$, $\beta(\cL_{\e/r})<\infty$.
Since $\e\to \cL_\e$ is monotone, we conclude that
$\beta(\cL_{\e})<\infty$ for all $\e\in (0,1]$. Since finally $\cup_{\e \in (0,1]} \cL_\e = \sS \cap E_{k_0}$
and since $\cL_\e$ is compact in $\sS \cap E_{k_0}$ for each $\e \in (0,1]$, we conclude as desired
that $\beta$ is Radon on $\sS \cap E_{k_0}$.

\vip

We next prove (ii). It holds that $\sS \subset E_{N}$, because for $u\in \sS$, 
we have $R_{\ig 1,N\id}(u)=1$. Hence if 
$k_0\geq N$, then $\sS \subset E_{N}\subset E_{k_0}$ , whence 
$\sS=\sS\cap E_{k_0}$ and thus $\beta$ is Radon on $\sS$ by point (i). 
Since finally $\sS$ is compact,
we conclude that $\beta(\sS)<\infty$.
\end{proof}

\section{Markov processes and Dirichlet spaces}\label{aapp}
\setcounter{thm}{0}
\setcounter{equation}{0}

In a first subsection, we recall some classical definitions and results 
about Hunt processes, diffusions and Dirichlet spaces found in 
Fukushima-Oshima-Takeda \cite{f}. In a second subsection, we mention 
a few results about martingales, times-changes, concatenation, killing and Girsanov transformation
of Hunt processes found in \cite{f} and elsewhere.

\subsection{Main definitions and properties}\label{ap1}
Let $E$ be a locally compact separable metrizable space endowed with a Radon measure $\alpha$
such that Supp $\alpha=E$. We set $\Et=E\cup\{\triangle\}$, where $\triangle$ is a cemetery point.
See \cite[Section A2]{f} for the definition of a Hunt process $\yY=(\Omega,\cM,(Y_t)_{t\geq 0},
(\PP_y)_{y\in \Et})$: it is a strong Markov process in its canonical filtration, $\PP_y(Y_0=y)=1$ for all $y\in \Et$,
$\triangle$ is an absorbing state, i.e. $Y_t=\triangle$ for all $t\geq 0$ under $\PP_\triangle$, 
and a few more technical properties are satisfied.
The life-time of $\yY$ is defined by $\zeta=\inf\{t\geq 0 : Y_t=\triangle\}$. 

\vip
Let us denote by $P_t(y,\dd z)$ its transition kernel.
Our Hunt process is said to be $\alpha$-symmetric if
$\int_{E} \varphi P_t \psi \dd \alpha= \int_{E} \psi P_t \varphi \dd \alpha$ 
for all measurable $\varphi,\psi:E\to\rr_+$ and all $t\geq 0$, see \cite[page 30]{f}. 
The Dirichlet space $(\cE,\cF)$ of our Hunt process on $L^2(E,\alpha)$ is then defined, see \cite[page 23]{f}, by
\begin{gather*}
\cF=\Big\{\varphi\in L^2(E,\alpha) : \lim_{t\to 0} \frac1t \int_{E} \varphi(P_t \varphi -\varphi)
\dd \alpha \hbox{  exists}\Big\},\\
\cE(\varphi,\psi)=-\lim_{t\to 0} \frac 1t \int_{E} \varphi(P_t \psi -\psi)\dd \alpha  \qquad\hbox{for all 
$\varphi,\psi\in \cF$}.
\end{gather*}
The generator $(\cA,\cD_A)$ of $\yY$ is defined as follows: 
\begin{gather*}
\cD_\cA=
\Big\{\varphi\in L^2(E,\alpha) : \lim_{t\to 0} \frac1t (P_t \varphi -\varphi) 
\hbox{ exists in $L^2(E,\alpha)$}\Big\},
\end{gather*}
and for $\varphi \in \cD_\cA$, we denote by $\cA \varphi\in L^2(E,\alpha)$ this limit.  
By \cite[Pages 20-21]{f}, it holds that
\begin{equation}\label{caracdomaine0}
\cD_{\cA} = \Big\{ \varphi \in \cF : \exists\; h \in L^2 (E,\alpha ) \mbox{ such that } \forall \;
\psi\in \cF, \; \hbox{ we have } \; \cE (\varphi,\psi) = -\int_E h \psi \dd \alpha \Big\}
\end{equation}
and in such a case $\cA \varphi = h$. 

\vip

The one-point compactification $\Et=E\cup\{\triangle\}$ of $E$ is endowed with the 
topology consisting of all the open sets of $E$ and of all the sets of the form 
$K^c \cup \{\triangle\}$ with $K$ compact in $E$, see page \cite[page 69]{f}.
Observe that for a $\Et$-valued sequence $(x_n)_{n\geq 0}$, we have
$\lim_n x_n=x$ if and only if  

\vip

\noindent $\bullet$ either $x \in E$, $x_n \in E$ for all $n$ large enough,
and $\lim_n x_n=x \in E$ in the usual sense;
\vip
\noindent $\bullet$ 
or $x=\triangle$ and for all compact subset $K$ of $E$, there is $n_K \in \nn$ such that
for all $n\geq n_K$,  $x_n\notin K$.

\vip

We say that our Hunt process is continuous if $t\to Y_t$ is continuous from $\rr_+$ into $\Et$,
where $\Et$ is endowed with the one-point compactification topology. \blue A continuous Hunt process
is called a {\it diffusion}. \bla

\vip

A Dirichlet space $(\cE,\cF)$  on $L^2(E,\alpha)$ is said to be regular
if it has a core, see \cite[page 6]{f}, i.e. a subset $\cC\subset C_c(E)\cap \cF$ which is
dense in $\cF$ for the norm $||\varphi||=[\int_E \varphi^2 \dd \alpha + \cE(\varphi,\varphi)]^{1/2}$ and 
dense in $C_c(E)$ for the uniform norm.

\vip

Observe two regular Dirichlet spaces $(\cE,\cF)$ and $(\cE',\cF')$ such that
$\cE(\varphi,\varphi)=\cE'(\varphi,\varphi)$ for all   $\varphi$   
in a common core $\cC$ are necessarily equal,
i.e. $\cF=\cF'$ and $\cE=\cE'$. This follows from the fact that by definition,
see \cite[page 5]{f}, a Dirichlet space is closed.

\vip

We say that a Borel set $A$ of $E$ is $(P_t)_{t\ge 0}$-invariant if for all
$\varphi\in L^2(E,\alpha) $, all $t>0$ we have $P_t(\indiq _A \varphi ) = \indiq _A P_t \varphi$ $\alpha$-a.e, see
\cite[page 53]{f}. According to \cite[page 55]{f}, we say that $(\cE, \cF)$ 
is irreducible if for all $(P_t)_{t\ge 0}$-invariant set $A$, we have either $\alpha(A) = 0$ or 
$\alpha(E\setminus A)=0$.

\vip

We say that $(\cE, \cF)$ is recurrent if for all nonnegative $\varphi\in L^1(E,\alpha)$, for 
$\alpha$-a.e. $y\in E$, we have $\E_y[\int _0 ^\infty \varphi(Y_s)\dd s] \in \{ 0, \infty \}$, see \cite[page 55]{f}.

\vip

We finally say that $(\cE, \cF)$ is transient if for all nonnegative $\varphi\in L^1(E,\alpha)$, for $\alpha$-a.e. 
$y\in E$, we have $\E_y[\int _0 ^\infty \varphi(Y_s)\dd s] <\infty $, with the convention
that $\varphi(\triangle)=0$, see \cite[page 55]{f}.

\vip

By \cite[Lemma 1.6.4 page 55]{f}, if $(\cE, \cF)$ is irreducible, then it is either recurrent or transient.

\vip

A Borel set $\cN \subset E$ is properly exceptional if $\alpha(\cN)=0$ and
$\PP_y (\exists t\ge 0 : Y_t \in   \cN  )=0$ for all $y \in E\setminus \cN$, see \cite[page 153]{f}.
A property is said to hold true \blue quasi-everywhere \bla if it holds true outside
a properly exceptional set. 

\begin{rk}\label{qua}
Two Hunt processes with the same Dirichlet
space share the same \blue quasi-everywhere \bla notion, up to the restriction that 
the capacity of every compact set is finite, which is always the case in the
present work.
\end{rk}

\begin{proof}
We fix a Hunt process $\yY$ and explain why its \blue quasi-everywhere \bla 
notion depends only on its Dirichlet space.
A set $\cN \subset E$ is exceptional, see \cite[page 152]{f}, if there exists a Borel set $\tilde \cN$ such that
$\cN\subset \tilde \cN$ and $\PP _y (\exists t\ge 0 : Y_t \in \tilde \cN )=0$ for $\alpha$-a.e. $y\in E$.  
A properly exceptional set is clearly exceptional and \cite[Theorem 4.1.1 page 155]{f} tells us that any 
exceptional set is included in a properly exceptional set. 
Thus, a property is true \blue quasi-everywhere \bla if and only if it holds true outside an exceptional set. 
Next, \cite[Theorem 4.2.1-(ii) page 161]{f} tells us that
a set $\cN$ is exceptional
if and only if its capacity is $0$, where the capacity of $\cN \subset E$ is entirely defined 
from the Dirichlet space. And for \cite[Theorem 4.2.1-(ii) page 161]{f} to apply,
one needs that the capacity of all compact sets is finite.
\end{proof}

\subsection{Toolbox}\label{tool}

We start with martingales.

\begin{lemma}\label{marting}
Let $E$ be a locally compact separable metrizable space endowed with a Radon measure $\alpha$
such that Supp $\alpha=E$, and $(\Omega,\cM,(Z_t)_{t\geq 0}, (\PP_z)_{z\in \Et})$ a $\alpha$-symmetric 
$\Et$-valued \blue diffusion \bla with regular Dirichlet space
$\left( \mathcal{E}, \mathcal{F} \right)$ on $L^2(E,\alpha)$
and generator $(\cA,\cD_\cA)$. Assume that $\varphi:E\mapsto\rr$ belongs to 
$\cD_\cA$ and that both $\varphi$ and $\cA \varphi$ are bounded. 
Define 
$$
M^\varphi_t=\varphi(Z_t)-\varphi(Z_0)-\int_0^t \cA\varphi (Z_s)\dd s,
$$ 
with the convention that
$\varphi(\triangle)=\cA\varphi(\triangle)=0$. \blue Quasi-everywhere, \bla $(M^\varphi_t)_{t\geq 0}$ is a
$\PP_z$-martingale in the canonical filtration of $(Z_t)_{t\geq 0}$.
\end{lemma}

This can be found in \cite[page 332]{f}. There the assumption on $\varphi$ is that there is $f$ bounded and
measurable such that $\varphi = R_1f$, i.e. $\varphi=(I-\cA)^{-1}f$, which simply means that $\varphi-\cA \varphi$
is bounded. Also, the conclusion is that $(M^\varphi_t)_{t\geq 0}$ is a MAF, which indeed implies
that $(M^\varphi_t)_{t\geq 0}$ is a martingale, see \cite[page 243]{f}.

\vip

Next, we deal with time-changes.

\begin{lemma}\label{chgmttemps}
Let  $E$ be a $C^\infty$-manifold, $\alpha$ a Radon measure on $E$ such that 
${\rm Supp}(\alpha) =E$,
and $(\Omega,\cM,(Z_t)_{t\geq 0}, (\PP_z)_{z\in \Et})$ a $\alpha$-symmetric 
$\Et$-valued \blue diffusion \bla with regular Dirichlet space
$\left( \mathcal{E}, \mathcal{F} \right)$ on $L^2(E,\alpha)$ with core $C^\infty_c(E)$.
We also fix $g:E \to (0,\infty)$ continuous and take the convention that $g(\triangle)=0$.
We consider the time-change $A_{t} = \int _{0} ^{t} g(Z_{s})\dd s$ and its generalized inverse 
$\rho_t=\inf\{s>0 : A_s>t\}$.
We introduce $Y_t = Z_{\rho _{t}}\indiq_{\{\rho_t<\infty\}}+\triangle\indiq_{\{\rho_t=\infty\}}$.
Then $(\Omega,\cM,(Y_t)_{t\geq 0}, (\PP_y)_{y\in\Et})$ is a $g\alpha$-symmetric 
$\Et$-valued \blue diffusion \bla with regular Dirichlet space 
$\left(\mathcal{E}, \mathcal{F'} \right)$ on $L^2(E,g\alpha)$ with core $C^\infty_c(E)$, i.e.
$\cF'$ is the closure of $C^\infty_c(E)$ with respect to the norm
$[\int_E \varphi^2 g \dd \alpha + \cE(\varphi,\varphi)]^{1/2}$.
\end{lemma}

\begin{rk}
If we apply the preceding result to the simple case where $E$ is an open subset of $\rr ^d$ 
and where $\cE (\varphi,\varphi)= \int _{\rr ^d } \| \nabla \varphi \| ^2 \dd \alpha $ 
for all $\varphi\in C^\infty_c(E)$, then when $\cE$ is seen as the Dirichlet form of a $g\alpha$-symmetric 
process, it may be
better understood as $\cE (\varphi,\varphi) = \int _{\rr ^d} \| g^{-1/2}\nabla \varphi \|^2 g \dd \alpha$.
\end{rk}

This lemma is nothing but a particular case of \cite[Theorem 6.2.1 page 316]{f}, see also the few pages before.
We only have to check that the Revuz measure in our case is $g \alpha$, i.e., see \cite[(5.1.13) page 229]{f}, 
that for 
all bounded nonnegative measurable functions $\varphi,\psi$ on $E$, for all $t>0$, 
$$
\int_E \E _x \Big[ \int_0 ^t \varphi(Z_s)g(Z_s)\dd s\Big] \psi(x) \alpha (\dd x) 
= \int _0 ^t \int_E (P^Z_s \psi)\varphi g \dd \alpha,
$$ 
where $P_t^Z$ is the semi-group of $Z$.
The left hand side equals $\int_0^t \int_E P_s^Z(\varphi g) \psi \dd \alpha$, so that the claim is obvious
since $Z$ is $\alpha$-symmetric.

\vip

The following concatenation result can be found in Li-Ying \cite[Proposition 3.2]{ly}.

\begin{lemma}\label{concatenation}
Let $E_V,E_W$ be two $C^\infty$-manifolds, $\alpha_V,\alpha_W$ be some Radon measures on $E_V$ and 
$E_W$ such that ${\rm Supp}(\alpha_V) =E_V$ and ${\rm Supp}(\alpha_W) =E_W$.
Let $(\Omega^V,\cM^V,(V_t)_{t\geq 0}, (\PP_v^V)_{v\in E_V \cup \{\triangle\}})$ be a 
$\alpha_V$-symmetric $(E_V\cup\{\triangle\})$-valued \blue diffusion \bla with regular Dirichlet space
$\left( \mathcal{E}^V, \mathcal{F}^V \right)$ on $L^2(E_V,\alpha_V)$ with core $C^\infty_c(E_V)$.
Consider $(\Omega^W,\cM^W,(W_t)_{t\geq 0}, (\PP_w^W)_{w\in E_W \cup \{\triangle\}})$, a
$\alpha_W$-symmetric $(E_W\cup\{\triangle\})$-valued \blue diffusion \bla with regular Dirichlet space
$\left( \mathcal{E}^W, \mathcal{F}^W \right)$ on $L^2(E_W,\alpha_W)$ with core $C^\infty_c(E_W)$.
Introduce the measure $\alpha = \alpha_V \otimes \alpha_W$ on $E = E_V \times E_W$.
We take the convention that $(v,\triangle)=(\triangle,w)=(\triangle,\triangle)=\triangle$
for all $v\in E_V$, all $w\in E_W$. Moreover, we set 
$\cM^{(V,W)}=\sigma(\{(V_t,W_t) : t\geq 0\})$ and we define 
$\PP^{(V,W)}_{(v,w)} = \PP_v^V\otimes\PP_w^W$
if $(v,w)\in E_V\times E_W$ and $\PP^{(V, W)}_\triangle=\PP_\triangle^V\otimes\PP_\triangle^W$.
The process 
$$
\Big(\Omega^V\times\Omega^W,\cM^{(V,W)},(V_t,W_t)_{t\geq 0}, 
(\PP_{(v,w)}^{(V,W)})_{(v,w)\in (E_V\times E_W)\cup \{\triangle\}}\Big)
$$ 
is a $\Et$-valued
$\alpha$-symmetric \blue diffusion \bla, with regular Dirichlet space $(\cE,\cF)$ on
 $L^2(E,\alpha)$ with core
$C^\infty_c(E)$ and, for $\varphi \in C^\infty_c(E)$,
$$
\cE(\varphi,\varphi)= \int_{E_V} \cE^W(\varphi(v,\cdot),\varphi(v,\cdot)) \alpha_V(\dd v) 
+ \int_{E_W} \cE^V(\varphi(\cdot,w),\varphi(\cdot,w)) 
\alpha_W(\dd w).
$$
\end{lemma}

Observe that $\cM^{(V,W)}$ may be strictly smaller than $\cM^V\otimes\cM^W$ due to the identification
of all the cemetery points. Also, it actually holds true that $\PP_\triangle^V\otimes\PP_w^W
= \PP_v^V\otimes\PP_\triangle^W = \PP_\triangle^V\otimes\PP_\triangle^W$ on $\cM^{(V,W)}$
so that the choice $\PP^{(V, W)}_\triangle=\PP_\triangle^V\otimes\PP_\triangle^W$ is arbitrary but legitimate.

\vip

The following killing result is a summary, adapted to our context, of Theorems 4.4.2 page 173 and 
4.4.3-(i) page 174  in \cite[Section 4.4]{f}.

\begin{lemma}\label{tuage}
Let  $E$ be a $C^\infty$-manifold, let $\alpha$ be a Radon measure on $E$ such that 
${\rm Supp}(\alpha) =E$,
and let $(\Omega,\cM,(Z_t)_{t\geq 0}, (\PP_z)_{z\in \Et})$ be a $\alpha$-symmetric 
$\Et$-valued \blue diffusion \bla with regular Dirichlet space
$\left( \mathcal{E}, \mathcal{F} \right)$ on $L^2(E,\alpha)$ with core $C^\infty_c(E)$.
Let $O$ be an open subset of $E$ and consider $\tau_O=\inf\{t\geq 0 : X_t \notin O\}$, with the convention
that $\inf \emptyset = \infty$. Then, setting 
$$
Z^O_t = Z_t \indiq _{\{ t <\tau_O \}} + \triangle \indiq _{\{ t\geq \tau_O \}},
$$
$(\Omega,\cM,(Z^O_t)_{t\geq 0}, (\PP_z)_{z\in O \cup \{\triangle\}})$ is a 
$\alpha|_O$-symmetric $O\cup \{\triangle \}$-valued \blue diffusion \bla with regular Dirichlet 
space $(\cE _O , \cF _O)$ on $L^2 (O, \alpha|_O )$ with core $C_c^\infty (O)$ and for $\varphi \in \cF _O$,
$$ 
 \cE _O (\varphi,\varphi) = \cE (\varphi,\varphi).
$$
\end{lemma}

Note that since $O$ is an open subset of the manifold $E$ and since the Hunt process is continuous, 
the regularity condition (4.4.6) of \cite[Theorem 4.4.2 page 173]{f} is obviously satisfied.

\vip

We finally give an adaptation of the 
Girsanov theorem in the context of Dirichlet spaces, which is 
a particular case of Chen-Zhang \cite[Theorem 3.4]{cz}.

\begin{lemma}\label{girsanovfuku}
Let $E$ be an open subset of $\rr^d$, with $d\geq 1$,
$\alpha$ be a Radon measure on $E$ such that ${\rm Supp} (\alpha ) = E$ and 
$(\Omega,\cM,(Z_t)_{t\geq 0}, (\PP_z)_{z\in \Et})$ be a $\alpha$-symmetric 
$\Et$-valued \blue diffusion \bla with regular Dirichlet space
$\left( \mathcal{E}, \mathcal{F} \right)$ on $L^2(E,\alpha)$ with core $C^\infty_c(E)$ such that for all 
$\varphi\in C^\infty_c(E)$, 
$$ \cE (\varphi,\varphi) = \int _E \|\nabla \varphi \|^2 \dd \alpha .$$
Let $(\cA,\cD_{\cA})$ stand for its generator.
Let $u \in \cF$ be bounded, such that \blue for $\varrho = e^u$, we have $\varrho - 1 \in \cD_\cA$
with $\cA[\varrho-1]$ is bounded. \bla
Set
$$
L^\varrho_t=\frac{\varrho (Z_t)}{\varrho (Z_0)} 
\exp \Big( -\int _0 ^t \frac{\cA \blue [\varrho-1] \bla (Z_s)}{\varrho (Z_s)} \dd s \Big),
$$
with the conventions that $\varrho (\triangle )=1$ and $\cA \blue[\varrho-1]\bla (\triangle)=0$.
\vip
Assume that $\varrho$ is continuous on $\Et$.
Then \blue quasi-everywhere, \bla $(L_t^\varrho)_{t\geq 0}$ is a bounded $(\cM_t)_{t\geq 0}$-martingale under $\PP_z$,
where we have set $\cM_t=\sigma(\{Z_s : s\in [0,t]\})$, and
there exists a probability measure $\tilde{\PP}_z$ on $(\Omega,\cM)$,
such that for all $t>0$, 
$\tilde{\PP}_z = L_t^\varrho \cdot \PP _z$ on $\cM_t$.
\vip

Moreover $(\Omega,\!\cM,\!(Z_t)_{t\geq 0},\! (\tilde{\PP}_z)_{z\in \Et})$ is a
$\varrho^2 \!\alpha$-symmetric 
$\Et$-valued \blue diffusion \bla with regular Dirichlet space $(\tilde{\cE},\cF)$ 
on $L^2(E, \varrho ^2 \alpha)$ such that for all $\varphi\in\cF$,
$$ 
\tilde{\cE } (\varphi,\varphi) = \frac{1}{2} \int _E \|\nabla \varphi \|^2 \varrho ^2 \dd \alpha.
$$
\end{lemma}

Actually, they speak of {\it right processes} in \cite{cz}, but this is not an issue since
we only consider continuous Hunt processes. Also, they assume that $L^\varrho$ is bounded from 
above and from below by some deterministic constants, on each compact time interval,
but this is obvious under our assumptions on $u$ and $\cA\varrho$.
Finally, their expression of $L^\varrho$ is different, see \cite[pages 485-486]{cz}: 
first, they define $M^\varrho_t$ as the martingale part of $\varrho(X_t)$.
By Lemma~\ref{marting} \blue (applied to $\varrho-1$), \bla we see
that $$M^{\varrho}_t=\varrho(Z_t)-\varrho(Z_0)-\intot \cA \blue[\varrho-1]\bla (Z_s)\dd s.$$ 
\blue
Then they put $M_t=\intot [\varrho(Z_s)]^{-1}\dd M^\varrho_s$ and define $L^\varrho$ as
$$
L^\varrho_t=\exp\Big(M_t-\frac12\langle M \rangle_t\Big).
$$
\bla
But by It\^o's formula, $\log \varrho(Z_t)= \log \varrho(Z_0)+\intot [\varrho(Z_s)]^{-1}\dd M^\varrho_s
+ \intot [\varrho(Z_s)]^{-1}\cA \blue[\varrho-1]\bla (Z_s)\dd s 
- \frac12 \intot [\varrho(Z_s)]^{-2}\dd \langle M^\varrho\rangle_s$,
whence $\log \varrho(Z_t)= \log \varrho(Z_0)+M_t+ \intot [\varrho(Z_s)]^{-1}\cA \blue[\varrho-1]\bla (Z_s)\dd s 
- \frac12 \langle M\rangle_t$,
so that $L^\varrho_t=\exp(M_t-\frac12\langle M \rangle_t)
=[\varrho(Z_0)]^{-1}\varrho(Z_t)\exp(-\intot \varrho(Z_s)^{-1}\cA \blue[\varrho-1]\bla (Z_s)\dd s )$ as desired.


\begin{thebibliography}{99}

\bibitem{at}{{\sc G. Arumugam, J. Tyagi}, Keller-Segel Chemotaxis Models: A Review,
{\it Acta Appl. Math.} {\bf 171} (2021), 6.}
 
\bibitem{bdp}{{\sc A. Blanchet, J. Dolbeault, B. Perthame}, Two-dimensional Keller-Segel model: 
optimal critical mass and qualitative properties of the solutions, {\it Electron. J. Differential Equations} 
{\bf 44} (2006), 32 pp.}

\bibitem{bjw}{{\sc D. Bresch, P.E. Jabin, Z. Wang} On mean-field limits and quantitative estimates 
with a large class of singular kernels: application to the Patlak-Keller-Segel model,
{\it C. R. Math. Acad. Sci. Paris} {\bf 357} (2019), 708--720.}

\bibitem{cp}{{\sc P. Cattiaux, L. P\'ed\`eches}, The 2-D stochastic Keller-Segel particle model:
existence and uniqueness, {\it ALEA, Lat. Am. J. Probab. Math. Stat.} {\bf  13} (2016), 447--463.}

\bibitem{cz}{{\sc Z. Chen, T.S. Zhang}, Girsanov and Feynman-Kac type transformations for 
symmetric Markov processes. {\it Ann. Inst. H. Poincar\'e Probab. Statist.} {\bf 38} (2002), 475--505.}

\bibitem{ds}{{\sc J. Dolbeault, C. Schmeiser}, The two-dimensional Keller-Segel model after blow-up, 
{\it Discrete Contin. Dyn. Syst.} {\bf 25} (2009), 109--121.}

\bibitem{fa}{{\sc I. Fatkullin}, A study of blow-ups in the Keller-Segel model of chemotaxis,
{\it Nonlinearity} {\bf 26} (2013), 81--94.}

\bibitem{fj}{{\sc N. Fournier, B. Jourdain}, Stochastic particle approximation of the Keller-Segel equation 
and two-dimensional generalization of Bessel processes, {\it Ann. Appl. Probab.} {\bf 27} (2017), 2807--2861.} 

\bibitem{fhm}{{\sc N. Fournier, M. Hauray, S. Mischler}, Propagation of chaos for the 2D viscous vortex model,
{\it J. Eur. Math. Soc.} {\bf 16} (2014), 1423--1466.}

\bibitem{mf}{{\sc M. Fukushima}, From one dimensional diffusions to symmetric Markov processes, 
{\it Stochastic Process. Appl.} {\bf 120} (2010), 590--604.}

\bibitem{f}{{\sc M. Fukushima, Y. Oshima, M. Takeda}, 
Dirichlet forms and symmetric Markov processes. Second revised and extended edition. 
Walter de Gruyter, 2011.}

\bibitem{hs1}{{\sc J. Ha\v{s}kovec, C. Schmeiser}, Stochastic particle approximation for measure valued 
solutions of the 2D Keller-Segel system, {\it J. Stat. Phys.} {\bf 135}, 133--151.}

\bibitem{hs2}{{\sc J. Ha\v{s}kovec, C. Schmeiser}, Convergence of a stochastic particle approximation 
for measure solutions of the 2D Keller-Segel system, {\it Comm. Partial Differential Equations} 
{\bf 36(6)} (2011), 940--960.}

\bibitem{jtt}{{\sc J.F. Jabir, D. Talay, M. Tomasevic}, Mean-field limit of a particle approximation of 
the one-dimensional parabolic-parabolic Keller-Segel model without smoothing, {\it Electron. Commun. Probab.} 
{\bf 23} (2018), Paper No. 84.}

\bibitem{kshreve}{{\sc I. Karatzas, S.E. Shreve}, Brownian motion and stochastic calculus. 
Second edition. Graduate Texts in Mathematics, 113. Springer-Verlag, New York, 1991.}

\bibitem{ks}{{\sc E.F. Keller, L.A. Segel}, Initiation of slime mold aggregation viewed as
an instability, {\it J. Theor. Biol.} {\bf 26} (1970), 399--415.}

\bibitem{ly}{{\sc L. Li,, J. Ying}, Regular subspaces of Dirichlet forms. 
Festschrift Masatoshi Fukushima, 397–420, Interdiscip. Math. Sci., 17, World Sci. Publ., Hackensack, NJ, 2015.}

\bibitem{ort}{{\sc C. Olivera, A. Richard, M. Tomasevic}, Particle approximation of the 2-d parabolic-elliptic 
Keller-Segel system in the subcritical regime, arXiv:2011.00537.}

\bibitem{o}{{\sc H. Osada}, Propagation of chaos for the two-dimensional {N}avier-{S}tokes equation,
{\it Proc. Japan Acad. Ser. A Math. Sci.} {\bf 62} (1986), 8--11.}

\bibitem{p}{{\sc C.S. Patlak}, Random walk with persistence and external bias, {\it Bull.
Math. Biophys.} {\bf 15} (1953), 311--338.}

\bibitem{ry}{{\sc D. Revuz, M. Yor}, Continuous martingales and Brownian motion.
Third Edition. Springer, 2005.}

\bibitem{st}{{\sc A. Stevens}, The derivation of chemotaxis equations as limit dynamics of moderately 
interacting stochastic many-particle systems, {\it SIAM J. Appl. Math.} {\bf 61} (2000), 183--212.}

\bibitem{s}{{\sc T. Suzuki} Exclusion of boundary blowup for 2D chemotaxis system provided
with Dirichlet boundary condition for the Poisson part, {\it J. Math. Pures Appl.} {\bf 100} (2013), 347--367.}

\bibitem{v1}{{\sc J.J.L. Velazquez}, Point dynamics in a singular limit of the Keller-Segel model. I. Motion of 
the concentration regions, {\it SIAM J. Appl. Math.}, {\bf 64} (2004), 1198--1223.}

\bibitem{v2}{{\sc J.J.L. Velazquez}, Point dynamics in a singular limit of the Keller-Segel model. II. Formation
of the concentration regions, {\it SIAM J. Appl. Math.}, {\bf 64} (2004), 1224--1248.}

\end{thebibliography}
\end{document}